\documentclass[11pt]{amsart} \textwidth=14.5cm \oddsidemargin=1cm
\usepackage{amsmath}
\usepackage{amsxtra}
\usepackage{amscd}
\usepackage{amsthm}
\usepackage{amsfonts}
\usepackage{amssymb}
\usepackage{eucal}

\input prepictex
\input pictex
\input postpictex

\newtheorem{thm}{Theorem}[section]
\newtheorem{lem}[thm]{Lemma}
\newtheorem{cor}[thm]{Corollary}
\newtheorem{prop}[thm]{Proposition}

\theoremstyle{definition}

\theoremstyle{remark}
\newtheorem{rem}[thm]{Remark}

\numberwithin{equation}{section}

\begin{document}

\newcommand{\thmref}[1]{Theorem~\ref{#1}}
\newcommand{\secref}[1]{Section~\ref{#1}}
\newcommand{\lemref}[1]{Lemma~\ref{#1}}
\newcommand{\propref}[1]{Proposition~\ref{#1}}
\newcommand{\corref}[1]{Corollary~\ref{#1}}
\newcommand{\remref}[1]{Remark~\ref{#1}}
\newcommand{\eqnref}[1]{(\ref{#1})}
\newcommand{\exref}[1]{Example~\ref{#1}}

\newcommand{\nc}{\newcommand}
 \nc{\Z}{{\mathbb Z}}
 \nc{\C}{{\mathbb C}}
 \nc{\N}{{\mathbb N}}
 \nc{\F}{{\mf F}}
 \nc{\Q}{\ol{Q}}
 \nc{\la}{\lambda}
 \nc{\ep}{\epsilon}
 \nc{\h}{\mathfrak h}
 \nc{\n}{\mf n}
 \nc{\A}{{\mf a}}
 \nc{\G}{{\mathfrak g}}
 \nc{\SG}{\overline{\mathfrak g}}
 \nc{\D}{\mc D}
 \nc{\Li}{{\mc L}}
 \nc{\La}{\Lambda}
 \nc{\is}{{\mathbf i}}
 \nc{\V}{\mf V}
 \nc{\bi}{\bibitem}
 \nc{\NS}{\mf N}
 \nc{\dt}{\mathord{\hbox{${\frac{d}{d t}}$}}}
 \nc{\E}{\mc E}
 \nc{\ba}{\tilde{\pa}}
 \nc{\half}{\frac{1}{2}}
 \nc{\mc}{\mathcal}
 \nc{\mf}{\mathfrak}
 \nc{\hf}{\frac{1}{2}}
 \nc{\hgl}{\widehat{\mathfrak{gl}}}
 \nc{\gl}{{\mathfrak{gl}}}
 \nc{\hz}{\hf+\Z}
 \nc{\vac}{|0 \rangle}
 \nc{\dinfty}{{\infty\vert\infty}}
 \nc{\SLa}{\overline{\Lambda}}
 \nc{\SF}{\overline{\mathfrak F}}
 \nc{\SP}{\overline{\mathcal P}}
 \nc{\U}{\mathfrak u}
 \nc{\SU}{\overline{\mathfrak u}}
 \nc{\ov}{\overline}
 \nc{\sL}{\ov{\mf{l}}}
 \nc{\sP}{\ov{\mf{p}}}
 \nc{\osp}{\mf{osp}}
 \nc{\spo}{\mf{spo}}
 \nc{\hosp}{\widehat{\mf{osp}}}
 \nc{\hspo}{\widehat{\mf{spo}}}
 \nc{\tL}{\mathcal{L}}
 \nc{\Sp}{\mf{Sp}}
 \nc{\glmn}{\mf{gl}(m|n)}
 \nc{\Pmn}{\mf{P}(m|n)}

 \advance\headheight by 2pt

\title[Kostant homology formula]
{Kostant homology formulas for oscillator modules of Lie
superalgebras}

\author[Cheng]{Shun-Jen Cheng$^\dagger$}
\thanks{$^\dagger$Partially supported by an NSC-grant and an Academia Sinica Investigator grant}
\address{Institute of Mathematics, Academia Sinica, Taipei,
Taiwan 11529} \email{chengsj@math.sinica.edu.tw}

\author[Kwon]{Jae-Hoon Kwon$^{\dagger\dagger}$}
\thanks{$^{\dagger\dagger}$Partially supported by KRF Grant 2008-314-C00004.}
\address{Department of Mathematics, University of Seoul, Seoul 130-743, Republic of Korea}
\email{jhkwon@uos.ac.kr}

\author[Wang]{Weiqiang Wang$^{\dagger\dagger\dagger}$}
\thanks{$^{\dagger\dagger\dagger}$Partially supported by NSF and NSA grants.}
\address{Department of Mathematics, University of Virginia, Charlottesville, VA 22904, USA}
\email{ww9c@virginia.edu}

\begin{abstract} \vspace{.3cm}
We provide a systematic approach to obtain formulas for characters
and Kostant $\U$-homology groups of the oscillator modules of the
finite dimensional general linear and ortho-symplectic
superalgebras, via Howe dualities for infinite dimensional Lie
algebras. Specializing these Lie superalgebras to Lie algebras, we
recover, in a new way, formulas for Kostant homology groups of
unitarizable highest weight representations of Hermitian symmetric
pairs. In addition, two new reductive dual pairs related to the
above-mentioned $\U$-homology computation are worked out.

\end{abstract}



 \maketitle

  \setcounter{tocdepth}{1}
\tableofcontents

\section{Introduction}

\subsection{}
Inspired by the Borel-Weil-Bott theorem, in a classical work
\cite{Ko} Kostant computed the $\U$-(co)homology groups for finite
dimensional simple modules of a semisimple Lie algebra, recovering
the celebrated Weyl character formula in a purely algebraic way.
Since then Kostant's calculation has been generalized in nontrivial
ways to various setups, e.g.~to integrable modules of Kac-Moody
algebras \cite{GL}, to unitarizable highest weight modules of
Hermitian symmetric pairs \cite{En}, to finite dimensional modules
of general linear superalgebras \cite{Se} (and in a different way to
polynomial representations \cite{CZ1}), and more recently to modules
of infinite dimensional Lie superalgebras \cite{CK}.

On the other hand, Howe's theory of reductive dual pairs \cite{H,
H2} has played important roles in the representation theory of real
and $p$-adic Lie groups, and there have been generalizations in
different directions. Dual pairs $(\G, G)$ have been formulated
systematically by the third author \cite{Wa} between {\em infinite
rank} affine Kac-Moody algebras $\G$ \cite{DJKM, K} and finite
dimensional reductive Lie groups $G$. More recently, dual pairs
between {\em finite dimensional} Lie superalgebras $\SG$ and Lie
groups $G$ have been studied in depth with application to
irreducible characters of $\SG$ in a number of papers \cite{CW1,
CZ2, CLZ}. A priori, no direct link between different dual pairs
$(\SG,G)$ and $(\G,G)$ was expected to exist, except that the
$\G$-modules $\{L(\G, \La(\la))\}$ and $\SG$-modules $\{L(\SG,
\widehat{\La}_f(\la))\}$ appearing in these Howe dualities can both
be parameterized by (part of) the same index set that parameterizes
the finite dimensional simple $G$-modules $\{V_G^\la\}$. Below is a
list of the reductive dual pairs used in this paper which share the
same $G$.

\vspace{.4cm}
\begin{center}
\begin{tabular}{|c|c|c|}
\hline Dual Pairs I $(\SG, G)$ & Dual Pairs II $(\G, G)$ & Dual
Pairs III
$(\G^*, G)$\\
\hline $(\gl(p+m|q+n), {\rm GL}(d))$ & $(\mf{a}_\infty, {\rm
GL}(d))$ &
$(\mf{a}_\infty, {\rm GL}(d))$\\
\hline $(\spo(2m|2n+1), {\rm Pin}(d))$ & $(\mf{b}_\infty, {\rm
Pin}(d))$ &
$(\mf{b}^{\mf 0}_\infty, {\rm Pin}(d))$ \\
\hline $(\osp(2m|2n), {\rm Sp}(d))$ & $(\mf{c}_\infty, {\rm Sp}(d))$
&
$(\mf{d}_\infty, {\rm Sp}(d))$ \\
\hline $(\spo(2m|2n), {\rm O}(d))$ & $(\mf{d}_\infty, {\rm O}(d))$ &
$(\mf{c}_\infty, {\rm O}(d))$  \\
\hline
\end{tabular}

\vspace{.1cm} Table 1: Reductive dual pairs
\end{center}

\subsection{}
The main goal of this paper is to develop a conceptual approach to
computing Kostant $\U$-homology groups  with coefficients in the
so-called oscillator $\SG$-modules, i.e. those appearing in the Howe
duality $(\SG,G)$ (see Column~I in Table~1). Remarkably, our results
show that the $\mf u$-homology groups of these oscillator
$\SG$-modules are dictated by those of the corresponding integrable
$\G$-modules, where $\G$ is the infinite dimensional classical
``counterpart'' (see Column~II) of $\SG$. The Howe dualities
$\left(\spo(2m|2n+1),{\rm Pin}(d)\right)$ and $\left(\mf{b}^{\mf
0}_\infty,{\rm Pin}(d)\right)$ appear to be new and are worked out
in Appendix \ref{appendix:A}.

The results in this paper
provided the first  supporting evidence for a super duality between categories $\ov{\mathcal O}$ of
$\SG$-modules and certain categories $\mc{O}_f$ of $\G_f$-modules
\cite{CW} (generalizing the super duality \cite{CWZ} for type $A$),
where $\G_f$ denotes certain finite dimensional reductive Lie
algebras corresponding to $\SG$. As classical Kazhdan-Lusztig
polynomials allow an interpretation in terms of $\U$-homology
\cite{V}, the main results here may be reformulated as an equality
of certain Kazhdan-Lusztig polynomials for categories $\ov{\mathcal
O}$ and $\mathcal O_f$.

In spite of the infinite dimensionality, the $\G$-modules above are
integrable and hence there is a standard approach to compute their
Kostant homology groups (cf. \cite{GL, J, L}). On the other hand,
the structures of the infinite dimensional oscillator modules of the
finite dimensional Lie superalgebras $\SG$ are not so well
understood.
In the approach of this paper, we use essentially only the Howe
dualities $(\SG, G)$ and $(\G, G)$  for the same $G$ (see Columns~I
and II in the above Table) together with some simple combinatorial
manipulations with the characters of the integrable $\G$-modules to
derive first a character formula of the oscillator $\SG$-modules in
a suitable form. Then from a comparison of the Casimir eigenvalues
of the corresponding $\G$-module and $\SG$-module, we obtain
formulas of the corresponding $\U$-homology groups with coefficients
in $L(\SG, \widehat{\La}_f(\la))$, which are expressed in terms of
the infinite Weyl group of $\G$, from the corresponding formulas for
${\rm H}_*({\U}, L(\G, \La(\la)))$.

\subsection{}
In the paper \cite{CK}, the Howe duality $(\G,G)$ was used together
with another Howe duality $(\G^s,G)$ (due to \cite{CW2} for types
$A,B$ and \cite{LZ} for types $C,D$) to derive a $\U$-homology
formula for modules over infinite dimensional Lie superalgebras
$\G^s$ from a corresponding $\U$-homology formula for modules of
$\G$. However, the connection between integrable modules over
infinite rank affine Kac-Moody algebras and oscillator
representations of finite dimensional Lie (super)algebras were not
suspected back then.

On the other hand, character formulas for $\SG$-modules above were
obtained in different ways and in different forms in earlier works
\cite{CZ2, CLZ}. In \cite{CZ2} a key role was played by a difficult
theorem of Enright on $\U$-homology for unitarizable highest weight
modules of real reductive Lie algebras established by intricate
arguments involving equivalence of categories and nontrivial
combinatorics on Weyl groups \cite{En}. Our present approach
bypasses Enright's theorem. Indeed, in the special case when the Lie
superalgebras specialize to Lie algebras (i.e. $q=n=0$ in Column I
and Rows 1, 3, and 4, of Table 1), we recover Kostant homology
formulas for unitarizable highest weight modules of three Hermitian
symmetric pairs of classical types (which can be shown by some
combinatorial argument to be equivalent to Enright's formula).

\subsection{}

The paper is organized as follows. In \secref{IDLA}, we review and
set up notations for various Howe dualities involving the infinite
dimensional Lie algebras $\G$ and finite dimensional Lie
superalgebras $\SG$. In \secref{sec:char}, we compute the character
formulas of the oscillator $\SG$-modules from Howe dualities. In
\secref{casimir:op}, the Casimir eigenvalues of modules of $\SG$ and
$\G$ are computed and compared, and they are used in
\secref{sec:homology} to obtain formulas for Kostant homology groups
for the oscillator $\SG$-modules.

Generalizing the type $A$ case in \cite{CK}, we compute in
\secref{sec:negative} the character formulas and Kostant homology
formulas for non-integrable $\G^*$-modules at negative integral
levels appearing in the Howe dualities $(\G^*, G)$ listed in
Column~III of Table 1. The dualities $(\G^*, G)$ with $G={\rm
GL}(d)$,  ${\rm Sp}(d)$ and ${\rm O}(d)$ were treated in \cite{KR2}
and \cite{Wa}. The new case involving ${\rm Pin}(d)$ is worked out
in Appendix \ref{appendix:A}, where one can also find an additional
Howe duality $\left(\spo(2m|2n+1),{\rm Pin}(d)\right)$.


\subsection{Notations} \label{sec:notation}
Denote by $\mathcal P^+$ the set of partitions. For $\la\in\mc{P}^+$
we denote by $\ell(\la)$ the length of $\la$, and by $|\la|$ the
size of $\la$. Given $d\in\N$, a non-increasing sequence
$\la=(\la_1,\la_2,\ldots,\la_d)$ of $d$ integers will be called a
{\em generalized partition of depth} $d$. For a generalized
partition $\la$ of depth $d$, we define
\begin{equation*}
\la^+:=(\langle\la_1\rangle,\ldots,\langle\la_d\rangle),\quad
\la^-:=(\langle-\la_1\rangle,\ldots,\langle-\la_d\rangle),
\end{equation*}
where here and further we set $\langle a\rangle={\rm max}\{a,0\}$
for $a\in\Z$. Then $\la^+$ is a partition and $\la^-$ is a
non-decreasing sequence of non-negative integers.

For a sequence of non-negative integers $\mu=(\mu_1,\mu_2,\ldots)$
we let $\mu'=(\mu'_1,\mu'_2,\ldots)$ be the partition with
$\mu'_j:=|\{\,i\,\vert\,\mu_i\ge j\,\}|$. Let $\Z$, $\N$, and $\Z_+$
stand for the set of all, positive, and non-negative integers,
respectively.

All vector spaces, algebras, etc.~are over the complex field $\C$.

\subsection{Acknowledgment}   The authors thank Academia Sinica in Taipei,
University of Virginia, and MSRI for hospitality and support, where
part of this work was carried out.

\section{Howe dualities for Lie algebras $\G$ and Lie superalgebras $\SG$}\label{IDLA}

In this section we review various Howe dualities involving infinite
dimensional Lie algebras $\G$ and finite dimensional Lie
superalgebras $\SG$.

\subsection{Infinite rank affine Lie algebras}\label{class:subalgebras}

The infinite dimensional Lie algebra $\widehat{\gl}_\infty$ and its
subalgebras $\mf{b}_\infty$, $\mathfrak{c}_{\infty}$,
$\mathfrak{d}_{\infty}$ of types $B, C, D$ \cite{DJKM, K} are well
known. Our notations regarding these Lie algebras in this paper will
be the same as those in \cite[Section 2.2]{CK}, and we refer to {\em
loc.~cit.} for explicit case-by-case description of the following
standard terminologies based on matrix elements $E_{ij}$ for
$\widehat{\gl}_\infty$, $\widetilde{E}_n$ for other types, and
weights $\epsilon_i$ etc.
\begin{itemize}
\item[$\cdot$] $\G$: $\widehat{\gl}_\infty
\equiv\mathfrak{a}_{\infty}$, $\mf{b}_\infty$,
$\mathfrak{c}_{\infty}$ or $\mathfrak{d}_{\infty}$, with triangular
decomposition $\G =\G_+ \oplus \h \oplus \G_-$,

\item[$\cdot$] $\mathfrak{h}$: a Cartan subalgebra of $\mathfrak{g}$,

\item[$\cdot$] $I$: an index set for simple roots for $\G$,

\item[$\cdot$] $\mathfrak l$: the Levi subalgebra of $\G$ with
simple roots indexed by $S=I\backslash \{0\}$,

\item[$\cdot$] $\U_\pm$: the nilradicals with $\G =\U_+ \oplus
\mathfrak l \oplus \U_-$, $\U_-:=\U$,

\item[$\cdot$] $\Pi=\{\,\alpha_i\,|\,i\in I\,\}$: a set of simple
roots for $\G$,

\item[$\cdot$] $\Pi^{\vee}=\{\,\alpha_i^{\vee}\,|\,i\in I\,\}$: a
set of simple coroots for $\G$,

\item[$\cdot$] $\Delta^{+}$  (resp. $\Delta^{-}$): a set of
positive (resp. negative) roots for $\G$,

\item[$\cdot$] $\Delta^{+}_S$ (resp. $\Delta^{-}_S$): a set of
positive (resp. negative) roots for $\mathfrak l$,

\item[$\cdot$] $\Delta^\pm (S)$: the subset of roots for
$\U_\pm$,

\item[$\cdot$] $\La_i^{\mathfrak x}$: the fundamental weights for
$\mf{x}_\infty$ with $i \in I$ and $\mf{x} \in \{\mf{a}, \mf{b},
\mf{c}, \mf{d}\}$,

\item[$\cdot$] $\rho_c$: ``half sum" of the positive roots in
$\Delta^{+}$,

\item[$\cdot$] $W$ (resp. $W_0$): Weyl group of $\G$ (resp.
$\mathfrak l$),

\item[$\cdot$] $W^0_k \equiv W^0_k(\mf{x})$: the set of the
minimal length representatives of the right coset space
$W_0\backslash W$ of length $k$ for $\mf{x}_\infty$ with $\mf{x} \in
\{\mf{a}, \mf{b}, \mf{c}, \mf{d}\}$,

\item[$\cdot$] $L(\G,\La)$: the irreducible highest weight $\G$-module of
highest weight $\La\in\h^*$.



\end{itemize}
\subsection{Reductive dual pairs on Fock spaces}\label{classical:dualpairs}

\subsubsection{Fermionic Fock spaces}
We fix a positive integer $\ell\geq 1$ and consider $\ell$ pairs
of free fermions $\psi^{\pm,i}(z)$ with $i=1,\ldots,\ell$. That
is, we have
\begin{align*}
\psi^{+,i}(z)&=\sum_{r\in\hf+\Z}\psi^{+,i}_rz^{-r-\hf},\quad\quad\
\psi^{-,i}(z)=\sum_{r\in\hf+\Z}\psi^{-,i}_rz^{-r-\hf},
\end{align*}
with non-trivial anti-commutation relations
$[\psi^{+,i}_r,\psi^{-,j}_s]=\delta_{ij}\delta_{r+s,0}$. Let
$\F^\ell$ denote the corresponding Fock space generated by the
vacuum vector $|0\rangle$, which is annihilated by
$\psi^{+,i}_r,\psi^{-,i}_s$ for $r,s>0$.

We introduce a neutral fermionic field
$\phi(z)=\sum_{r\in\hf+\Z}\phi_rz^{-r-\hf}$ with non-trivial
anti-commutation relations $[\phi_r,\phi_s]=\delta_{r+s,0}$. Denote
by $\F^{\hf}$ the Fock space of $\phi(z)$ generated by a vacuum
vector that is annihilated by $\phi_r$ for $r> 0$. We denote by
$\F^{\ell+\hf}$ the tensor product of $\F^{\ell}$ and $\F^{\hf}$.

\subsubsection{The $(\widehat{\gl}_\infty,{\rm GL}(d))$-duality}

We denote by $e_{ij}$ ($1\leq i,j\leq d$) the elementary $d\times d$
matrix with $1$ in the $i$th row and $j$th column and $0$ elsewhere.
Then $H=\sum_{i}\C e_{ii}$ is a Cartan subalgebra of $\mf{gl}(d)$,
while $\sum_{i\le j}\C e_{ij}$ is a Borel subalgebra of $\mf{gl}(d)$
containing $H$. An irreducible rational representation of
$\mf{gl}(d)$ (or ${\rm GL}(d)$) is determined by its highest weight
$\la\in H^*$ with $\langle \la, e_{ii}\rangle=\la_i\in\Z$ $(1\le
i\le d)$ and $\la_1\geq\ldots \geq \la_{d}$. Identifying $\la$ with
$(\la_1,\ldots,\la_{d})$ as usual, we denote by $V^\la_{{\rm
GL}(d)}$ the irreducible ${\rm GL}(d)$-module of highest weight
$\la$. These irreducible modules are parameterized by the set of
generalized partitions of depth $d$:
\begin{equation*}
{\mc P}({\rm GL}(d)):=\{\,\la=(\la_1,
\ldots,\la_d)\,|\,\lambda_i\in\Z, \ \la_1\geq\ldots \geq
\la_{d}\,\}.
\end{equation*}

\begin{prop} \label{duality} \cite{Fr} {\rm (cf.~\cite[Theorem 3.1]{Wa})}
There exists a commuting action of $\widehat{\gl}_\infty$ and ${\rm
GL}(d)$ on $\F^d$. Furthermore, under this joint action, we have
\begin{equation} \label{eq:dual}
\F^d\cong\bigoplus_{\la\in {\mc P}({\rm GL}(d))}
L(\hgl_\infty,\La^{\mf a}(\la))\otimes V_{{\rm GL}(d)}^\la ,
\end{equation}
where $\Lambda^{\mf a}(\la):=d\La^{\mf a}_0+\sum_{j\ge
1}(\la^+)'_j\epsilon_j-\sum_{j\ge 0}(\la^-)'_{j+1}\epsilon_{-j} =
\sum_{i=1}^d \Lambda^{\mf a}_{\la_i}$.
\end{prop}

We assume below that $x_n$ ($n\in\Z$) and $z_i$ ($i\in\N$) are
formal indeterminates. Computing the trace of the operator
$\prod_{n\in\Z}x_n^{E_{nn}}\prod_{i=1}^d z_i^{e_{ii}}$ on both sides
of \eqnref{eq:dual}, we obtain the following identity:
\begin{equation}\label{combid-classical1}
\prod_{i=1}^d\prod_{n\in\N}(1+x_nz_i)(1+x_{1-n}^{-1}z^{-1}_{i})=\sum_{\la\in
{\mc P}({\rm GL}(d))}{\rm ch}L(\hgl_\infty,\Lambda^{\mf a}(\la))\;
{\rm ch}V^\la_{{\rm GL}(d)}.
\end{equation}

\subsubsection{The $(\mathfrak{c}_{\infty},{\rm Sp}(d))$-duality}

Let $d$ be even and ${\rm Sp}(d)$ denote the symplectic group, which
may be viewed as the subgroup of ${\rm GL}(d)$ preserving the
non-degenerate skew-symmetric bilinear form on $\C^{d}$ given by
$$\left(
    \begin{array}{cc}
      0 & J_{\frac{d}{2}} \\
      -J_{\frac{d}{2}} & 0 \\
    \end{array}
  \right).
$$
Here $J_k$ is the following $k\times k$ matrix:
\begin{equation}\label{side:diagonal}
J_{k}=\begin{pmatrix}
0&0&\cdots&0&1\\
0&0&\cdots&1&0\\
\vdots&\vdots&\vdots&\vdots&\vdots\\
0&1&\cdots&0&0\\
1&0&\cdots&0&0\\
\end{pmatrix}.
\end{equation}
Let $\mf{sp}(d)$ be the Lie algebra of ${\rm Sp}(d)$. We take as a
Borel subalgebra of $\mf{sp}(d)$ the subalgebra of upper triangular
matrices, and as a Cartan subalgebra $H$ the subalgebra spanned by
${e}_{i}=e_{ii}-e_{d+1-i,d+1-i}$ ($1\leq i\leq \frac{d}{2}$). A
finite dimensional irreducible representation of $\mf{sp}(d)$ is
determined by its highest weight $\la\in H^*$ with $\langle \la,
{e}_{i}\rangle=\la_i\in\Z_+$ $(1\le i\le \frac{d}{2})$ and
$\la_1\geq\ldots \geq \la_{\frac{d}{2}}$. Furthermore each such
representation lifts to an irreducible representation of ${\rm
Sp}(d)$, which is denoted by $V_{{\rm Sp}(d)}^{\lambda}$. Put
\begin{equation*}
{\mc P}({\rm Sp}(d)):=\{\,\la=(\la_1,
\ldots,\la_{\frac{d}{2}})\,|\,\lambda_i\in\Z_+, \ \la_1\geq\ldots
\geq \la_{\frac{d}{2}}\,\},
\end{equation*}
which is the set of partitions of length no more than $\frac{d}{2}$.

\begin{prop} \label{duality-c} \cite[Theorem 3.4]{Wa}
There exists a commuting action of $\mathfrak{c}_{\infty}$ and ${\rm
Sp}(d)$ on $\F^{\frac{d}{2}}$. Furthermore, under this joint action,
we have
\begin{equation} \label{eq:dual-c}
\F^{\frac{d}{2}}\cong\bigoplus_{\la\in {\mc P}({\rm
Sp}(d))}L(\mathfrak{c}_{\infty},\La^{\mf c}(\la))\otimes V_{{\rm
Sp}(d)}^\la,
\end{equation}
where $\Lambda^{\mf c}(\la) :=\frac{d}{2}\La^{\mf c}_0+\sum_{k\geq
1}\la'_k\epsilon_k = \sum_{k=1}^{\frac{d}{2}} \Lambda^{\mf
c}_{\la_k}$.
\end{prop}

Computing the trace of $\prod_{n\in
\N}x_n^{\widetilde{E}_n}\prod_{i=1}^\frac{d}{2} z_i^{{e}_{i}}$ on
both sides of \eqnref{eq:dual-c}, we have
\begin{equation}\label{combid-classical-c}
\prod_{i=1}^\frac{d}{2}\prod_{n\in\N}(1+x_nz_i)(1+x_{n}z^{-1}_{i})=\sum_{\la\in
{\mc P}({\rm Sp}(d))} {\rm ch}L(\mathfrak{c}_{\infty},\Lambda^{\mf
c}(\la))\; {\rm ch}V^\la_{{\rm Sp}(d)}.
\end{equation}

\subsubsection{The $(\mathfrak{d}_{\infty},{\rm O}(d))$-duality}
Write $d\in \N$ as $d=2\ell$ or $d=2\ell+1$ with $\ell \in\N$. Let
${\rm O}(d)$ denote the orthogonal group which is the subgroup of
${\rm GL}(d)$ preserving  the non-degenerate symmetric bilinear form
on $\C^{d}$ determined by $J_d$ of \eqnref{side:diagonal}. Let
$\mf{so}(d)$ be the Lie algebra of ${\rm O}(d)$. We take as a Cartan
subalgebra $H$ of $\mf{so}(d)$ the subalgebra spanned by
${e}_{i}:=e_{ii}-e_{d+1-i,d+1-i}$ ($1\leq i\leq \ell$), while we
take as the Borel subalgebra the subalgebra of upper triangular
matrices.

For $\la\in H^*$ let $\la_i=\langle\la,{e}_{i}\rangle$, for $1\le
i\le \ell$. Then a finite dimensional irreducible module of
$\mf{so}(2\ell)$ is determined by its highest weight $\la$
satisfying the condition $\la_1\geq \ldots\geq \la_{\ell-1}\geq
|\la_{\ell}|$ with either $\la_i\in\Z$ or else $\la_i\in\hf+\Z$, for
all $1\le i\le \ell$. Furthermore it lifts to a module of ${\rm
SO}(2\ell)$ if and only if $\la_i\in\Z$ for $1\le i\le \ell$. Also a
finite dimensional irreducible module of $\mf{so}(2\ell+1)$ is
determined by its highest weight $\la$ satisfying the conditions
$\la_1\geq \ldots\geq \la_{\ell}$ with either  $\la_i\in\Z_+$ or
else $\la_i\in\hf+\Z_+$, for all $1\le i\le \ell$. Furthermore it
lifts to a module of ${\rm SO}(2\ell+1)$ if and only if
$\la_i\in\Z_+$, for $1\le i\le \ell$. Put
\begin{equation*}
{\mc P}({\rm O}(d)):=\{\,\la=(\la_1,
\ldots,\la_d)\,|\,\lambda_i\in\Z_+, \ \la_1\geq\ldots \geq \la_d,\
\la'_1+\la'_2\leq d\,\}.
\end{equation*}
For $\la\in {\mc P}({\rm O}(d))$, let $\tilde{\la}$  be the
partition obtained from $\la$ by replacing its first column with
$d-\la'_1$.

Suppose that $d=2\ell$ and let $\lambda=(\la_1,\ldots,\la_{\ell},
0,\ldots, 0)\in {\mc P}({\rm O}(2\ell))$. For $\la_{\ell}>0$, the
irreducible ${\rm O}(2\ell)$-module $V_{{\rm O}(2\ell)}^{\lambda}$,
viewed as an $\mf{so}(2\ell)$-module, is isomorphic to the direct
sum of irreducible modules of highest weights
$(\la_1,\ldots,\la_{\ell})$ and $(\la_1,\ldots,-\la_{\ell})$. If
$\la_{\ell}=0$, the ${\rm O}(2\ell)$-module $V_{{\rm
O}(2\ell)}^{\lambda}$, viewed as an $\mf{so}(2\ell)$-module, is
isomorphic to the irreducible module of highest weight
$(\la_1,\ldots,\la_{\ell-1},0)$, on which the element
$\tau=\sum_{i\not=\ell,\ell+1}{e_{ii}}+e_{\ell,\ell+1}+e_{\ell+1,\ell}\in
{\rm O}(2\ell)\setminus {\rm SO}(2\ell)$ transforms trivially on
highest weight vectors. Set $V_{{\rm
O}(2\ell)}^{\tilde{\lambda}}:=V_{{\rm O}(2\ell)}^{\lambda}\otimes
{\rm det}$, where ${\rm det}$ is the one-dimensional non-trivial
module of ${\rm O}(2\ell)$.

Suppose that $d=2\ell+1$ and let $\lambda=(\la_1,\ldots,\la_{\ell},
0,\ldots, 0)\in {\mc P}({\rm O}(2\ell+1))$. Let $V_{{\rm
O}(2\ell+1)}^{\lambda}$ be the irreducible ${\rm O}(2\ell+1)$-module
isomorphic  to the irreducible module of highest weight
$(\la_1,\ldots,\la_{\ell})$ as an ${\rm SO}(2\ell+1)$-module, on
which $-I_{d}$ acts trivially. Here $I_d$ is the $d\times d$
identity matrix. Also, we let $V_{{\rm
O}(2\ell+1)}^{\tilde{\lambda}}:=V_{{\rm
O}(2\ell+1)}^{\lambda}\otimes {\rm det}$ (cf. e.g.~\cite{BT, H}).

\begin{prop} \label{duality-d} \cite[Theorems 3.2 and 4.1]{Wa}
There exists a commuting action of $\mathfrak{d}_{\infty}$ and ${\rm
O}(d)$ on $\F^{\frac{d}{2}}$. Furthermore, under this joint action,
we have
\begin{equation} \label{eq:dual-d}
\F^{\frac{d}{2}}\cong\bigoplus_{\la\in {\mc P}({\rm
O}(d))}L(\mathfrak{d}_{\infty},\La^{\mf d}(\la))\otimes V_{{\rm
O}(d)}^\la,
\end{equation}
where $\Lambda^{\mf d}(\la):=d\Lambda^{\mf d}_{0}+\sum_{k\geq 1}
\la'_k\epsilon_k$.
\end{prop}

Suppose that $d=2\ell$. Computing the trace of $\prod_{n\in
\N}x_n^{\widetilde{E}_n}\prod_{i=1}^\ell z_i^{{e}_{i}}$ on
\eqnref{eq:dual-d} gives us
\begin{equation}\label{combid-classical-d1}
\prod_{i=1}^\ell\prod_{n\in\N}(1+x_nz_i)(1+x_{n}z^{-1}_{i})=\sum_{\la\in
{\mc P}({\rm O}(2\ell))} {\rm
ch}L(\mathfrak{d}_{\infty},\Lambda^{\mf d}(\la)){\rm ch}V^\la_{{\rm
O}(2\ell)}.
\end{equation}

Suppose that $d=2\ell+1$. Let $\epsilon$ be the eigenvalue of $-I_d$
on ${\rm O}(2\ell+1)$-modules satisfying $\epsilon^2=1$. From the
computation of the trace of $\prod_{n\in
\N}x_n^{{E}_{n}}\prod_{i=1}^\ell z_i^{{e}_{i}}(-I_d)$ on both sides
of \eqnref{eq:dual-d}, we obtain
\begin{equation}\label{combid-classical-d2}
\prod_{i=1}^\ell\prod_{n\in\N}(1+\epsilon x_nz_i)(1+ \epsilon
x_{n}z^{-1}_{i})(1+\epsilon x_n)=\sum_{\la\in {\mc P}({\rm
O}(2\ell+1))} {\rm ch}L(\mathfrak{d}_{\infty},\Lambda^{\mf
d}(\la)){\rm ch}V^\la_{{\rm O}(2\ell+1)}.
\end{equation}
Note that ${\rm ch}V^\la_{{\rm O}(2\ell)}$ is a Laurent polynomial
in $z_1,\ldots,z_{\ell}$ and ${\rm ch}V^\la_{{\rm O}(2\ell)}={\rm
ch}V^{\tilde{\la}}_{{\rm O}(2\ell)}$, while ${\rm ch}V^\la_{{\rm
O}(2\ell+1)}$ is the Laurent polynomial in
$z_1,\ldots,z_{\ell},\epsilon$ and ${\rm ch}V^\la_{{\rm
O}(2\ell+1)}=\epsilon~{\rm ch}V^{\tilde{\la}}_{{\rm O}(2\ell+1)}$.

\subsubsection{The $(\mf{b}_\infty,{\rm Pin}(d))$-duality} Let
$d$ be even. The Lie group ${\rm Pin}(d)$ is a double cover of ${\rm
O}(d)$, with ${\rm Spin}(d)$ as the inverse image of ${\rm SO}(d)$
under the covering map (see e.g.~\cite{BT}). An irreducible
representation of ${\rm Spin}(d)$ that does not factor through ${\rm
SO}(d)$ is an irreducible representation of $\mf{so}(d)$ of highest
weight of the form
\begin{equation}\label{hwforSpin}
(\la_1+\hf,\ldots,\la_{\frac{d}{2}-1}+\hf,\la_{\frac{d}{2}}+\hf), \
\ \text{or} \ \
(\la_1+\hf,\ldots,\la_{\frac{d}{2}-1}+\hf,-\la_{\frac{d}{2}}-\hf),
\end{equation}
where $\la_1\geq \ldots\geq \la_{\frac{d}{2}}$ with $\la_i\in\Z_+$
for $1\le i\le \frac{d}{2}$. We put
\begin{equation*}
{\mc P}({\rm Pin}(d)):=\{\,\la=(\la_1,
\cdots,\la_\frac{d}{2})\,|\,\lambda_i\in\Z_+, \ \la_1\geq\ldots \geq
\la_{\frac{d}{2}}\,\}.
\end{equation*}
For $\lambda\in {\mc P}({\rm Pin}(d))$, let us denote by
$V^{\lambda}_{{\rm Pin}(d)}$ the irreducible representation of ${\rm
Pin}(d)$ induced from the irreducible representation of ${\rm
Spin}(d)$ whose highest weight is given by either of the two weights
in \eqref{hwforSpin}. When restricted to ${\rm Spin}(d)$,
$V^{\lambda}_{{\rm Pin}(d)}$ decomposes into a direct sum of two
irreducible representations of highest weights given by those in
\eqref{hwforSpin}. The modules in \{$V^\la_{{\rm
Pin}(d)}\vert\la\in\mc{P}({\rm Pin}(d))\}$ are precisely those
finite dimensional irreducible representations of ${\rm Pin}(d)$
that do not factor through ${\rm O}(d)$.

\begin{prop} \label{duality-b} \cite[Theorem 3.3]{Wa} There exists a commuting action of $\mathfrak{b}_{\infty}$ and ${\rm Pin}(d)$ on $\F^{\frac{d}{2}}$. Furthermore, under this joint
action, we have
\begin{equation} \label{eq:dual-b}
\F^{\frac{d}{2}}\cong\bigoplus_{\la\in {\mc P}({\rm
Pin}(d))}L(\mathfrak{b}_{\infty},\La^{\mf b}(\la))\otimes V_{{\rm
Pin}(d)}^\la,
\end{equation}
where $\Lambda^{\mf b}(\la):=d\Lambda^{\mf b}_{0}+\sum_{k\geq 1}
\la'_k\epsilon_k$.
\end{prop}
Computing the trace of the operator $\prod_{n\in
\Z}x_n^{\widetilde{E}_n}\prod_{i=1}^\frac{d}{2} z_i^{{e}_{i}}$ on
both sides of \eqnref{eq:dual-b}, gives
\begin{equation}\label{combid-classical-b}
\prod_{i=1}^\frac{d}{2}\prod_{n\in\N}(z_i^{\hf}+z_i^{-\hf})(1+x_nz_i)(1+x_{n}z^{-1}_{i})=\sum_{\la\in
{\mc P}({\rm Pin}(d))} {\rm ch}L(\mathfrak{b}_{\infty},\Lambda^{\mf
b}(\la)){\rm ch}V^\la_{{\rm Pin}(d)}.
\end{equation}

\subsection{Formulas for $\mathfrak u_-$-homology groups of ${\mf
g}$-modules}\label{classical:homology}

Recall that the Weyl group $W$ can be written as $W =W_0W^0$ with
$W^0 =\bigsqcup_{k \ge 0} W^0_k$. For $\mu\in\h^*$ and $w\in W$ we
set $w\circ\mu:=w(\mu+\rho_c)-\rho_c$.

Since we have $\langle w\circ\La^{\mf
x}(\la),\alpha_j^\vee\rangle\in\Z_+$, for $w\in W^0$ and $j\in S$,
it follows that we may find partitions
$\la^{\pm}_w=((\la^{\pm}_w)_1,(\la^{\pm}_w)_2,\ldots)$ and
$\la_w=((\la_w)_1,(\la_w)_2,\ldots)$ such that $w\circ\La^{\mf
x}(\la)$ can be written as
\begin{equation*}
w\circ\La^{\mf x}(\la)=
\begin{cases}
d\Lambda^{\mf x}_0+\sum_{j>0}(\la^+_w)_j\epsilon_j-\sum_{j\ge
0}(\la^-_w)_{j+1}\epsilon_{-j}, & \text{if ${\mf
x}={\mf a}$}, \\
d\Lambda^{\mf x}_0+\sum_{j>0}(\la_w)_j\epsilon_j, & \text{if ${\mf
x}=\mf{b,d}$}, \\
\frac{d}{2}\Lambda^{\mf x}_0+\sum_{j>0}(\la_w)_j\epsilon_j, &
\text{if ${\mf x}=\mf{c}$}.
\end{cases}
\end{equation*}
The following is obtained by applying the Kostant homology formula
for integrable (=standard) modules over Kac-Moody algebra $\G$ (cf.
e.g. \cite[Theorem 3.13]{J}) and the Euler-Poincar\'e principle to
the complex for the corresponding Lie algebra homology (see
\cite[Section~2.4, (2.18)]{CK}).

\begin{prop} \label{char:schur}
We have  the following character formula:
\begin{equation*}
{\rm ch}L(\G,\La^{\mathfrak x}(\la))=\frac{1}{D^{\mathfrak
x}}\sum_{k=0}^\infty (-1)^k {\rm ch H}_k(\mathfrak
u_-;L(\G,\La^{\mathfrak x}(\la))),
\end{equation*}
where ${\rm ch H}_k({\mathfrak u}_-;L(\G,\La^{\mathfrak x}(\la)))$ is given by
\begin{equation}\label{homology:schur}
\begin{cases}
\sum_{w\in W^0_k(\mathfrak x)}s_{\la^+_w}(x_1,x_2,\ldots)
s_{\la^-_w}(x^{-1}_0,x^{-1}_{-1},\ldots), & \text{if
${\mathfrak x}={\mf a}$}, \\
\sum_{w\in W^0_k(\mathfrak x) }s_{\la_w}(x_1,x_2,\ldots), & \text{if
${\mathfrak x}=\mf{b, c,d}$},
\end{cases}
\end{equation}
and
\begin{equation*}
D^{\mathfrak x}:=
\begin{cases}
\prod_{i,j}(1-x_{-i+1}^{-1}x_j), & \text{if
${\mathfrak x}={\mf a}$}, \\
\prod_{i}(1-x_i)\prod_{i< j}(1-x_ix_j), & \text{if $\mf{x}=\mf{b}$}, \\
\prod_{i\le j}(1-x_ix_j), & \text{if } {\mathfrak x}=\mf{c},
\\ \prod_{i< j}(1-x_ix_j), & \text{if } {\mathfrak x}=\mf{d},
\end{cases}\quad i,j\in\N.
\end{equation*}
\end{prop}
Here, for a partition $\gamma$, $s_\gamma$ denotes the corresponding
Schur function.

\subsection{The general linear and ortho-symplectic superalgebras}
Let $p,q,m,n\in\Z_+$. We briefly recall the general linear
superalgebra $\gl(p+m|q+n)$ and ortho-symplectic superalgebras
$\spo(2m|2n+1)$, $\osp(2m|2n)$, $\spo(2m|2n)$ (see e.g.~\cite{K1}),
which will be called of type ${\mf a}$, $\mf{b}$, ${\mf c}$, and
${\mf d}$, respectively, for reasons of Howe dualities appearing
later on. The following notations associated with these Lie
superalgebras will be assumed throughout the paper.
\begin{itemize}
\item[$\cdot$] $\SG$: $\gl(p+m|q+n)$, $\spo(2m|2n+1)$, $\osp(2m|2n)$,
$\spo(2m|2n)$, or their central extensions in
\secref{sec:centralExt},

\item[$\cdot$] $\ov{\h}$: a Cartan subalgebra of $\SG$,

\item[$\cdot$] $\ov{I}$: an index set for simple roots for $\SG$,

\item[$\cdot$] $\ov{\Pi}=\{\,\beta_i\,|\,i\in \ov{I}\,\}$: a set
of simple roots for $\SG$,

\item[$\cdot$] $\ov{\Pi}^{\vee}=\{\,\beta_i^{\vee}\,|\,i\in \ov{I}\,\}$: a set of
simple coroots for $\SG$,

\item[$\cdot$] $\ov{\Delta}^{+}$: a set of positive
roots for $\SG$,

\item[$\cdot$] $L(\SG,\la)$: the irreducible highest weight $\SG$-module of
highest weight $\la\in\ov{\h}^*$.
\end{itemize}

\subsubsection{}

Denote by $\C^{p+m|q+n}$ the complex superspace of dimension
$(p+m|q+n)$ with basis $\{\,v_{-p},\ldots,v_{-1},
w_{-q},\ldots,w_{-1},v_1,\ldots,v_m,w_1,\ldots,w_n\,\}$, where ${\rm
deg}v_i:=\ov{0}$ and ${\rm deg}w_j:=\ov{1}$. With respect to this
basis the general linear superalgebra $\gl(p+m|q+n)$ may be regarded
as the Lie superalgebra of complex matrices $(a_{ij})$, with $i,j\in
I_{p+m|q+n}$, where
$I_{p+m|q+n}:=\{-p,\ldots,-1,1,\ldots,m\}\cup\{-\ov{q},\ldots,-\ov{1},\ov{1},\ldots,\ov{n}\}$.
For notational convenience we declare a linear ordering of
$I_{p+m|q+n}$ by setting
\begin{equation*}
-p<\cdots<-1<-\ov{q}<\cdots<-\ov{1}<1<\cdots<m<\ov{1}<\cdots<\ov{n}.
\end{equation*}
For $i\in I_{p+m|q+n}\setminus\{\bar{n}\}$, $i+1$ means the minimum
of those $j\in I_{p+m|q+n}$ with $j>i$. In a similar way one defines
$i-1$, for $i\in I_{p+m|q+n}\setminus\{-p\}$.

Denote by $E_{ij}$, $i,j\in I_{p+m|q+n}$, the elementary matrix with
$1$ at the $i$th row and the $j$th column, and zero elsewhere. Then
$\ov{\h}$ is spanned by $E_{kk}$, $k\in I_{p+m|q+n}$. Let
$\varepsilon_i$ and $\delta_j$, with
$i\in\{-p,\ldots,-1,1,\ldots,m\}$ and
$j\in\{-q,\ldots,-1,1,\ldots,n\}$, form the basis of $\ov{\h}^*$
dual to $E_{ii}$ and $E_{\ov{j},\ov{j}}$, respectively. We choose
our Borel subalgebra to be the subalgebra spanned by $E_{ij}$ with
$i\le j$. With respect to this Borel subalgebra, we have
\begin{align*}
\ov{\Pi}^{\vee}=\{\,& \beta_i^{\vee}=E_{i-1,i-1}-E_{ii} \ \ (\,i\in\{-p+1,\ldots,-\ov{1}\,\}\setminus\{-\ov{q}\}),  \\
& \beta_i^{\vee}=E_{ii}-E_{i+1,i+1} \ \ (\,i\in\{1,\ldots,\ov{n-1}\,\}\setminus\{m\}),  \\
&  \beta_{-\ov{q}}^{\vee}=E_{-1,-1}+E_{-\ov{q},-\ov{q}},\
\beta_{0}^{\vee}=E_{-\ov{1},-\ov{1}}+E_{11},\
\beta_{m}^{\vee}=E_{mm}+E_{\ov{1},\ov{1}} \, \},\allowdisplaybreaks \\
\ov{\Pi}=\{\, & \beta_{i}=\varepsilon_{i-1}-\varepsilon_{i} \
(\,i=-p+1,\ldots,-1), \ \beta_{i}=\varepsilon_i-\varepsilon_{i+1} \
(\,i=1,\ldots,m-1), \\
& \beta_{\ov{j}}=\delta_{j-1}-\delta_{j} \ (\,j=-{q}+1,\ldots,-{1}),\ \beta_{\ov{j}}=\delta_j-\delta_{j+1} \ (\,j=1,\ldots,n-1),\allowdisplaybreaks\\
& \beta_{-\ov{q}}=\varepsilon_{-1}-\delta_{-q},\ \beta_{0}=\delta_{-{1}}-\varepsilon_1,\ \beta_{m}=\varepsilon_m-\delta_1\, \}, \\
\ov{\Delta}^+ =\{\,& \varepsilon_i-\varepsilon_j,\
\varepsilon_i-\delta_j,\ \delta_i-\varepsilon_j,\ \delta_i-\delta_j
\ (i<j) \,\}.
\end{align*}
The associated Dynkin diagram is as follows: ($\bigotimes$ denotes
an odd isotropic root)

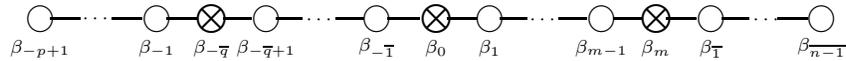
\begin{center}
\setlength{\unitlength}{0.13in}
\begin{picture}(31,4)
\put(.3,2){\makebox(0,0)[c]{$\bigcirc$}}
\put(5,2){\makebox(0,0)[c]{$\bigcirc$}}
\put(7.2,2){\makebox(0,0)[c]{$\bigotimes$}}
\put(9.4,2){\makebox(0,0)[c]{$\bigcirc$}}
\put(13.85,2){\makebox(0,0)[c]{$\bigcirc$}}
\put(16.25,2){\makebox(0,0)[c]{$\bigotimes$}}
\put(18.4,2){\makebox(0,0)[c]{$\bigcirc$}}

\put(22.9,2){\makebox(0,0)[c]{$\bigcirc$}}
\put(25.1,2){\makebox(0,0)[c]{$\bigotimes$}}
\put(27.3,2){\makebox(0,0)[c]{$\bigcirc$}}
\put(31.75,2){\makebox(0,0)[c]{$\bigcirc$}}

\put(.7,2){\line(1,0){1.2}} \put(3.2,2){\line(1,0){1.2}}
\put(5.4,2){\line(1,0){1.2}}\put(7.7,2){\line(1,0){1.1}}\put(9.82,2){\line(1,0){1}}
\put(12.2,2){\line(1,0){1.15}} \put(14.28,2){\line(1,0){1.3}}
\put(16.7,2){\line(1,0){1.2}} \put(18.81,2){\line(1,0){1.1}}
\put(25.5,2){\line(1,0){1.2}} \put(21.2,2){\line(1,0){1.1}}
\put(23.4,2){\line(1,0){1.1}}
\put(27.75,2){\line(1,0){1.1}}\put(30,2){\line(1,0){1.2}}
\put(2.6,2){\makebox(0,0)[c]{\tiny{$\cdots$}}}
\put(11.6,1.95){\makebox(0,0)[c]{\tiny{$\cdots$}}}
\put(20.6,1.95){\makebox(0,0)[c]{\tiny{$\cdots$}}}
\put(29.6,1.95){\makebox(0,0)[c]{\tiny{$\cdots$}}}

\put(.3,.8){\makebox(0,0)[c]{\tiny $\beta_{-p+1}$}}
\put(5,.8){\makebox(0,0)[c]{\tiny $\beta_{-1}$}}
\put(7.2,.8){\makebox(0,0)[c]{\tiny $\beta_{-\ov{q}}$}}
\put(9.4,.8){\makebox(0,0)[c]{\tiny $\beta_{-\ov{q}+1}$}}
\put(13.85,.8){\makebox(0,0)[c]{\tiny $\beta_{-\ov{1}}$}}
\put(16.25,.8){\makebox(0,0)[c]{\tiny $\beta_{0}$}}
\put(18.4,.8){\makebox(0,0)[c]{\tiny $\beta_1$}}

\put(22.9,.8){\makebox(0,0)[c]{\tiny $\beta_{m-1}$}}
\put(25.1,.8){\makebox(0,0)[c]{\tiny $\beta_m$}}
\put(27.3,.8){\makebox(0,0)[c]{\tiny $\beta_{\ov{1}}$}}
\put(31.75,.8){\makebox(0,0)[c]{\tiny $\beta_{\ov{n-1}}$}}
\end{picture}
\end{center}

\subsubsection{}

The ortho-symplectic Lie superalgebra $\mf{osp}(2m|2n)$ is a
subalgebra of
$\gl(2m|2n)$ 
consisting of the linear transformations that preserve a
non-degenerate even supersymmetric bilinear form
$(\cdot\vert\cdot)$, namely,
$\osp(2m|2n)=\osp(2m|2n)_{\bar{0}}\oplus\osp(2m|2n)_{\bar{1}}$
with $\osp(2m|2n)_{\kappa}$ equal to
\begin{equation*}
\{A\in\gl(2m|2n)_\kappa\vert (Av|w)=(-1)^{1+\kappa{\rm
deg}v}(v|Aw) \text{ for homogeneous } v,w \in \C^{2m|2n} \}.
\end{equation*}
Below let $(\cdot\vert\cdot)$  be the supersymmetric bilinear form
corresponding to the matrix (see (\ref{side:diagonal}) for $J_k$)
\begin{equation*}
\begin{pmatrix}
J_{2m}&0&0\\
0&0&J_{n}\\
0&-J_{n}&0
\end{pmatrix}.
\end{equation*}
The Cartan subalgebra $\ov{\h}$ of $\osp(2m|2n)$ is spanned by
$E_i:=E_{ii}-E_{2m+1-i,2m+1-i}$ and
$E_{\bar{j}}:=E_{\bar{j},\bar{j}}-E_{\ov{2n-j+1},\ov{2n-j+1}}$,
$i=1,\ldots,m$, and $j=1,\ldots,n$. We use the same notation for the
restrictions of $\varepsilon_i$ and $\delta_j$ to the Cartan
subalgebra of $\osp(2m|2n)$. Let
\begin{align*}
\ov{\Pi}^{\vee} =\{\,& \beta_0^{\vee}=-E_1-E_2,\ \beta_i^{\vee}
= E_{i}-E_{i+1} \ \ (\,i\in I_{m|n}\setminus\{\,m,\ov{n}\,\}\,),\allowdisplaybreaks  \\
& \beta_{m}^{\vee}=E_{m}+E_{\ov{1}}\},\allowdisplaybreaks \\
\ov{\Pi}=\{\,& \beta_{0}=-\varepsilon_1-\varepsilon_{2}, \
\beta_i=\varepsilon_i-\varepsilon_{i+1}\
(\,i=1, \ldots,m-1),\allowdisplaybreaks \\
& \beta_{m}=\varepsilon_m-\delta_{1},\ \beta_{\ov{j}}
 =\delta_j-\delta_{j+1} \ (\,j=1,\ldots,n-1)\, \},\allowdisplaybreaks \\
\ov{\Delta}^+ =\{\,& \pm\varepsilon_i-\varepsilon_j,\
\pm\varepsilon_i-\delta_j,\ -2\delta_i,\ \pm \delta_i-\delta_j \
(i<j) \,\}.
\end{align*}
The associated Dynkin diagram is as follows:
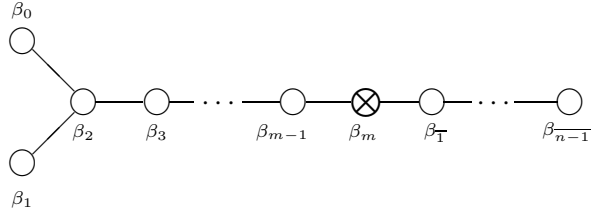
\begin{center}
\hskip -3cm \setlength{\unitlength}{0.16in} \medskip
\begin{picture}(24,5.5)
\put(6,0){\makebox(0,0)[c]{$\bigcirc$}}
\put(6,4){\makebox(0,0)[c]{$\bigcirc$}}
\put(8,2){\makebox(0,0)[c]{$\bigcirc$}}
\put(10.4,2){\makebox(0,0)[c]{$\bigcirc$}}
\put(14.85,2){\makebox(0,0)[c]{$\bigcirc$}}
\put(17.25,2){\makebox(0,0)[c]{$\bigotimes$}}
\put(19.4,2){\makebox(0,0)[c]{$\bigcirc$}}
\put(23.9,2){\makebox(0,0)[c]{$\bigcirc$}}
\put(6.35,0.3){\line(1,1){1.35}} \put(6.35,3.7){\line(1,-1){1.35}}
\put(8.4,2){\line(1,0){1.55}} \put(10.82,2){\line(1,0){0.8}}
\put(13.2,2){\line(1,0){1.2}} \put(15.28,2){\line(1,0){1.45}}
\put(17.7,2){\line(1,0){1.25}} \put(19.8,2){\line(1,0){0.9}}
\put(22.1,2){\line(1,0){1.4}}
\put(12.5,1.95){\makebox(0,0)[c]{$\cdots$}}
\put(21.5,1.95){\makebox(0,0)[c]{$\cdots$}}
\put(6,5){\makebox(0,0)[c]{\tiny $\beta_0$}}
\put(6,-1.2){\makebox(0,0)[c]{\tiny $\beta_1$}}
\put(8,1){\makebox(0,0)[c]{\tiny $\beta_2$}}
\put(10.4,1){\makebox(0,0)[c]{\tiny $\beta_3$}}
\put(14.5,1){\makebox(0,0)[c]{\tiny $\beta_{m-1}$}}
\put(17.15,1){\makebox(0,0)[c]{\tiny $\beta_m$}}
\put(19.5,1){\makebox(0,0)[c]{\tiny $\beta_{\ov{1}}$}}
\put(23.8,1){\makebox(0,0)[c]{\tiny $\beta_{\ov{n-1}}$}}
\end{picture}
\end{center}

\subsubsection{}

Taking the skew-supersymmetric bilinear form corresponding to the
matrix
\begin{equation*}
\begin{pmatrix}
0&J_m&0\\
-J_m&0&0\\
0&0&J_{2n}
\end{pmatrix},
\end{equation*}
we can construct similarly the Lie superalgebra $\spo(2m|2n)$, which
is isomorphic to $\osp(2n|2m)$.  Similarly we define the Cartan
subalgebra $\ov{\h}$, $\varepsilon_i$, and $\delta_j$. Let
\begin{align*}
\ov{\Pi}^{\vee}=\{\,& \beta_0^{\vee}=-E_1,\ \beta_i^{\vee}= E_{i}-E_{i+1} \ \ (\,i\in I_{m|n}\setminus\{\,m\, ,\ov{n}\}\,),  \\
& \beta_{m}^{\vee}=E_{m}+E_{\ov{1}} \},\allowdisplaybreaks \\
\ov{\Pi}=\{\,& \beta_{0}=-2\varepsilon_1, \
\beta_i=\varepsilon_i-\varepsilon_{i+1}\
(\,i=1, \ldots,m-1), \\
& \beta_{m}=\varepsilon_m-\delta_{1},\ \beta_{\ov{j}}=\delta_j-\delta_{j+1} \ (\,j=1,\ldots,n-1)\, \},\allowdisplaybreaks \\
\ov{\Delta}^+ =\{\,& -2\varepsilon_i,\
\pm\varepsilon_i-\varepsilon_j,\ \pm\varepsilon_i-\delta_j,\ \pm
\delta_i-\delta_j \ (i<j) \,\}.
\end{align*}
The associated Dynkin diagram is as follows:
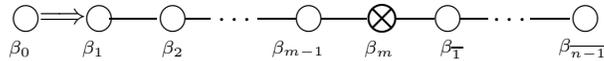
\begin{center}
\hskip -3cm \setlength{\unitlength}{0.16in}
\begin{picture}(24,4)
\put(5.6,2){\makebox(0,0)[c]{$\bigcirc$}}
\put(8,2){\makebox(0,0)[c]{$\bigcirc$}}
\put(10.4,2){\makebox(0,0)[c]{$\bigcirc$}}
\put(14.85,2){\makebox(0,0)[c]{$\bigcirc$}}
\put(17.25,2){\makebox(0,0)[c]{$\bigotimes$}}
\put(19.4,2){\makebox(0,0)[c]{$\bigcirc$}}
\put(23.9,2){\makebox(0,0)[c]{$\bigcirc$}}
\put(8.35,2){\line(1,0){1.5}}
\put(10.82,2){\line(1,0){0.8}}
\put(13.2,2){\line(1,0){1.2}}
\put(15.28,2){\line(1,0){1.45}}
\put(17.7,2){\line(1,0){1.25}}
\put(19.81,2){\line(1,0){0.9}}
\put(22.1,2){\line(1,0){1.28}}
\put(6.8,2){\makebox(0,0)[c]{$\Longrightarrow$}}
\put(12.5,1.95){\makebox(0,0)[c]{$\cdots$}}
\put(21.5,1.95){\makebox(0,0)[c]{$\cdots$}}
\put(5.4,1){\makebox(0,0)[c]{\tiny $\beta_0$}}
\put(7.8,1){\makebox(0,0)[c]{\tiny $\beta_1$}}
\put(10.4,1){\makebox(0,0)[c]{\tiny $\beta_2$}}
\put(14.5,1){\makebox(0,0)[c]{\tiny $\beta_{m-1}$}}
\put(17.15,1){\makebox(0,0)[c]{\tiny $\beta_m$}}
\put(19.5,1){\makebox(0,0)[c]{\tiny $\beta_{\ov{1}}$}}
\put(23.8,1){\makebox(0,0)[c]{\tiny $\beta_{\ov{n-1}}$}}
\end{picture}
\end{center}

\subsubsection{}

The skew-supersymmetric bilinear form corresponding to the matrix
\begin{equation*}
\begin{pmatrix}
0&J_m&0\\
-J_m&0&0\\
0&0&J_{2n+1}
\end{pmatrix}
\end{equation*}
gives rise to $\spo(2m|2n+1)$.  The Cartan subalgebra $\ov{\h}$ of
$\spo(2m|2n+1)$ is spanned by $E_i:=E_{ii}-E_{2m+1-i,2m+1-i}$ and
$E_{\bar{j}}:=E_{\bar{j},\bar{j}}-E_{\ov{2n-j+2},\ov{2n-j+2}}$,
$i=1,\ldots,m$, and $j=1,\ldots,n$. We define $\varepsilon_i$, and
$\delta_j$ analogously. Let
\begin{align*}
\ov{\Pi}^{\vee}=\{\,& \beta_0^{\vee}=E_1,\ \beta_i^{\vee}= E_{i}-E_{i+1} \ \ (\,i\in I_{m|n}\setminus\{\,m\, ,\ov{n}\}\,),  \\
& \beta_{m}^{\vee}=E_{m}+E_{\ov{1}} \},\allowdisplaybreaks \\
\ov{\Pi}=\{\,& \beta_{0}=-\varepsilon_1, \
\beta_i=\varepsilon_i-\varepsilon_{i+1}\
(\,i=1, \ldots,,m-1), \\
& \beta_{m}=\varepsilon_m-\delta_{1},\ \beta_{\ov{j}}=\delta_j-\delta_{j+1} \ (\,j=1,\ldots,n-1)\, \},\allowdisplaybreaks \\
\ov{\Delta}^+ =\{\,& -\varepsilon_i,\ -2\varepsilon_i,\
\pm\varepsilon_i-\varepsilon_j,\ \pm\varepsilon_i-\delta_j,\
-\delta_j,\ \pm \delta_i-\delta_j \ (i<j) \,\}.
\end{align*}
The associated Dynkin diagram is as follows (\ \ \makebox(6,6){\circle*{9}} denotes an odd non-isotropic simple root):
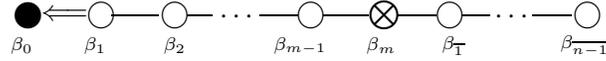
\begin{center}
\hskip -3cm \setlength{\unitlength}{0.16in}
\begin{picture}(24,4)
\put(5.6,2){\circle*{0.9}} \put(8,2){\makebox(0,0)[c]{$\bigcirc$}}
\put(10.4,2){\makebox(0,0)[c]{$\bigcirc$}}
\put(14.85,2){\makebox(0,0)[c]{$\bigcirc$}}
\put(17.25,2){\makebox(0,0)[c]{$\bigotimes$}}
\put(19.4,2){\makebox(0,0)[c]{$\bigcirc$}}
\put(23.9,2){\makebox(0,0)[c]{$\bigcirc$}}
\put(8.35,2){\line(1,0){1.5}} \put(10.82,2){\line(1,0){0.8}}
\put(13.2,2){\line(1,0){1.2}} \put(15.28,2){\line(1,0){1.45}}
\put(17.7,2){\line(1,0){1.25}} \put(19.81,2){\line(1,0){0.9}}
\put(22.1,2){\line(1,0){1.28}}
\put(6.8,2){\makebox(0,0)[c]{$\Longleftarrow$}}
\put(12.5,1.95){\makebox(0,0)[c]{$\cdots$}}
\put(21.5,1.95){\makebox(0,0)[c]{$\cdots$}}
\put(5.4,1){\makebox(0,0)[c]{\tiny $\beta_0$}}
\put(7.8,1){\makebox(0,0)[c]{\tiny $\beta_1$}}
\put(10.4,1){\makebox(0,0)[c]{\tiny $\beta_2$}}
\put(14.5,1){\makebox(0,0)[c]{\tiny $\beta_{m-1}$}}
\put(17.15,1){\makebox(0,0)[c]{\tiny $\beta_m$}}
\put(19.5,1){\makebox(0,0)[c]{\tiny $\beta_{\ov{1}}$}}
\put(23.8,1){\makebox(0,0)[c]{\tiny $\beta_{\ov{n-1}}$}}
\end{picture}
\end{center}

\subsection{Reductive dual pairs $(\SG, G)$}\label{dualpairs:finite}

Below we recall Howe dualities involving the Lie superalgebras
$\gl(p+m|q+n)$, $\spo(2m|2n+1)$, $\osp(2m|2n)$ and $\spo(2m|2n)$.

%
%

Let $d\in\N$. Suppose that $\la$ is a generalized partition of depth
$d$ with $\la_{m+1}\le n$ and $\la_{d-p}\ge -q$. Then
$\la^+_{m+1}\le n$, and $\la^-_{d-p}\le q$ (recall $\la^\pm$ from
Section~\ref{sec:notation}). Define $\La^{\mf a}_f(\la)\in\ov{\h}^*$
to be
\begin{align*}\label{htwt:glpmqn}
\La^{\mf a}_f(\la) & :=
-\sum_{i=-p}^{-1}\left(\langle\la^-_{d+i+1}-q\rangle+d\right)\varepsilon_i\allowdisplaybreaks
\\ & -
\sum_{j=-q}^{-1}\left((\la^-)'_{-j}-d\right)\delta_j
 +\sum_{i=1}^m \la^+_i\varepsilon_i+\sum_{j=1}^n\langle(\la^+)'_j-m\rangle \,
 \delta_j.
 \nonumber
\end{align*}

The following Howe duality was built on the special case when
$p=q=0$ obtained in \cite[Theorem 3.2]{CW1} (also cf. \cite{S2}).
According to \cite{CLZ} $S(\C^{p|q*}\otimes
\C^{d*}\oplus\C^{m|n}\otimes \C^d)$ carries a natural commuting
action of $\gl(p+m|q+n)$ and ${\rm GL}(d)$, which form a reductive
dual pair in the sense of \cite{H}.
\begin{prop}\label{glpmqn:duality} \cite[Theorem 3.3]{CLZ}
As a $\gl(p+m|q+n)\times{\rm GL}(d)$-module we have
\begin{equation*}
S(\C^{p|q*}\otimes \C^{d*}\oplus\C^{m|n}\otimes \C^d)\cong
\bigoplus_{\la}L(\gl(p+m|q+n),\La^{\mf a}_f(\la))\otimes V_{{\rm
GL}(d)}^\la,
\end{equation*}
where the summation is over all $\la\in\mathcal{P}({\rm GL}(d))$,
subject to the conditions $\la_{m+1}\le n$ and $\la_{d-p}\ge
-q$.
\end{prop}

Let $d$ be even. On the superspace $S(\C^{m|n}\otimes \C^d)$ we have
natural actions of $\osp(2m|2n)$ and ${\rm Sp}(d)$, which form a
reductive dual pair \cite[Section 3]{H}.

\begin{prop}\label{osp:duality} \cite[Theorem 5.2]{CZ2}
As an $\osp(2m|2n)\times{\rm Sp}(d)$-module we have
\begin{equation*}
S(\C^{m|n}\otimes \C^d)\cong \bigoplus_{\la} L(\osp(2m|2n),\La^{\mf
c}_f(\la))\otimes V_{{\rm Sp}(d)}^\la,
\end{equation*}
where the summation is over all $\la\in\mathcal{P}({\rm Sp}(d))$
with $\la_{m+1}\le n$, and
$$
\La^{\mf c}_f(\la) :=\sum_{i=1}^m(\la_i+ \hf d)\varepsilon_i
+\sum_{j=1}^n(\langle\la'_j-m\rangle- \hf d) \delta_j.
$$
\end{prop}

Let $d$ be even or odd. On the superspace $S(\C^{m|n}\otimes
\C^d)$ we have a natural action of $\spo(2m|2n)$ and ${\rm O}(d)$,
which form a reductive dual pair \cite[Section 3]{H}.

\begin{prop}\label{spo:duality} \cite[Theorem 5.1]{CZ2}
Let $d=2\ell$ or $2\ell+1$ with $\ell \in \Z_+$. As an
$\spo(2m|2n)\times{\rm O}(d)$-module we have
\begin{equation*}
S(\C^{m|n}\otimes \C^d)\cong \bigoplus_{\la}
L(\spo(2m|2n),\La^{\mf d}_f(\la))\otimes V_{{\rm O}(d)}^\la,
\end{equation*}
where the sum is over all $\la\in\mathcal{P}({\rm O}(d))$ with
$\la_{m+1}\le n$, and
$$
\La^{\mf d}_f(\la) :=\sum_{i=1}^m(\la_i+ \hf
d)\varepsilon_i+\sum_{j=1}^n(\langle\la'_j-m\rangle- \hf
d)\delta_j.
$$
\end{prop}

The following new Howe duality is worked out in Appendix
\ref{spo:pin:duality}.
\begin{prop}\label{spopin-duality1} {\rm (\thmref{spopin-duality})} Let $d$ be even.
As an $\spo(2m|2n+1)\times{\rm Pin}(d)$-module we have
\begin{equation*}
S(\C^{2m|2n+1}\otimes\C^{\frac{d}{2}})\cong \bigoplus_{\la}
L(\spo(2m|2n+1),\La^{\mf b}_f(\la))\otimes V^\la_{{\rm Pin}(d)},
\end{equation*}
where the summation is over all $\la\in\mathcal{P}({\rm Pin}(d))$
with  $\la_{m+1}\le n$, and $$\La^{\mf
b}_f(\la):=\sum_{i=1}^m(\la_i+ \hf d)\varepsilon_i
+\sum_{j=1}^n(\langle\la'_j-m\rangle- \hf d) \delta_j.$$
\end{prop}

The simple $\SG$-modules constructed in this section will be
referred to as {\em oscillator modules} of $\SG$.

From now on, we mean by $(\G,G)$ and $(\SG,G)$ one of the dual pairs
of type $\mf{x}\in\{\mf{a,b,c,d}\}$ in Sections
\ref{classical:dualpairs} and \ref{dualpairs:finite}, respectively.
We let $\mathcal P^{\SG} (G)$ be the subset of weights in $\mathcal
P (G)$ that appear in the $(\SG,G)$-duality decompositions.

\section{Character formulas of oscillator modules of Lie
superalgebras} \label{sec:char}

In this section we derive character formulas for the oscillator
representations of $\SG$, essentially only using the character
formulas of the integrable modules of infinite dimensional Lie
algebra $\G$. To that end, we introduce (trivial) central extensions
of $\SG$.

\subsection{Central extensions} \label{sec:centralExt}

Given $\mathfrak J \in \gl(p+m|q+n)$, we define
\begin{equation*}
\beta_{\mathfrak J}(A,B):=\text{Str}([\mathfrak J,A]B), \qquad
A,B\in\gl(p+m|q+n),
\end{equation*}
where $\text{Str}(c_{ij}) = \sum_{i=-p}^{-1}c_{ii}
-\sum_{j=-\bar{q}}^{-\bar{1}}c_{jj} +\sum_{i=1}^mc_{ii}
-\sum_{j=\bar{1}}^{\bar{n}} c_{jj}. $
 It follows that
$\beta_{\mathfrak J}$ is a $2$-cocycle which defines a central
extension of $\gl(p+m|q+n)$.  We note that this cocycle is a
coboundary, since we can construct this extension on the Lie
superalgebra $\gl(p+m|q+n)\oplus\C K$, with a central element $K$,
as follows. Define for $A\in\gl(p+m|q+n)$
\begin{equation*}
\widehat{A}:=A-{\rm Str}({\mathfrak J}A)K.
\end{equation*}
Then $[\widehat{A},\widehat{B}]=\widehat{[A,B]}+\beta_{\mathfrak
J}(A,B)K$.

From now on we let
\begin{equation}\label{matrixJ}
{\mathfrak J} =\begin{pmatrix}
-I_{p}&0&0&0\\
0&-I_q&0&0\\
0&0&0&0\\
0&0&0&0
\end{pmatrix}.
\end{equation}
We denote by $\hgl(p+m|q+n)$ the resulting central extension of
$\gl(p+m|q+n)$ by the one-dimensional center $\C K$. The simple
roots of the derived algebra of $\hgl(p+m|q+n)$ are the same as
those of the derived algebra of $\gl(p+m|q+n)$.  The simple coroots
are also the same except that
$\beta^\vee_{0}=E_{-\bar{1},-\bar{1}}+E_{11}$ is replaced by
$\widehat{E}_{-\bar{1},-\bar{1}}+\widehat{E}_{11}+K$.

Consider the $(2m+2n)\times(2m+2n)$ matrix of the form
\begin{equation}\label{matrixJ'}
{\mathfrak J}'=\hf\begin{pmatrix}
-I_{m}&0&0&0\\
0&I_m&0&0\\
0&0&-I_n&0\\
0&0&0&I_n\\
\end{pmatrix}.
\end{equation}
The cocycle $\beta_{{\mathfrak J}'}$ gives rise to a central
extension of $\gl(2m|2n)$ by a one-dimensional center $\C K$. The
induced central extensions of $\osp(2m|2n)$ and $\spo(2m|2n)$ by a
one-dimensional center $\C K$ are denoted by $\hosp(2m|2n)$ and
$\hspo(2m|2n)$, respectively.


Similarly we let $\widehat{\spo}(2m|2n+1)$ stand for the central
extension of $\spo(2m|2n+1)$ by the one-dimensional center $\C K$,
induced from the central extension of $\gl(2m|2n+1)$ associated with
$\beta_{\mf{J}''}$, where $\mf{J}''$ is the following
$(2m+2n+1)\times (2m+2n+1)$ matrix
\begin{equation}\label{matrixJ''}
{\mathfrak J}''=\hf\begin{pmatrix}
-I_{m}&0&0&0&0\\
0&I_m&0&0&0\\
0&0&-I_n&0&0\\
0&0&0&0&0\\
0&0&0&0&I_n\\
\end{pmatrix}.
\end{equation}

\subsection{$\hgl(p+m|q+n)$}



For later convenience of comparing with infinite dimensional Lie
algebras, we introduce an additional element $K$, which commutes
with $\gl(p+m|q+n)$ and ${\rm GL}(d)$, and regard
$S(\C^{p|q*}\otimes\C^{d*}\oplus\C^{m|n}\otimes \C^d)$ as a module
over $(\gl(p+m|q+n)\oplus\C K)\times{\rm GL}(d)$, where we declare
that $K$ acts as the scalar $d$.  For ${\mathfrak J}$ as in
(\ref{matrixJ}) set
\begin{align}\label{eijhat}
\widehat{E}_{ij}&:= E_{ij}-{\rm Str}({\mathfrak J}E_{ij})K.
\end{align}
Then $\sum_{i,j}\C\widehat{E}_{ij}\oplus\C K$ gives rise to the
central extension $\hgl(p+m|q+n)$. Now $\hgl(p+m|q+n)$ and ${\rm
GL}(d)$ form a reductive dual pair on
$S(\C^{p|q*}\otimes\C^{d*}\oplus\C^{m|n}\otimes \C^d)$.

Define $\SLa^{\mf a}_0 \in {\ov \h}^*$ by $\langle\SLa^{\mf a}_0,K
\rangle = 1 \text{ and } \langle\SLa^{\mf a}_0,E_{kk}\rangle=0
\text{ for all } k\in I_{p+m|q+n}.$ For $\la\in \mathcal{P}({\rm
GL}(d))$ with $\la_{m+1}\le n$ and $\la_{d-p}\ge -q$, we define
\begin{align*}
\widehat{\La}^{\mf a}_f(\la) &:=d\widetilde{\La}^{\mf a}_0
-\sum_{i=-p}^{-1}\langle \la^-_{d+i+1}-q\rangle{\varepsilon}_i\allowdisplaybreaks \\
& -\sum_{j=-q}^{-1}(\la^-)'_{-j}{\delta}_j + \sum_{i=1}^m
\la^+_i{\varepsilon}_i+\sum_{j=1}^n\langle(\la^+)'_j-m\rangle{\delta}_j,
\nonumber
\end{align*}
where $ \widetilde{\La}^{\mf a}_0 = \SLa^{\mf
a}_0-\sum_{i=-p}^{-1}\varepsilon_i+\sum_{j=-q}^{-1}\delta_j.$

Recall that a partition $\la\in\mc{P}^+$ is called an {\em
$(m|n)$-hook partition}, if $\la_{m+1}\le n$ \cite{BR}. For an
$(m|n)$-hook partition $\la$, we let $
\la^\natural:=\sum_{i=1}^m\la_i\varepsilon_i
+\sum_{j=1}^n\langle\la'_j-m\rangle\delta_j\in\h_{m|n}^*$.

Let $\xi:=\{\,{\xi}_1,\ldots,{\xi}_n\,\}$,
$y:=\{\,{y}_1,\ldots,{y}_m\,\}$,
${\eta}^{-1}:=\{\,{\eta}^{-1}_{-1},\ldots,{\eta}^{-1}_{-q}\,\}$, and
$x^{-1}:=\{\,{x}^{-1}_{-1},\ldots,{x}^{-1}_{-p}\,\}$ be
indeterminates. The character of $L(\gl(m|n),\la^\natural)$,
i.e.~the trace of the operator
$\prod_{i=1}^my_i^{E_{ii}}\prod_{j=1}^n\xi_j^{E_{\bar{j}\bar{j}}}$,
is given by the {\em hook Schur polynomial} \cite{BR,Se}
\begin{equation}\label{hook:def}
hs_\la(y,\xi):=
\sum_{\mu\subseteq\la}s_\mu(y_1,\ldots,y_m)s_{\la'/\mu'}(\xi_1,\ldots,\xi_n).
\end{equation}
Observe that $hs_\la(y,\xi)=hs_{\la'}(\xi,y)$.

\begin{thm}\label{superchar:A}
For $\la\in \mathcal{P}({\rm GL}(d))$ with $\la_{m+1}\le n$ and
$\la_{d-p}\ge -q$, we have
\begin{align*}
{\rm ch}L&(\hgl(p+m|q+n), \widehat{\La}^{\mf a}_f(\la)) \\
&= \frac{\sum_{k=0}^\infty (-1)^k \sum_{w\in W^0_k(\mf
a)}hs_{\la^+_w}(\xi,y)\, hs_{\la^-_w}(\eta^{-1},x^{-1})}
{\prod_{i,j,s,t}(1-{x}_{i}^{-1}{y}_s)(1+{x}_{i}^{-1}{\xi}_t)^{-1}
(1+{\eta}_{i}^{-1}{y}_s)^{-1}(1-{\eta}_{i}^{-1}{\xi}_t)}.
\end{align*}
\end{thm}

\begin{proof}
Computing the trace of the operator
$\prod_{i,j,s,t}x_i^{\widehat{E}_{ii}}\eta_j^{\widehat{E}_{\bar{j},\bar{j}}}
y_s^{\widehat{E}_{ss}}\xi_t^{\widehat{E}_{\bar{t},\bar{t}}}\prod_{k=1}^d
z_k^{e_{kk}}$ on both sides of the isomorphism in
\propref{glpmqn:duality}, where $-p\le i\le -1$, $-q\le j\le -1$,
$1\le s\le m$, and $1\le t\le n$, we obtain
%
\begin{equation}\label{glpmqn:totalchar}
\prod_{k=1}^d\prod_{i,j,s,t}
\frac{(1+{\eta}_j^{-1}z_k^{-1})(1+{\xi}_{t}z_k)}{(1-{x}_i^{-1}z_k^{-1})(1-{y}_sz_k)}
= \sum_{\la}{\rm ch}L(\hgl(p+m|q+n),\widehat{\La}^{\mf
a}_f(\la)){\rm ch}V^\la_{{\rm GL}(d)},
\end{equation}
where the summation is over all $\la\in \mathcal{P}({\rm GL}(d))$
with $\la_{m+1}\le n$ and $\la_{d-p}\ge -q$.

Recall the identity \eqnref{combid-classical1} which results from
the $(\widehat{\gl}_\infty, {\rm GL}(d))$-duality. Now
(\ref{combid-classical1}) and \propref{char:schur} imply that
\begin{align*}
\prod_{k=1}^d & \prod_{n\in\N}
(1+x_nz_k)(1+x_{1-n}^{-1}z^{-1}_{k}) = \sum_{\la\in {\mc P}({\rm
GL}(d))}  {\rm ch}L(\widehat{\gl}_\infty,\La^{\mf a}(\la)) \;
{\rm ch}V^\la_{{\rm GL}(d)}\allowdisplaybreaks \\
&= \sum_{\la\in {\mc P}({\rm GL}(d))} \frac{\sum_{k=0}^\infty (-1)^k
\sum_{w\in W^0_k}s_{\la^+_w}(x_1,x_2,\ldots)
s_{\la^-_w}(x^{-1}_0,x^{-1}_{-1},\ldots)}{\prod_{i,j}(1-x_{-i+1}^{-1}x_j)}\times
{\rm ch}V^\la_{{\rm GL}(d)}.
\end{align*}
We set ${\xi}_t:=x_t$ (\,$t=1,\ldots,n$\,), ${y}_s:=x_{n+s}$
(\,$s\in\N$\,), ${\eta}_{j-1}:=x_j$ (\,$j=0,\ldots,-q+1$\,), and
${x}_{i}:=x_{-q+1+i}$ (\,$i\in-\N$\,), and rewrite the above as
\begin{align*}
\prod_{k=1}^d &
\prod_{i,j,s,t}(1+{x}^{-1}_iz^{-1}_k)(1+{\eta}_{j}^{-1}z^{-1}_{k})
(1+{y}_sz_k)(1+{\xi}_{t}z_{k})= \sum_{\la\in {\mc P}({\rm GL}(d))} {\rm ch}V^\la_{{\rm GL}(d)}\times \\
& \frac{\sum_{k=0}^\infty (-1)^k \sum_{w\in
W^0_k}s_{\la^+_w}({\xi}_1,\ldots,{\xi}_n,{y}_1,\ldots)
s_{\la^-_w}({\eta}^{-1}_{-1},\ldots,{\eta}^{-1}_{-q},{x}^{-1}_{-1},\ldots)}
{\prod_{i,j,s,t}(1-{x}_{i}^{-1}{y}_s)(1-{x}_{i}^{-1}{\xi}_t)
(1-{\eta}_{i}^{-1}{y}_s)(1-{\eta}_{i}^{-1}{\xi}_t)}.
\end{align*}

Let $\omega$ be the standard involution of symmetric functions that
interchanges elementary symmetric functions and complete symmetric
functions (e.g.~\cite[(2.7)]{Mac}). Note that $\omega$ sends $s_\mu$
to $s_{\mu'}$ and recall (\ref{hook:def}). Applying $\omega$ twice
to the above identity, once on the variables ${x}^{-1}_{-1},
{x}_{-2}^{-1},\ldots$ and another on the variables
${y}_1,{y}_2,\ldots$, we obtain
\begin{align*}
\prod_{k=1}^d\prod_{i,j,s,t}
&\frac{(1+{\xi}_{t}z_{k})(1+{\eta}_{j}^{-1}z^{-1}_{k})}{(1-{y}_sz_k)(1-{x}^{-1}_iz^{-1}_k)}
=\sum_{\la\in {\mc P}({\rm GL}(d))}{\rm ch}V^\la_{{\rm GL}(d)}\ \times\\
&\frac{\sum_{k=0}^\infty (-1)^k \sum_{w\in
W^0_k}hs_{\la^+_w}({\xi}_1,\ldots,{\xi}_n,{y}_1,\ldots)
hs_{\la^-_w}({\eta}^{-1}_{-1},\ldots,{\eta}^{-1}_{-q},{x}^{-1}_{-1},\ldots)}
{\prod_{i,j,s,t}(1-{x}_{i}^{-1}{y}_s)(1+{x}_{i}^{-1}{\xi}_t)^{-1}
(1+{\eta}_{i}^{-1}{y}_s)^{-1}(1-{\eta}_{i}^{-1}{\xi}_t)}.
\end{align*}
Finally setting $x^{-1}_{-p-i}=y_{m+i}=0$ for $i\in\N$, we get an
identity which shares the same left hand side as
(\ref{glpmqn:totalchar}). The theorem now follows from comparison of
the right hand sides and the linear independence of the characters
$\{{\rm ch}V^\la_{{\rm GL}(d)}\vert\la\in\mc{P}({\rm GL}(d))\}$.
\end{proof}

\begin{rem}
The character here, which is the trace of the corresponding
operators with hats from $\hgl(p+m|q+n)$, differs from the usual one
from $\gl(p+m|q+n)$ by exactly a factor of $\left(
\frac{\eta_{-q}\cdots\eta_{-1}}{x_{-p}\cdots x_{-1}}\right)^{d}$.
Similar remarks apply below to $\spo$/$\osp$-characters versus
$\hspo$/$\hosp$-characters, where the power $d$ here is replaced by
$\frac{d}2$.

Our formula here differs from the character formula obtained in
\cite[Theorem~5.3]{CLZ}.
\end{rem}

\begin{rem} \label{omega-ring}
Denote by $\omega^{\mf a}$ a linear extension of the composition of maps in the proof of
\thmref{superchar:A} that sends ${\rm
ch}L(\widehat{\gl}_\infty,\La^{\mf a}(\la))$ to ${\rm
ch}L(\hgl(p+m|q+n),\widehat{\La}^{\mf a}_f(\la))$. Since each map in the composite is
either an involution $\omega$ or an evaluation of some variables at
zero, $\omega^{\mf a}$ respects the multiplication of characters.
\end{rem}

\subsection{$\hosp(2m|2n)$} Suppose that $d$ is even.
Recalling that $\hosp(2m|2n)$ is the central extension induced from
$\hgl(2m|2n)$ with respect to ${\mathfrak J}'$ in \eqref{matrixJ'},
we first introduce an element $K$ that acts on $S(\C^{m|n}\otimes
\C^d)$ as the scalar $\frac{d}{2}$, and that commutes with the
actions of $\osp(2m|2n)$ and ${\rm Sp}(d)$. Setting
$$\widehat{A}:=A-{\rm Str}({\mathfrak J}'A)K\in\osp(2m|2n)\oplus\C
K,\quad A\in\osp(2m|2n),$$ gives rise to an action of $\hosp(2m|2n)$
on $S(\C^{m|n}\otimes\C^d)$. Clearly $\hosp(2m|2n)$ and ${\rm
Sp}(d)$ also form a reductive dual pair. Define an element
$\SLa^{\mf c}_0\in\ov{\h}^*$ by $\langle\SLa^{\mf c}_0,K \rangle=1$
and $\langle\SLa^{\mf c}_0,E_{i}\rangle=\langle\SLa^{\mf
c}_0,E_{\ov{j}}\rangle=0$, for all $1\le i\le m$ and $1\le j\le n$.
Further for $\la\in\mathcal{P}({\rm Sp}(d))$ with $\la_{m+1}\leq n$,
we define

\begin{equation*}
\widehat{\La}^{\mf c}_f(\la): = \frac{d}{2}\widetilde{\La}^{\mf
c}_0+\sum_{i=1}^m\la_i\varepsilon_i+\sum_{j=1}^n\langle\la'_j-m\rangle\delta_j,
\end{equation*}
where $\widetilde{\La}^{\mf c}_0 = \SLa^{\mf
c}_0+\sum_{i=1}^{m}\varepsilon_i-\sum_{j=1}^{n}\delta_j$. Then
$\{\varepsilon_i,\delta_j,\widetilde{\La}_0^{\mf c}\}$ is the basis
dual to $\{\widehat{E}_{i},\widehat{E}_{\ov{j}},K\}$.

\begin{thm}\label{superchar:OSP}
For $\la\in\mathcal{P}({\rm Sp}(d))$ with $\la_{m+1}\leq n$, we have
\begin{align*}
{\rm ch}L(\hosp(2m|2n),\widehat{\La}^{\mf
c}_f(\la))=\frac{\sum_{k=0}^\infty (-1)^k \sum_{w\in W^0_k (\mf
c)}hs_{\la_w}(\eta,x)} {\prod_{1\le i < j\le m}\prod_{1\le s \le
t\le n}(1-\eta_i\eta_j)(1-x_sx_t)(1+\eta_ix_s)^{-1}},
\end{align*}
where $\eta:=\{\,{\eta}_1,\ldots,{\eta}_n\,\}$,
$x:=\{\,{x}_1,\ldots,{x}_m\,\}$.
\end{thm}
\begin{proof}
Computing the trace of the operator
$\prod_{i,j}\eta_i^{\widehat{E}_{i}}x_j^{\widehat{E}_{\bar{j}}}\prod_{k=1}^{\frac{d}{2}}
z_k^{{e}_{k}}$ on both sides of the isomorphism in
\propref{osp:duality}, where $1\le i\le m$  and $1\le j\le n$, we
obtain
%
\begin{equation}\label{ops2m2n:totalchar}
\prod_{k=1}^{\frac{d}{2}}\prod_{i,j}
\frac{(1+\eta_jz_k^{-1})(1+\eta_{j}z_k)}{(1-x_iz_k^{-1})(1-x_iz_k)}
= \sum_{\substack{\la\in\mc{P}({\rm Sp}(d)) \\ \la_{m+1}\leq n}}{\rm
ch}L(\hosp(2m|2n),\widehat{\La}^{\mf c}_f(\la)){\rm ch}V^\la_{{\rm
Sp}(d)}.
\end{equation}
Replacing $\text{ch}\, L(\G, \La^{\mf{c}}(\la))$ in
\eqnref{combid-classical-c} by the expression in
\propref{char:schur}, we obtain an identity of symmetric functions
in variables $x_1, x_2,\ldots$. We apply to this identity the
involution on the ring of symmetric functions in $x_1,x_2,\ldots$.
Next, we replace $x_i$ by $\eta_i$, $i=1,\ldots,n$, and then
$x_{i+n}$ by $x_i$, $i\in\N$. Finally, we put $x_i=0$ for $i\geq
n+1$. Under the composition of those maps, we obtain a new identity
which shares the same left hand side as \eqnref{ops2m2n:totalchar}.
Comparing the right hand sides of this new identity and of
\eqnref{ops2m2n:totalchar}, we obtain the result thanks to the
linear independence of $\{{\rm ch}V^\la_{{\rm Sp}(d)}\vert
\la\in\mc{P}({\rm Sp}(d))\}$.
\end{proof}

We denote by $\omega^{\mf c}$ the map which sends ${\rm
ch}L(\mathfrak{c}_{\infty},\La^{\mf c}(\la))$ to ${\rm
ch}L(\hosp(2m|2n),\widehat{\La}^{\mf c}_f(\la))$ and extend
$\omega^{\mf c}$ by linearity (compare \eqnref{combid-classical-c}
and \eqnref{ops2m2n:totalchar}). Note that $\omega^{\mf c}$
respects the multiplication of characters.

\subsection{$\hspo(2m|2n)$}

Recall the central extension $\hspo(2m|2n)$ induced from
$\hgl(2m|2n)$ determined by ${\mathfrak J}'$ in \eqref{matrixJ'}.
Introduce an element $K$ that acts on $S(\C^{m|n}\otimes \C^d)$ as
the scalar $\frac{d}{2}$ and commutes with the actions of
$\spo(2m|2n)$ and ${\rm O}(d)$. Putting $\widehat{A}:=A-{\rm
Str}({\mathfrak J}'A)K\in\spo(2m|2n)\oplus\C K$, for
$A\in\spo(2m|2n)$, defines an action of $\hspo(2m|2n)$ on
$S(\C^{m|n}\otimes\C^d)$. Then $\hspo(2m|2n)$ and ${\rm O}(d)$ form
a reductive dual pair. Define an element $\SLa^{\mf
d}_0\in\ov{\h}^*$ by $\langle\SLa^{\mf d}_0,K \rangle=1$ and
$\langle\SLa^{\mf d}_0,E_{i}\rangle=\langle\SLa^{\mf
d}_0,E_{\ov{j}}\rangle=0$, for all $1\le i\le m$ and $1\le j\le n$.

When $d=2\ell$, calculating the trace of the operator
$\prod_{i,j}\eta_i^{\widehat{E}_{i}}x_j^{\widehat{E}_{\bar{j}}}\prod_{k=1}^{\ell}
z_k^{{e}_{k}}$ on both sides of the isomorphism in
\propref{spo:duality}, where $1\leq i\leq m$ and $1\leq j\leq n$,
gives
\begin{equation}\label{spo2m2n:totalchar1}
\prod_{k=1}^{\ell}\prod_{i,j}
\frac{(1+\eta_jz_k^{-1})(1+\eta_{j}z_k)}{(1-x_iz_k^{-1})(1-x_iz_k)}=
\sum_{\substack{\la\in\mc{P}({\rm O}(2\ell)) \\ \la_{m+1}\leq
n}}{\rm ch}L(\hspo(2m|2n),\widehat{\La}^{\mf d}_f(\la)){\rm
ch}V^\la_{{\rm O}(2\ell)},
\end{equation}
and when $d=2\ell+1$, the trace of the operator
$\prod_{i,j}\eta_i^{\widehat{E}_{i}}x_j^{\widehat{E}_{\bar{j}}}\prod_{k=1}^{\ell}
z_k^{{e}_{k}}(-I_d)$ gives
\begin{align}\label{spo2m2n:totalchar2}
\prod_{k=1}^{\ell}\prod_{i,j}
 & \frac{(1+\epsilon\eta_jz_k^{-1})(1+\epsilon\eta_{j}z_k)(1+\epsilon\eta_{j})}{(1-\epsilon
x_iz_k^{-1})(1-\epsilon x_iz_k)(1-\epsilon x_i)} \\
 =& \sum_{\substack{\la\in\mc{P}({\rm O}(2\ell+1)) \\ \la_{m+1}\leq
n}}{\rm ch}L(\hspo(2m|2n),\widehat{\La}^{\mf d}_f(\la)){\rm
ch}V^\la_{{\rm O}(2\ell+1)}. \nonumber
\end{align}
Here for $\la\in \mc{P}({\rm O}(d))$ with $\la_{m+1}\leq n$,
\begin{align*}
\widehat{\La}^{\mf d}_f(\la) := \dfrac{d}{2}\widetilde{\La}^{\mf
d}_0+\sum_{i=1}^m\la_i\varepsilon_i+\sum_{j=1}^n\langle\la'_j-m\rangle\delta_j,
\end{align*}
with $\widetilde{\La}^{\mf d}_0 = \SLa^{\mf
d}_0+\sum_{i=1}^{m}\varepsilon_i-\sum_{j=1}^{n}\delta_j$.

Now the next theorem follows by similar arguments as for
\thmref{superchar:OSP} using \eqref{spo2m2n:totalchar1},
\eqref{spo2m2n:totalchar2}, \eqref{combid-classical-d1} and
\eqref{combid-classical-d2} and \propref{char:schur}.

\begin{thm}\label{superchar:SPO}
Put $\eta:=\{\,{\eta}_1,\ldots,{\eta}_m\,\}$ and
$x:=\{\,{x}_1,\ldots,{x}_n\,\}$. For $\la\in\mathcal P ({\rm O}(d))$
with $\la_{m+1}\leq n$, we have
\begin{itemize}
\item[(1)] If $d=2\ell+1$, then we have
\begin{equation*}
{\rm ch}L(\hspo(2m|2n),\widehat{\La}^{\mf
d}_f(\la))=\frac{\sum_{k=0}^\infty (-1)^k \sum_{w\in
W^0_k(\mf{d})}hs_{\la_w}(\eta,x)} {\prod_{1\le i \le j\le
m}\prod_{1\le s <t\le
n}(1-\eta_i\eta_j)(1-x_sx_t)(1+\eta_ix_s)^{-1}}.
\end{equation*}

\item[(2)] If $d=2\ell$, then we have
\begin{equation*}
\begin{split}
{\rm ch}L(\hspo(2m|2n),\widehat{\La}^{\mf d}_f(\la))+&{\rm
ch}L(\hspo(2m|2n),\widehat{\La}^{\mf d}_f(\tilde{\la}))= \\
&\frac{\sum_{k=0}^\infty (-1)^k \sum_{w\in
W^0_k(\mf{d})}\bigl[hs_{\la_w}(\eta,x)+hs_{\tilde{\la}_w}(\eta,x)\bigr]}
{\prod_{1\le i \le j\le m}\prod_{1\le s <t\le
n}(1-\eta_i\eta_j)(1-x_sx_t)(1+\eta_ix_s)^{-1}}.
\end{split}
\end{equation*}
\end{itemize}
\end{thm}

\begin{rem} \thmref{superchar:SPO}~(2) gives the
character of the sum of two irreducible modules.  However, in the
case when $\la_1'=\ell$, we have $\tilde{\la} =\la$, and hence
$\la_w=\tilde{\la}_w$. Thus in this case it actually gives the
character of one irreducible module, i.e. we have an identity as in
(1). Similar remark applies in the sequel as well when dealing with
$G={\rm O}(2\ell)$ and $\la=\tilde{\la}$. In particular it applies
to the character formulas of Kostant homology groups in Theorems
\ref{mainthm} and \ref{mainthm-negative c}, and also to
(\ref{def:L}).

Character formulas in forms different from \thmref{superchar:OSP}
and \thmref{superchar:SPO} were also obtained in \cite[Theorems 6.2,
6.3]{CZ2} using \cite[Theorem 2.2]{En}. Our approach uses Howe
dualities involving infinite dimensional Lie algebras, thus
bypassing \cite{En}.
\end{rem}

Introduce a map $\omega^{\mf d}$ such that $\omega^{\mf d}({\rm
ch}L(\mathfrak{d}_{\infty},\La^{\mf d}(\la)))={\rm
ch}L(\hspo(2m|2n),\widehat{\La}^{\mf d}_f(\la))$ if $d=2\ell+1$,
and $\omega^{\mf d}\big[{\rm ch}L(\mathfrak{d}_{\infty},\La^{\mf
d}(\la))+{\rm ch}L(\mathfrak{d}_{\infty},\La^{\mf
d}(\widetilde{\la}))\big]={\rm
ch}L(\hspo(2m|2n),\widehat{\La}^{\mf d}_f(\la))+{\rm
ch}L(\hspo(2m|2n),\widehat{\La}^{\mf d}_f(\widetilde{\la}))$ if
$d=2\ell$. Extended by linearity, $\omega^{\mf d}$ sends either
side of \eqnref{combid-classical-d1} and
\eqnref{combid-classical-d2} to the corresponding side of
\eqnref{spo2m2n:totalchar1} and \eqnref{spo2m2n:totalchar2},
respectively.

\subsection{$\hspo(2m|2n+1)$}

Recall that $\hspo(2m|2n+1)$ is the central extension induced from
$\hgl(2m|2n+1)$ determined by ${\mathfrak J}''$ in
\eqref{matrixJ''}. Let $K$ act on $S(\C^{2m|2n+1}\otimes
\C^{\frac{d}{2}})$ as the scalar $\frac{d}{2}$ and commute with the
actions of $\spo(2m|2n+1)$ and ${\rm Pin}(d)$. Putting
$\widehat{A}:=A-{\rm Str}({\mathfrak J}''A)K\in\spo(2m|2n+1)\oplus\C
K$, for $A\in\spo(2m|2n+1)$, defines an action of $\hspo(2m|2n+1)$
on $S(\C^{2m|2n+1}\otimes\C^{\frac{d}{2}})$. Then $\hspo(2m|2n+1)$
and ${\rm Pin}(d)$ form a reductive dual pair. Define an element
$\SLa^{\mf b}_0\in\ov{\h}^*$ by $\langle\SLa^{\mf b}_0,K \rangle=1$
and $\langle\SLa^{\mf b}_0,E_{i}\rangle=\langle\SLa^{\mf
b}_0,E_{\ov{j}}\rangle=0$, for all $1\le i\le m$ and $1\le j\le n$.

Computing the trace of the operator
$\prod_{i,j}\eta_i^{\widehat{E}_{i}}x_j^{\widehat{E}_{\bar{j}}}\prod_{k=1}^{\ell}
z_k^{{e}_{k}}$ on both sides of the isomorphism in
\propref{spopin-duality1} gives
\begin{align}\label{spo2m2n+1:totalchar1}
\prod_{k=1}^{\ell}\prod_{i,j} (z_k^\hf+z_k^{-\hf})
&\frac{(1+\eta_jz_k^{-1})(1+\eta_{j}z_k)}{(1-x_iz_k^{-1})(1-x_iz_k)}=\\
&\sum_{\substack{\la\in\mc{P}({\rm Pin}(2\ell))\\ \la_{m+1}\leq
n}}{\rm ch}L(\hspo(2m|2n+1),\widehat{\La}^{\mf b}_f(\la)){\rm
ch}V^\la_{{\rm Pin}(2\ell)},\nonumber
\end{align}
where
\begin{align*}
\widehat{\La}^{\mf b}_f(\la) := \dfrac{d}{2}\widetilde{\La}^{\mf
b}_0+\sum_{i=1}^m\la_i\varepsilon_i+\sum_{j=1}^n\langle\la'_j-m\rangle\delta_j,
\end{align*}
with $\widetilde{\La}^{\mf b}_0 = \SLa^{\mf
b}_0+\sum_{i=1}^{m}\varepsilon_i-\sum_{j=1}^{n}\delta_j$.

Now the next theorem follows from similar arguments as for
\thmref{superchar:OSP} using \eqref{spo2m2n+1:totalchar1},
\eqref{combid-classical-b}, and \propref{char:schur}.

\begin{thm}\label{superchar:oddSPO}
Put $\eta:=\{\,{\eta}_1,\ldots,{\eta}_m\,\}$ and
$x:=\{\,{x}_1,\ldots,{x}_n\,\}$. For $\la\in\mathcal P ({\rm
Pin}(d))$ with $\la_{m+1}\leq n$, we have
\begin{equation*}
\begin{split}
&{\rm ch}L(\hspo(2m|2n+1),\widehat{\La}^{\mf b}_f(\la))
\\ &=\frac{\sum_{k=0}^\infty (-1)^k \sum_{w\in
W^0_k(\mf{b})}hs_{\la_w}(\eta,x)}{\prod_{1\le i \le j\le
m}\prod_{1\le s <t\le
n}{(1-\eta_i\eta_j)(1-x_sx_t)(1-\eta_j)}{(1+\eta_ix_s)^{-1}(1+x_s)^{-1}}}.
\end{split}
\end{equation*}
\end{thm}
We denote by $\omega^{\mf b}$ the map which sends ${\rm
ch}L(\mathfrak{b}_{\infty},\La^{\mf b}(\la))$ to ${\rm
ch}L(\hosp(2m|2n+1),\widehat{\La}^{\mf b}_f(\la))$ and extend
$\omega^{\mf b}$ by linearity.

\section{The bilinear forms and Casimir operators for Lie (super)algebras}\label{casimir:op}

\subsection{The bilinear form $(\cdot|\cdot)_c$ and the Casimir operator $\Omega$ of $\G$}\label{casimir}


Recall the symmetric bilinear form $(\cdot\vert\cdot)_c$ defined  on
$\h^*$ in \cite[Section 4.1]{CK}, which induces a  non-degenerate
invariant symmetric bilinear form on $\G'=[\G,\G]$, also denoted by
$(\cdot\vert\cdot)_c$. Using this one defines a Casimir operator
$\Omega$ which commutes with the action of $\G$ on a highest weight
$\G$-module $V$ of highest weight $\la$, and which acts as the
scalar $(\la+2\rho_c\vert\la)_c$ on $V$. The details here are
completely analogous to \secref{super:casimir} below, and we refer
to \cite[Section~4.1]{CK} for more detail.

\subsection{The bilinear form $(\cdot|\cdot)_s$ and the Casimir operator $\overline{\Omega}$ of
$\SG$}\label{super:casimir}


Suppose first that $\SG=\hgl({p+m|q+n})$. Note that
$\{\varepsilon_i,\delta_j\,\}\cup\{\,\widetilde{\La}_0^{\mf a}\}$ is
the basis of $\ov{\h}^*$ dual to the basis
$\{\widehat{E}_{kk}\,\}\cup\{\,K\}$ of $\ov{\h}$. Set
\begin{equation}\label{super:rho}
\rho_s:=\sum_{i=-p}^{-1}(-i-q)\varepsilon_i+\sum_{i=1}^m(1-i)\varepsilon_i+\sum_{j=-q}^{-1}(-j-1)\delta_{j}
+\sum_{j=1}^n(m-j)\delta_j.
\end{equation}

We choose a symmetric bilinear form $(\cdot\vert\cdot)_s$ on
$\ov{\h}^*$ satisfying
\begin{align*}
&(\la\vert \varepsilon_i)_s=
-\langle\la,E_{ii}+\frac{K}{2}\rangle, \quad
i\in\{-p,\ldots,-1,1,\ldots,m\},\\
&(\la\vert \delta_j)_s= \langle\la,E_{\bar{j}
\bar{j}}-\frac{K}{2}\rangle, \quad
j\in\{-q,\ldots,-1,1,\ldots,n\}.
\end{align*}
We check easily that
\begin{align*}
(\varepsilon_i\vert\varepsilon_j)_s=-\delta_{ij},\quad
(\varepsilon_i\vert\delta_j)_s=0,\quad
(\delta_i\vert\delta_j)_s=\delta_{ij}.
\end{align*}
Also we have
 $(\widetilde{\La}^{\mf
a}_0\vert \varepsilon_{-i})_s=(\widetilde{\La}^{\mf a}_0\vert
\delta_{-j})_s=-(\widetilde{\La}^{\mf a}_0\vert
\varepsilon_i)_s=-(\widetilde{\La}^{\mf a}_0\vert \delta_j)_s=\hf,$ for
$i,j>0$. Furthermore, recalling the simple roots $\beta_k$ for
$\SG$, we have
\begin{equation*}
(\rho_s\vert\beta_k)_s=\hf(\beta_k\vert\beta_k)_s, \ \ \ k\in
\ov{I}.
\end{equation*}

Next, suppose that $\SG$ is $\hspo(2m|2n+1)$, $\hosp(2m|2n)$ or
$\hspo(2m|2n)$. Note that $\{\varepsilon_i,
\delta_j\,\}\cup\{\,\widetilde{\La}^{\mathfrak x}_0\}$ is the basis
of $\ov{\h}^*$ dual to $\{\widehat{E}_{i},
\widehat{E}_{\ov{j}}\,\}\cup\{\,K\}$, ${\mathfrak x}\in
\{\mf{b},\mf{c},\mf{d}\}$. We set
\begin{equation}\label{super:rhocd}
\rho_s:=
\begin{cases}
\sum_{i=1}^m (-i+\hf)\varepsilon_i + \sum_{j=1}^n
(m-j+\hf)\delta_j, & \text{for ${\mathfrak x}=\mf{b}$},\\
\sum_{i=1}^m(1-i)\varepsilon_i +\sum_{j=1}^n(m-j)\delta_j, & \text{for ${\mathfrak x}=\mf{c}$}, \\
\sum_{i=1}^m -i\varepsilon_i +\sum_{j=1}^n(m-j+1)\delta_j, &
\text{for ${\mathfrak x}=\mf{d}$}.
\end{cases}
\end{equation}  We choose a symmetric bilinear form
$(\cdot\vert\cdot)_s$ on $\ov{\h}^*$ satisfying for
$\la\in\ov{\h}^*$
\begin{align*}
&(\la\vert \varepsilon_i)_s= \langle\la,E_{i}\rangle, \quad
i\in\{1,\ldots,m\},\\
&(\la\vert \delta_j)_s= -\langle\la,E_{\bar{j}}\rangle, \quad
j\in\{1,\ldots,n\}.
\end{align*}
Similarly, we can check that
$(\varepsilon_i\vert\varepsilon_j)_s=\delta_{ij}$,
$(\varepsilon_i\vert\delta_j)_s=0$,
$(\delta_i\vert\delta_j)_s=-\delta_{ij}$. Also we have
$(\widetilde{\La}^{\mathfrak x}_0\vert
\varepsilon_{i})_s=(\widetilde{\La}^{\mathfrak x}_0\vert
\delta_{j})_s=1$, and $(\rho_s|\beta_k)_s=\hf(\beta_k|\beta_k)_s$
for $k\in\ov{I}$.

For $i\in \ov{I}$ define
\begin{equation*}
\begin{aligned}
\bar{s}_i^{\mf a}&:=
\begin{cases}
-1, & \text{if $i\in\{-p+1,\ldots,-1,-\ov{q},1,\ldots,m\}$}, \\
\ \ 1, & \text{if
$i\in\{-\ov{q-1},\ldots,-\bar{1},0,\bar{1},\ldots,\ov{n-1}\}$}.
\end{cases} \\
\bar{s}_i^{\mathfrak x}&:=
\begin{cases}
 -1, & \text{if $i=0$ and ${\mathfrak x}={\mf b}$}, \\
\ \ 1, & \text{if $i=0$ and ${\mathfrak x}={\mf c}$}, \\
\ \ 2, & \text{if $i=0$ and ${\mathfrak x}={\mf d}$}, \\
\ \ 1, & \text{if } i\in\{1,\ldots,m\} \text{ and } {\mathfrak x}\in \{\mf{b, c,d}\}, \\
- 1, & \text{if } i\in\{\ov{1},\ldots,\ov{n-1}\}  \text{ and }
{\mathfrak x}\in \{\mf{b, c,d}\}.
\end{cases}
\end{aligned}
\end{equation*}

Then we have
\begin{equation*}\label{aux:scasimir2}
(\la\vert\beta_i)_s=\bar{s}^{\mathfrak
x}_i\langle\la,\beta_i^{\vee}\rangle, \quad i\in\ov{I},\
\la\in\ov{\h}^*.
\end{equation*}
By defining $({\beta}_i^{\vee}\vert{\beta}_j^{\vee})_s:=
(\bar{s}_i^{\mathfrak x}\bar{s}_j^{\mathfrak x})^{-1}
(\beta_i\vert\beta_j)_s$, we obtain a symmetric bilinear form on
the Cartan subalgebra of $\SG'=[\SG,\SG]$, which can be extended
to a non-degenerate invariant super-symmetric bilinear form on
$\SG'$ such that
\begin{equation*}
(\ov{e}_i\vert \ov{f}_{j})_s=\delta_{ij}/\bar{s}_i^{\mathfrak x},
\end{equation*}
where $\ov{e}_i$ and $\ov{f}_i$ denote the Chevalley generators of
$\SG'$ with $[\ov{e}_i,\ov{f}_i]=\beta_i^{\vee}$.

Let $\SG_\beta$ be the root space of $\beta\in\ov{\Delta}^\pm$. Take
$u_{\beta}\in \SG_\beta$ and $u^{\beta}\in \SG_{-\beta}$ for $\beta
\in\ov{\Delta}^+$ such that $(u_{\alpha} \vert u^{\beta})_s
=\delta_{\alpha\beta}$. For any highest weight $\SG$-module $V$,
with weight space decomposition $V=\bigoplus_{\mu} V_\mu$, we define
$\overline{\Gamma}_1:V\rightarrow V$ to be the linear map that acts
as the scalar $(\mu+2\rho_s\vert\mu)_s$ on $V_\mu$. Let
$\overline{\Gamma}_2:=2\sum_{\beta\in\overline{\Delta}^+}u^{\beta}u_{\beta}$.
Define the Casimir operator to be
\begin{equation}\label{supercasimir}
\overline{\Omega}:=\overline{\Gamma}_1+\overline{\Gamma}_2.
\end{equation}
It is straightforward to check the following.

\begin{prop}
The operator $\overline{\Omega}$ commutes with the action of $\SG$
on a highest weight module $V$ with highest weight $\la$, and acts
on $V$ as the scalar $(\la+2\rho_s\vert\la)_s$.
\end{prop}

\subsection{The sets of weights $\mc{P}^+_{{\mf l}}$ and $\mc{P}^{++}_{{\mf l}}$ in ${\h}^*$}
\label{sec:Plc}


Recall the Levi subalgebra ${\mf l}$ of $\G$ with simple roots
indexed by $S$. Let $c\in\C$.

For $\G={\mf a}_{\infty}$, let $\mathcal P^+_{{\mf l},c}$ consist of
$\mu\in{\h}^*$ of the following form:
\begin{equation} \label{wt:ainf}
\mu=c\La^{\mf a}_0 + \sum_{i\geq 1}\eta_i\epsilon_i - \sum_{j\geq
0}\zeta_{j}\epsilon_{-j},
\end{equation}
where $\eta=(\eta_1,\eta_2,\ldots)$ and
$\zeta=(\zeta_0,\zeta_1,\ldots)\in \mc{P}^+$. (Note the shift of
index for $\zeta$.) We denote by $\mc{P}^{++}_{\mf l,c}$ the subset
of $\mc{P}^{+}_{\mf l,c}$ which consists of $\mu$ above with $\eta$
and $\zeta$ being $(n|m)$ and $(q|p)$-hook partitions, respectively.

For $\G=\mf{b}_{\infty},{\mf c}_{\infty},{\mf d}_{\infty}$, let
$\mathcal P^+_{{\mf l},c}$ consist of $\mu\in{\h}^*$ of the
following form:
\begin{equation} \label{wt:xinfty}
\mu=c\La^{\mathfrak x}_0+\sum_{i\geq 1}\mu_i\epsilon_i,
\end{equation}
where $(\mu_1, \mu_2, \ldots)\in\mc{P}^+$. Denote by
$\mc{P}^{++}_{\mf l,c}$ the subset of $\mc{P}^{+}_{\mf l,c}$ of the
form above with $\mu$ being an $(n|m)$-hook partition. We put
\begin{equation*}
\mc{P}^{+}_{\mf l}:=\bigsqcup_{c\in\C}\mc{P}^{+}_{\mf l,c}, \ \ \ \
\mc{P}^{++}_{\mf l}:=\bigsqcup_{c\in\C}\mc{P}^{++}_{\mf l,c}.
\end{equation*}

\subsection{The sets of weights $\mc{P}^{++}_{\ov{\mf l}}$ in $\ov{\h}^*$}


Let $\ov{\Delta}:=\ov{\Delta}^+\cup
\ov{\Delta}^{\,-}$ be the set of roots of $\SG$, where
$\ov{\Delta}^{\,-}=-\ov{\Delta}^+$. Put
\begin{equation*}
S:= \ov{I}\setminus \{0\}.
\end{equation*}
Let $\ov{\Delta}_S^\pm:=\ov{\Delta}^\pm\cap(\sum_{r\in S}\Z\beta_r)$
and $\ov{\Delta}^\pm(S):=\ov{\Delta}^\pm\setminus\ov{\Delta}^\pm_S$.
Let
\begin{equation*}\label{superparabolic}
\begin{aligned}
&\ov{\mf u}_{\pm} := \sum_{\beta\in \ov{\Delta}^\pm(S)}\SG_{\beta},
\quad \ov{\mf l}  := \sum_{\beta\in \ov{\Delta}_S^+ \cup
\ov{\Delta}_S^-}\SG_{\beta}\oplus\ov{\h}
\end{aligned}
\end{equation*}
so that $\SG=\ov{\mf{u}}_+\oplus\ov{\mf{l}}\oplus\ov{\mf{u}}_-$. If
$\SG=\hgl({p+m|q+n})$, then $\ov{\mf{l}}=\gl(p|q) \oplus
\gl(m|n)\oplus\C K$, and $\ov{\mf{l}}= \gl(m|n)\oplus\C K$,
otherwise. The Lie superalgebras $\ov{\mf{l}}$ and $\SG$ share the
same Cartan subalgebra $\ov\h$.

Let $c\in\C$. For $\SG=\hgl({p+m|q+n})$, let $\mathcal
P^{++}_{\ov{\mf l},c}$ consist of $\ov\mu\in\ov{\h}^*$ of the
following form:
\begin{equation*}
\begin{split}
\ov\mu =&\; c\widetilde{\La}^{\mf a}_0 +
\tau_1\varepsilon_{1}+\cdots+\tau_m\varepsilon_{m}+
\langle\tau'_{1}-m\rangle\delta_1+\cdots+\langle\tau'_{n}-m\rangle\delta_n
\\
&-\left(\xi_1\delta_{-1}+\cdots+\xi_{q}\delta_{-q}+
\langle\xi'_{1}-q\rangle\varepsilon_{-1}+\cdots+\langle\xi'_{p}-q\rangle\varepsilon_{-p}\right),
\end{split}
\end{equation*}
where $\tau=(\tau_1,\tau_2,\ldots)$ and $\xi=(\xi_1,\xi_2,\ldots)$
are $(m|n)$- and $(q|p)$-hook partitions, respectively.

When $\SG$ is $\hspo(2m|2n+1)$, $\hosp(2m|2n)$ or $\hspo(2m|2n)$, we
let $\mathcal P^{++}_{\ov{\mf l},c}$ consist of $\ov\mu\in\ov{\h}^*$
of the following form:
\begin{equation*}
\begin{split}
\ov\mu&= c\widetilde{\La}^{\mathfrak x}_0+
\tau_1\varepsilon_{1}+\cdots+\tau_m\varepsilon_{m}+
\langle\tau'_{1}-m\rangle\delta_1+\cdots+\langle\tau'_{n}-m\rangle\delta_n,
\end{split}
\end{equation*}
where $\tau=(\tau_1,\tau_2,\ldots)$ runs over $(m|n)$-hook
partitions. We put
\begin{equation*}
\mc{P}^{++}_{\ov{\mf
l}}:=\bigsqcup_{c\in\C}\mc{P}^{++}_{\ov{\mf l},c}.
\end{equation*}

\subsection{Casimir eigenvalues of $\hgl(p+m|q+n)$ versus those of
$\widehat{\gl}_\infty$}

Let $\SG=\hgl(p+m|q+n)$. Let $ \mu \in \mc{P}^{++}_{\mf l,c}$ be
as in \eqref{wt:ainf}. Set $\tau=\eta'$,
$\nu=(\nu_1,\nu_2,\ldots):=(\tau_{m+1},\tau_{m+2},\ldots)$ and
$\chi=(\chi_1,\chi_2,\ldots):=(\zeta_q,\zeta_{q+1},\ldots)$.
Define a bijection
$$
\vartheta: \mc{P}^{++}_{\mf l,c}\rightarrow \mc{P}^{++}_{\ov{\mf
l},c}
$$
by
\begin{equation*}
\begin{split}
\vartheta(\mu):=& \;  c\widetilde{\La}^{\mf a}_0 +
\tau_1\varepsilon_{1}+\cdots+\tau_m\varepsilon_{m}
+\nu'_{1}\delta_1+\cdots+\nu'_{n}\delta_n
\\
&-\left(\zeta_0\delta_{-1}+\cdots+\zeta_{q-1}\delta_{-q}
+\chi'_{1}\varepsilon_{-1}+\cdots+\chi'_{p}\varepsilon_{-p}\right).
\end{split}
\end{equation*}
Note that
\begin{equation} \label{wt:identify}
\vartheta(\La^{\mf a}(\la))=\widehat{\La}^{\mf a}_f(\la), \quad
\text{ for }  \la\in\mc{P}({\rm GL}(d))\ \text{with
$\la_{m+1}\leq n$ and $\la_{d-p}\ge -q$}.
\end{equation}

For $\eta$ and $\zeta$ as above it is convenient to introduce the
following symbols:
\begin{align*}
&(\eta+2\rho_1|\eta)_1:=\sum_{i\ge 1}\eta_i(\eta_i-2i),\\
&(\zeta+2\rho_2|\zeta)_2:=\sum_{j\ge 0}\zeta_j(\zeta_j-2j).
\end{align*}

By \cite[Lemma 7.1]{CK} or \cite[(1.7)]{Mac}, for
$\la=(\la_1,\la_2,\ldots)\in\mc{P}^+$, we have
\begin{equation} \label{macd}
\sum_{i\geq 1}\la_i(\la_i-2i)=-\sum_{i\geq
1}\la'_i(\la_i'-2(i-1)).
\end{equation}

\begin{lem}\label{aux113}
For $\mu\in\mc{P}^{++}_{\mf l,c}$, we have
\begin{equation} \label{casimir=}
(\vartheta(\mu)+2\rho_s|\vartheta(\mu))_s=(\mu+2\rho_c|\mu)_c+\ov{C},
\end{equation}
where
$$
\ov{C} =c^2(\widetilde{\La}^{\mf a}_0|\widetilde{\La}^{\mf
a}_0)_s+2c(\rho_s|\widetilde{\La}^{\mf a}_0)_s.
$$
In particular, for $\la\in \mc{P}({\rm GL}(d))$ with
$\la_{m+1}\leq n$ and $\la_{d-p}\ge -q$, $(\La^{\mf
a}(\la)+2\rho_c|\La^{\mf a}(\la))_c =(\mu+2\rho_c|\mu)_c$ if and
only if $(\widehat{\La}^{\mf a}_f(\la) +2\rho_s|\widehat{\La}^{\mf
a}_f(\la))_s=(\vartheta(\mu)+2\rho_s|\vartheta(\mu))_s$.
\end{lem}
\begin{proof}{
Let $ \mu \in \mc{P}^{++}_{\mf l,c}$ be as in \eqref{wt:ainf}.
Using \eqref{macd}, we have
\begin{equation*}
\begin{split}
(\mu+2\rho_c|\mu)_c&=(\eta+2\rho_1|\eta)_1+(\zeta+2\rho_2|\zeta)_2-c(|\eta|+|\zeta|)  \\
&=-(\tau+2\rho_2|\tau)_2+(\zeta+2\rho_2|\zeta)_2-c(|\eta|+|\zeta|).
\end{split}
\end{equation*}

For convenience, put $\eta=\sum_{i\geq 1}\eta_i\epsilon_i$ and
$\zeta=\sum_{j\geq 0}\zeta_j\epsilon_{-j}$. Now, we have
\begin{equation*}
\begin{split}
&(\vartheta(\mu)+2\rho_s|\vartheta(\mu))_s =(c\widetilde{\La}^{\mf
a}_0 +\vartheta(\eta) -\vartheta(\zeta)+2\rho_s
|c\widetilde{\La}^{\mf a}_0 +\vartheta(\eta)-\vartheta(\zeta))_s
\\
&=c^2(\widetilde{\La}^{\mf a}_0|\widetilde{\La}^{\mf a}_0)_s
+2c(\rho_s|\widetilde{\La}^{\mf a}_0)_s
+(\vartheta(\eta)-\vartheta(\zeta)+2\rho_s
|\vartheta(\eta)-\vartheta(\zeta))_s
+2c(\widetilde{\La}^{\mf a}_0|\vartheta(\eta)-\vartheta(\zeta))_s  \\
&=\ov{C}+(\vartheta(\eta)+2\rho_s|\vartheta(\eta))_s
+(\vartheta(\zeta)-2\rho_s|\vartheta(\zeta))_s -c(|\eta|+|\zeta|).
\end{split}
\end{equation*}

First, we have
\begin{align}\label{vartheta1}
(\vartheta(\eta)&
+2\rho_s|\vartheta(\eta))_s\allowdisplaybreaks\nonumber
\\
=&-\left[ \tau_1(\tau_1-2\cdot 0)+\cdots+\tau_m(\tau_m-2(m-1))
\right]\nonumber\allowdisplaybreaks  \\
&+ \left[ \nu'_{1}(\nu'_{1}+2(m-1))+\cdots+\nu'_{n}(\nu'_{n}+2(m-n))
\right]\nonumber\allowdisplaybreaks  \\
=&-\left[ \tau_1(\tau_1-2\cdot 0)+\cdots+\tau_m(\tau_m-2(m-1))
\right]\nonumber\allowdisplaybreaks  \\
&+ \left[ \nu'_{1}(\nu'_{1}-2\cdot1)+\cdots+\nu'_{n}(\nu'_{n}-2n)
\right] +2m \sum_{k=1}^n \nu'_{k}\nonumber\allowdisplaybreaks   \\
=&-\left[ \tau_1(\tau_1-2\cdot 0)+\cdots+\tau_m(\tau_m-2(m-1))
\right] \allowdisplaybreaks \\
&- \left[ \nu_{1}(\nu_{1}-2\cdot 0)+\nu_{2}(\nu_{2}-2\cdot
1)+\cdots)\right] +2m \sum_{k\geq 1}
\nu_{k}\nonumber\allowdisplaybreaks\\
=&-\left[ \tau_1(\tau_1-2\cdot 0)+\cdots+\tau_m(\tau_m-2(m-1))
\right]\nonumber\allowdisplaybreaks \\
&- \left[ \tau_{m+1}(\tau_{m+1}-2\cdot
0)+\tau_{m+2}(\tau_{m+2}-2\cdot 1)+\cdots\right] +2m \sum_{k\geq 1}
\tau_{m+k}\nonumber\allowdisplaybreaks \\
=&-(\tau+2\rho_2|\tau)_2\nonumber\allowdisplaybreaks.
\end{align}

Similarly, we have
\begin{align*}
 (\vartheta(\zeta)& -2\rho_s|\vartheta(\zeta))_s \allowdisplaybreaks
\\
=&\left[ \zeta_0(\zeta_0-2\cdot
0)+\cdots+\zeta_{q-1}(\zeta_{q-1}-2(q-1))
\right]\allowdisplaybreaks  \\
&- \left[
\chi'_{1}(\chi'_{1}+2(q-1))+\cdots+\chi'_{p}(\chi'_{p}+2(q-p))
\right]\allowdisplaybreaks \\
=&\left[ \zeta_0(\zeta_0-2\cdot
0)+\cdots+\zeta_{q-1}(\zeta_{q-1}-2(q-1))
\right]\allowdisplaybreaks  \\
&- \left[ \chi'_{1}(\chi'_{1}-2\cdot
1))+\cdots+\chi'_{p}(\chi'_{p}-2\cdot p)
\right] -2q \sum_{k=1}^{p}\chi'_{k}\allowdisplaybreaks  \\
=&\left[ \zeta_0(\zeta_0-2\cdot
0)+\cdots+\zeta_{q-1}(\zeta_{q-1}-2(q-1))
\right]\allowdisplaybreaks  \\
&+ \left[ \chi_{1}(\chi_{1}-2\cdot 0))+\chi_{2}(\chi_{2}-2\cdot
1)+\cdots
\right] -2q \sum_{k\geq 1}\chi_{k}\allowdisplaybreaks  \\
=&\left[ \zeta_0(\zeta_0-2\cdot
0)+\cdots+\zeta_{q-1}(\zeta_{q-1}-2(q-1))
\right]\allowdisplaybreaks  \\
&+ \left[ \zeta_{q}(\zeta_{q}-2\cdot
0))+\zeta_{q+1}(\zeta_{q+1}-2\cdot 1)+\cdots
\right] -2q \sum_{k\geq 0}\zeta_{q+k}\allowdisplaybreaks  \\
=& (\zeta+2\rho_2|\zeta)_2.
\end{align*}
This completes the proof of \eqref{casimir=}. The remaining part
of the lemma follows from the definition of the bijection
$\vartheta$, \eqref{wt:identify}, and \eqref{casimir=}.}
\end{proof}

\subsection{Casimir eigenvalues of $\hosp$  and $\hspo$ versus those of $\mf{x}_\infty$}

Suppose that $(\SG, G)$ is $( \hspo(2m|2n+1), {\rm Pin}(d))$,
$(\hosp(2m|2n), {\rm Sp}(d) )$ or $( \hspo(2m|2n), {\rm O}(d))$. Let
$\mu \in \mc{P}^{++}_{\mf l,c}$ be as in \eqref{wt:xinfty} with
${\mathfrak x}\in\{\mf{b},{\mf c},{\mf d}\}$ and $c\in\C$. Then
$\nu:=(\nu_1,\nu_2,\ldots)=(\mu_1,\mu_2,\ldots)'$ is of $(m|n)$-hook
shape. Set
$\tau:=(\tau_1,\tau_2,\ldots):=(\nu_{m+1},\nu_{m+2},\ldots)$. We
define a bijection $\vartheta:\mc{P}^{++}_{\mf l,c}\rightarrow
\mc{P}^{++}_{\ov{\mf l},\ov{c}}$ by letting
\begin{equation*}
\vartheta(\mu):=\ov{c}\widetilde{\La}^{\mathfrak x}_0+
\nu_1\varepsilon_1 +\cdots+\nu_m\varepsilon_m + \tau'_{1}\delta_1
+\cdots +\tau'_{n}\delta_n,
\end{equation*}
where $\ov{c}={c}{\langle\La^{\mathfrak x}_0,K \rangle}$. Note
that
$$
\vartheta(\La^{\mathfrak x}(\la))=\widehat{\La}^{\mathfrak
x}_f(\la), \quad \text{ for } \la\in\mc{P}^{\SG}(G).
$$

\begin{lem}\label{aux113-cd}
For $\mu\in\mc{P}^{++}_{\mf l,c}$, we have
$$(\vartheta(\mu)+2\rho_s|\vartheta(\mu))_s=-(\mu+2\rho_c|\mu)_c+ \ov C,$$
where $\ov C =\ov{c}^2(\widetilde{\La}^{\mathfrak
x}_0|\widetilde{\La}^{\mathfrak
x}_0)_s+2\ov{c}(\rho|\widetilde{\La}^{\mathfrak x}_0)_s$. In
particular, $(\La^{\mathfrak x}(\la)+2\rho_c |\La^{\mathfrak
x}(\la))_c =(\mu+2\rho_c|\mu)_c$ if and only if
$(\widehat{\La}^{\mathfrak
x}_f(\la)+2\rho_s|\widehat{\La}^{\mathfrak
x}_f(\la))_s=(\vartheta(\mu)+2\rho_s|\vartheta(\mu))_s$, for $\la\in
\mc{P}^{\SG}(G)$.
\end{lem}
\begin{proof}
The equivalence of the identities follows easily once we establish
the first identity. Let $\mu=c\La^{\mathfrak x}_0+\sum_{i\geq
1}\mu_i\epsilon_i \in \mc{P}^{++}_{\mf l,c}$, and so
$\mu^{\circ}:=(\mu_1,\mu_2, \ldots)$ is of $(n|m)$-hook shape.

{\it Case 1}.  $\SG=\hspo (2m|2n+1)$. Let $(\cdot\vert\cdot)_c$
stand for the bilinear form on the dual Cartan subalgebra of
$\mf{b}_\infty$. We compute
\begin{align*}
(\mu+2\rho_c|\mu)_c&=(\sum_{i\ge
1}\mu_i\epsilon_i+2\rho_c|\sum_{i\ge 1}\mu_i\epsilon_i)_c
+ 2c(\sum_{i\ge 1}\mu_i\epsilon_i|\La_0^{\mf b})_c\\
&=\sum_{i\ge
1}\mu_i(\mu_i-2i+1)-c|\mu^\circ|=(\mu^\circ+2\rho_1|\mu^\circ)_1-(c-1)|\mu^\circ|.
\end{align*}
On the other hand we have
\begin{align*}
(\vartheta(\mu)+2\rho_s&|\vartheta(\mu))_s= \ov{C} +
c(\sum_{i=1}^m\nu_i\varepsilon_i+\sum_{j=1}^n\tau'_j\delta_j|\widetilde{\La}^{\mf
b}_0)_s\allowdisplaybreaks \\
&+(\sum_{i=1}^m\nu_i\varepsilon_i+\sum_{j=1}^n\tau'_j\delta_j+2\rho_s|
\sum_{i=1}^m\nu_i\varepsilon_i+\sum_{j=1}^n\tau'_j\delta_j)_s\allowdisplaybreaks \\
&=\ov{C} + c|\mu^\circ|\allowdisplaybreaks \\
&+(\sum_{i=1}^m(\nu_i-2i+1)\varepsilon_i
+\sum_{j=1}^n(\tau'_j+2m-2i+1)\delta_j|\sum_{i=1}^m\nu_i\varepsilon_i+\sum_{j=1}^n\tau'_j\delta_j)_s\allowdisplaybreaks \\
&=\ov{C}+c|\mu^\circ| + \sum_{i=1}^m\nu_i(\nu_i-2i+1) -
\sum_{j=1}^n\tau'_j(\tau'_j+2m-2j+1)\allowdisplaybreaks \\
&=\ov{C}+(c-1)|\mu^\circ| + \sum_{i=1}^m\nu_i(\nu_i-2(i-1)) +
\sum_{j=1}^n \tau_j(\tau_j-2(m+j-1))\allowdisplaybreaks \\
&=\ov{C}+(c-1)|\mu^\circ| + \sum_{i=1}^m\nu_i(\nu_i-2(i-1)) +
\sum_{j=1}^n \nu_{m+j}(\nu_{m+j}-2(m+j-1))\allowdisplaybreaks \\
&=\ov{C}+(c-1)|\mu^\circ| - \sum_{i\ge
1}\mu_i(\mu_i-2i)|=\ov{C}+(c-1)|\mu^\circ| -
(\mu^\circ+2\rho_1|\mu^\circ)_1.
\end{align*}
Above we have used \eqnref{macd} in the third to last identity.

{\it Case 2}. $\SG=\hosp(2m|2n)$. By \eqref{macd}, we have
\begin{equation*}
(\mu+2\rho_c\vert\mu)_c =\sum_{i\geq 1}\mu_i(\mu_i-2i)
-2c|\mu^\circ|=-(\nu+2\rho_2|\nu)_2-2c|\mu^\circ|,
\end{equation*}
where $(\cdot|\cdot)_c$ is the bilinear form on the dual Cartan
subalgebra of $\mf{c}_{\infty}$.

On the other hand,
\begin{align*}
&(\vartheta(\mu)+2\rho_s|\vartheta(\mu))_s\allowdisplaybreaks\\
&=c^2(\widetilde{\La}^{\mf c}_0|\widetilde{\La}^{\mf
c}_0)_s+2c(\rho_s|\widetilde{\La}^{\mf
c}_0)_s+(\vartheta(\mu)+2\rho_s|\vartheta(\mu))_s
+2c(\widetilde{\La}^{\mf c}_0|\vartheta(\mu))_s\allowdisplaybreaks\\
&=\ov C +(\vartheta(\mu)+2\rho_s|\vartheta(\mu))_s +2c|\mu^\circ|.
\end{align*}
And by the same argument as in \eqref{vartheta1}, we can show that
\begin{equation*}
(\vartheta(\mu)+2\rho_s|\vartheta(\mu))_s=(\nu+2\rho_2|\nu)_2.
\end{equation*}

{\it Case 3}. $\SG=\hspo (2m|2n)$. By \eqref{macd}, we have
\begin{equation*}
(\mu+2\rho_c\vert\mu)_c =\sum_{i\geq 1}\mu_i(\mu_i-2(i-1))
-c|\mu^\circ|=-(\nu+2\rho_1|\nu)_1-c|\mu^\circ|,
\end{equation*}
where $(\cdot|\cdot)_c$ is the bilinear form on the dual Cartan
subalgebra of $\mf{d}_{\infty}$.

On the other hand,
\begin{align*}
&(\vartheta(\mu)+2\rho_s|\vartheta(\mu))_s\allowdisplaybreaks\\
&=(c/2)^2(\widetilde{\La}^{\mf d}_0|\widetilde{\La}^{\mf
d}_0)_s+c(\rho_s|\widetilde{\La}^{\mf
d}_0)_s+(\vartheta(\mu)+2\rho_s|\vartheta(\mu))_s
+c(\widetilde{\La}^{\mf d}_0|\vartheta(\mu))_s\allowdisplaybreaks\\
&= \ov C +(\vartheta(\mu)+2\rho_s|\vartheta(\mu))_s +c|\mu^\circ|.
\end{align*}
Similarly, arguing as  in \eqref{vartheta1}, we have
\begin{equation*}
(\vartheta(\mu)+2\rho_s|\vartheta(\mu))_s=(\nu+2\rho_1|\nu)_1.
\end{equation*}\vskip 2mm
This completes the proof of the first identity in the lemma.
\end{proof}

\section{Kostant homology formulas for Lie superalgebras}\label{sec:homology}

\subsection{The ${\U}_-$-homology groups of $\G$-modules}

Recall the dual pair $(\G,G)$ of type ${\mathfrak
x}\in\{\mf{a,b,c,d}\}$ in \secref{classical:dualpairs}, the Levi
subalgebra $\mf{l}$ of $\G$, and $\mc{P}^+_{\mf l}$, the set of
dominant weights for $\mf{l}$. We denote by $L({\mf l},\mu)$ the
irreducible highest weight $\mf{l}$-module with highest weight
$\mu\in\mf{h}^*$.
The following results for integrable modules of Kac-Moody
algebras apply to our setting.

\begin{prop}\label{eigenvalue:classical}
{\rm (}cf.~\cite{GL}, \cite[Prop.~3.1]{J}, \cite[Prop.~18 and
Lemma~20]{L}{\rm )} Let $\la\in\mc{P}(G)$.
\begin{itemize}
\item[(1)] If a weight $\eta\in\mathcal P^+_{\mathfrak l}$ with
$(\eta+2\rho_c\vert\eta)_c=(\La^{\mathfrak
x}(\la)+2\rho_c\vert\La^{\mathfrak x}(\la))_c$ appears in
$\Lambda^k\mathfrak u_-\otimes L(\G,\La^{\mathfrak x}(\la))$, then
$\eta=w\circ\La^{\mathfrak x}(\la)$ for some $w\in W^0_k(\mathfrak
x)$ and $\eta$ appears with multiplicity one.

\item[(2)] The $\mathfrak l$-module ${\rm H}_k({\mathfrak
u}_-;L(\G,\La^{\mathfrak x}(\la)))$ is completely reducible.
Moreover, if $L({\mathfrak l},\eta)$ is a summand of ${\rm
H}_k({\mathfrak u}_-;L(\G,\La^{\mathfrak x}(\la)))$, then
$(\eta+2\rho_c\vert\eta)_c=(\La^{\mathfrak
x}(\la)+2\rho_c\vert\La^{\mathfrak x}(\la))_c$.
\end{itemize}
\end{prop}
For $\la\in \mc{P}(G)$ and $w\in W^0$, we put
\begin{equation}\label{def:L}
\begin{split}
\tL(\G,\La^{\mathfrak x}(\la))&:=
\begin{cases}
L(\G,\La^{\mf d}(\la))\oplus L(\G,\La^{\mf d}(\tilde{\la})), & \text{if $G={\rm O}(2\ell)$}, \\
L(\G,\La^{\mathfrak x}(\la)), & \text{otherwise},
\end{cases}\\
\tL(\mf{l},w\circ\La^{\mathfrak x}(\la))&:=
\begin{cases}
L(\mf{l},w\circ\La^{\mf d}(\la))\oplus L(\mf{l},w\circ\La^{\mf d}(\tilde{\la})), & \text{if $G={\rm O}(2\ell)$}, \\
L(\mf{l},w\circ\La^{\mathfrak x}(\la)), & \text{otherwise}.
\end{cases}
\end{split}
\end{equation}

\subsection{The $\ov{\U}_-$-homology groups of $\SG$-modules}

Recall the dual pair $(\SG,G)$ of type ${\mathfrak
x}\in\{\mf{a,b,c,d}\}$ given in \secref{dualpairs:finite}. For
$\mu\in\ov{\h}^*$ we denote by ${L}(\ov{\mf l},\mu)$ the irreducible
highest weight $\ov{\mf l}$-module with highest weight $\mu$.

\begin{lem}\label{aux411}
Let $\la\in\mc{P}^{\SG}(G)$. The $\overline{\mathfrak l}$-module
$L(\SG,\widehat{\La}_f^{\mathfrak x}(\la))$ is completely reducible.
\end{lem}

\begin{proof}
Suppose that $\SG=\hgl(p+m|q+n)$. By \cite[Theorem 3.2]{CW1} (also
cf. \cite{S2}), $S(\C^{p|q*}\otimes\C^{d*})$ and $S(\C^{m|n}\otimes
\C^d)$ are completely reducible over $\gl(p|q)$ and $\gl(m|n)$,
respectively.  Indeed $S(\C^{m|n}\otimes \C^d)$ lies in the
semisimple tensor category $\mc{O}_{m|n}^{++}$ of finite dimensional
$\gl(m|n)$-modules with composition factors of the form
$L(\gl(m|n),\la^\natural)$, for some $\la\in\mc{P}^+$ with
$\la_{m+1}\le n$ (\cite[Theorem 3.1]{CK}). Also,
$S(\C^{p|q*}\otimes\C^{d*})$ lies in a similar semisimple tensor
category $\mc{O}_{p|q}^{++*}$. Thus
$S(\C^{p|q*}\otimes\C^{d*}\oplus\C^{m|n}\otimes \C^d)$, as a
$\gl(m|n) \oplus \gl(p|q)$-module, lies in
``$\mc{O}_{m|n}^{++}\otimes\mc{O}_{p|q}^{++*}$'', and hence is
completely reducible over $\ov{\mf l}=\gl(m|n) \oplus
\gl(p|q)\oplus\C K$. Therefore $L(\SG,\widehat{\La}_f^{\mf a}(\la))$
is completely reducible over $\ov{\mf l}$.

If $\SG$ is of type $\mf{b, c, d}$, then
$\ov{\mf{l}}=\gl(m|n)\oplus\C K$. Thus
$L(\SG,\widehat{\La}_f^{\mathfrak x}(\la))$ is completely reducible
over $\ov{\mf{l}}$, since it is a submodule of $S(\C^{m|n}\otimes
\C^d)$ or $S(\C^{m|n}\otimes\C^{d})\otimes
\Lambda(\C^{\frac{d}{2}})$.
\end{proof}

Now consider the homology groups of Lie superalgebras ${\rm
H}_k(\overline{\mathfrak u}_-;L(\SG,\widehat{\La}_f^{\mathfrak
x}(\la)))$, which are defined analogously (see e.g.~\cite{KK}).

\begin{lem}\label{complete:reducibility}
Let $\la\in\mc{P}^{\SG}(G)$. The $\overline{\mathfrak l}$-module
${\rm H}_k(\overline{\mathfrak u}_-;L(\SG,\widehat{\La}_f^{\mathfrak
x}(\la)))$ is completely reducible.
\end{lem}

\begin{proof}
Suppose that $\SG=\hgl(p+m|q+n)\supseteq \ov{\mathfrak l} =\gl(m|n)
\oplus \gl(p|q)\oplus\C K$. It follows from \cite[Theorem 3.3]{CW1}
that $\La^k(\bar{\mathfrak u}_-)$ lies in
``$\mc{O}_{m|n}^{++}\otimes\mc{O}_{p|q}^{++*}$''. By \lemref{aux411}
$L(\G,\widehat{\La}_f^{\mf a}(\la))$, as an $\gl(m|n) \oplus
\gl(p|q)$-module, lies in
``$\mc{O}_{m|n}^{++}\otimes\mc{O}_{p|q}^{++*}$'', and hence
$\Lambda^k(\bar{\mathfrak u}_-)\otimes L(\SG,\widehat{\La}_f^{\mf
a}(\la))$ is a completely reducible $\bar{\mathfrak l}$-module.
Since any subquotient of a completely reducible module is also
completely reducible, the result follows.

The ortho-symplectic cases are easier and hence omitted.
\end{proof}

The proof of \lemref{complete:reducibility} implies that if
$L(\ov{\mf l},\gamma)$ is an $\ov{\mf l}$-submodule of ${\rm
H}_k(\overline{\mathfrak u}_-;L(\SG,\widehat{\La}_f^{\mf x}(\la)))$,
then $\gamma\in\mc{P}^{++}_{\ov{\mf l}}$. We have the following
super-analogue of the action of the Casimir operator on homology
groups (see~\propref{eigenvalue:classical}).

\begin{prop}\label{eigenvalue:super}
Let $\gamma\in\mathcal P^{++}_{\bar{\mathfrak l}}$.  If
$L(\bar{\mathfrak l},\gamma)$ is a summand of ${\rm
H}_k(\overline{\mathfrak u}_-;L(\SG,\widehat{\La}_f^{\mf x}(\la)))$,
then $(\gamma+2\rho_s\vert\gamma)_s=(\widehat{\La}_f^{\mf
x}(\la)+2\rho_s\vert\widehat{\La}_f^{\mf x}(\la))_s$.
\end{prop}

\begin{proof}
The proof follows the same type of arguments as for
\cite[Proposition~18]{L} and thus will be omitted. We only remark
that in the process we use the same bilinear form
$(\cdot\vert\cdot)_s$ and the same $\rho_s$ to define the
corresponding Casimir operator for $\ov{\mf{l}}$ as in
\eqnref{supercasimir}.
\end{proof}

\subsection{Formulas for the $\overline{\mathfrak u}_-$-homology groups of $\SG$-modules}

Recalling $\vartheta(\La^{\mathfrak x}(\la))
=\widehat{\La}_f^{\mathfrak x}(\la)$ for $\la\in\mc{P}^{\SG}(G)$, we
define $\tL(\SG,\widehat{\La}_f^{\mathfrak x}(\la))$ and
$\tL(\ov{\mf{l}},\vartheta(w\circ\La^{\mathfrak x}(\la)))$ similarly
as in \eqref{def:L}.

\begin{lem}\label{aux112}
Let $\la\in\mathcal P^{\SG}(G)$ and $\bar{\mu}\in{\mathcal
P}^{++}_{\bar{\mathfrak l}}$. If $L(\bar{\mathfrak l},\bar{\mu})$
appears in the decomposition of $\Lambda^k\bar{{\mathfrak
u}}_-\otimes \tL (\SG,\widehat{\La}^{\mathfrak x}_f(\la))$ with
multiplicity $m_{\ov\mu}$, then there exists a unique
$\mu\in\mathcal P^{++}_{\mathfrak l}$ with
$\vartheta(\mu)=\bar{\mu}$, and $L(\mathfrak l,\mu)$ appears in the
decomposition of $\Lambda^k\mathfrak u_-\otimes\tL
(\G,\La^{\mathfrak x}(\la))$ with the same multiplicity
$m_{\ov\mu}$.
\end{lem}

\begin{proof} Since $\omega^{\mathfrak x}\left({\rm
ch}\tL(\G,\La^{\mathfrak x}(\la))\right)={\rm
ch}\tL(\SG,\widehat{\La}^{\mathfrak x}_f(\la))$, for
$\la\in\mc{P}^{\SG}(G)$, and $\omega^{\mathfrak x}\left({\rm
ch}\La^k({\mf u}_-)\right)={\rm ch}\La^k(\ov{\mf u}_-)$, we conclude
that
\begin{equation*}
\omega^{\mathfrak x}\left({\rm ch}\big[\La^k({\mf u}_-)\otimes
\tL(\G,\La^{\mathfrak x}(\la))\big]\right)={\rm
ch}\big[\La^k(\ov{\mf u}_-)\otimes
\tL(\SG,\widehat{\La}^{\mathfrak x}_f(\la))\big].
\end{equation*}
Since $\vartheta$ is a bijection on $\mc{P}^{++}_{\mf l}$, there
exists a unique $\mu\in\mc{P}^{++}_{\mf l}$ such that
$\vartheta(\mu)=\ov{\mu}$. Therefore $L(\mathfrak l,\mu)$ is a
composition factor of
 $\Lambda^k\mathfrak u_-\otimes
\tL(\G,\La^{\mf a}(\la))$, if $L(\bar{\mathfrak l},\vartheta(\mu))$
is a composition factor of $\Lambda^k\bar{{\mathfrak u}}_-\otimes
\tL(\SG,\widehat\La^{\mathfrak x}(\la))$. Furthermore they have the
same multiplicity.
\end{proof}

\begin{lem}\label{aux222}
Let $\la\in\mathcal P^{\SG}(G)$ and $\eta\in\mc{P}^{++}_{\ov{\mf
l}}$ such that  ${L}(\ov{\mf{l}},\eta)$ is a summand of
$\La^k(\SU_-)\otimes\tL(\SG,\widehat{\La}_f^{\mathfrak x}(\la))$.
Then
$(\eta+2\rho_s\vert\eta)_s=(\widehat{\La}^{\mathfrak
x}_f(\la)+2\rho_s\vert\widehat{\La}^{\mathfrak x}_f(\la))_s$ if and
only if there exists $w\in W^0_k(\mf{x})$ with
$\eta=\vartheta(w\circ\La^{\mathfrak x}(\la))$.
\end{lem}

\begin{proof}
Assume that there exists $w\in W^0_k({\mathfrak x})$ with
$w\circ\La^{\mathfrak x}(\la)\in\mc{P}^{++}_{\mf l}$. Set
$\eta=\vartheta(w\circ\La^{\mathfrak x}(\la))$. Then by
\lemref{aux113} and \lemref{aux113-cd} and the $W$-invariance of
$(\cdot\vert\cdot)_c$ we have
\begin{align*}
(\eta+2\rho_s\vert \eta)_s& =\pm(w\circ\La^{\mathfrak
x}(\la)+2\rho_c\vert
w\circ\La^{\mathfrak x}(\la))_c+\ov{C} \\
&= \pm(\La^{\mathfrak x}(\la)+2\rho_c\vert \La^{\mathfrak
x}(\la))_c+\ov{C}= (\widehat{\La}^{\mathfrak x}_f(\la)+2\rho_s\vert
\widehat{\La}^{\mathfrak x}_f(\la))_s.
\end{align*}
On the other hand suppose that
$(\eta+2\rho_s\vert\eta)_s=(\widehat{\La}^{\mathfrak
x}_f(\la)+2\rho_s\vert\widehat{\La}^{\mathfrak x}_f(\la))_s$. By
\lemref{aux112} there exists a unique
$\vartheta^{-1}(\eta)\in\mc{P}^{++}_{\mf l}$ with
$\vartheta(\vartheta^{-1}(\eta))=\eta$ satisfying the conditions of
\lemref{aux112}. By \lemref{aux113} and \lemref{aux113-cd}, we have
\begin{equation*}
(\vartheta^{-1}(\eta)+2\rho_c\vert\vartheta^{-1}(\eta))_c=(\La^{\mathfrak
x}(\la)+2\rho_c\vert\La^{\mathfrak x}(\la))_c.
\end{equation*}
By \propref{eigenvalue:classical} (1) there exists $w\in
W^0_k(\mathfrak x)$ such that
$\vartheta^{-1}(\eta)=w\circ\La^{\mathfrak x}(\la)$.
\end{proof}

Let $W^0_k(\mathfrak x)'$ denote the subset of $W^0_k(\mathfrak x)$
consisting of those $w$ with $\vartheta(w\circ\La^{\mathfrak
x}(\la)) \in \mc{P}^{++}_{\ov{\mf{l}}}$.  The proof of the following
theorem is inspired by \cite{A}.

\begin{thm}\label{mainthm}
Let $\la\in\mathcal P^{\SG}(G)$, $k\in\Z_+$, and $\vartheta$ be as
before. As $\ov{\mf{l}}$-modules we have
\begin{equation*}
{\rm H}_k(\ov{\mf{u}}_-;\tL(\SG,\widehat{\La}^{\mathfrak
x}_f(\la)))\cong\bigoplus_{w\in W^0_k(\mathfrak x)'}
\tL(\ov{\mf{l}},\vartheta(w\circ\La^{\mathfrak x}(\la))).
\end{equation*}
In particular, ${\rm ch}\big{[}{\rm H}_k(\bar{\mathfrak
u}_-;\tL(\SG,\widehat{\La}^{\mathfrak
x}_f(\la)))\big{]}=\omega^{\mathfrak x}\big{(}{\rm ch}\big{[}{\rm
H}_k({\mathfrak u}_-;\tL(\G,\La^{\mathfrak x}(\la))\big{]}\big{)}$.
\end{thm}

\begin{proof}
Let $\mu\in\mathcal P^+_{\mathfrak l}$ be such that $L(\mathfrak
l,\mu)$ is a summand of ${\rm H}_k({\mathfrak
u}_-;\tL(\G,\La^{\mathfrak x}(\la)))$. Then it is precisely a
summand of $\Lambda^k\mathfrak u_-\otimes \tL(\G,\La^{\mathfrak
x}(\la))$ with $(\mu+2\rho_c\vert\mu)_c=(\La^{\mathfrak
x}(\la)+2\rho_c\vert\La^{\mathfrak x}(\la))_c$ by
\propref{eigenvalue:classical}. Furthermore each appears with
multiplicity one or two (cf. Remark 5.2 \cite{CK}). By
\lemref{aux222} the corresponding $\vartheta(\mu)$'s, for
$\mu\in\mc{P}^{++}_{\mf l}$, are precisely the weights in
$\mc{P}^{++}_{\ov{\mf l}}$ such that $L(\ov{\mf{l}},\vartheta(\mu))$
appear as summands of $\Lambda^k\bar{{\mathfrak u}}_-\otimes
\tL(\SG,\widehat{\La}_f^{\mathfrak x}(\la))$ with
$(\vartheta(\mu)+2\rho_s\vert\vartheta(\mu))_s
=(\widehat{\La}^{\mathfrak
x}_f(\la)+2\rho_s\vert\widehat{\La}^{\mathfrak x}_f(\la))_s$;
moreover it appears with the same multiplicity.

Theorems \ref{superchar:A}, \ref{superchar:OSP}, and
\ref{superchar:SPO} together with the Euler-Poincar\'e principle
imply that
\begin{align*}
&\sum_{k=0}^\infty(-1)^k{\rm ch}\big{[}{\rm H}_k(\bar{\mathfrak
u}_-;\tL(\SG,\widehat{\La}^{\mathfrak x}_f(\la)))\big{]}
=\sum_{k=0}^\infty(-1)^k\sum_{w\in W^0_k(\mathfrak x)'}
\mc{H}_{\la,w},
\end{align*}
where
\begin{align*}
\mc{H}_{\la,w}  &=\begin{cases} hs_{\la^+_w}(\xi,y)
hs_{\la^-_w}(\eta^{-1},x^{-1}),
& \text{if $G={\rm GL}(d)$}, \\
hs_{\la_w}(\eta,x),
& \text{if $G={\rm Sp}(d)$, ${\rm O}(2\ell+1)$, ${\rm Pin}(d)$ },\\
hs_{\la_w}(\eta,x) +hs_{\tilde{\la}_w}(\eta,x), & \text{if $G={\rm
O}(2\ell)$}.
\end{cases}
\\
&={\rm ch}\tL(\ov{\mf{l}},\vartheta(w\circ\La^{\mathfrak
x}(\la))).
\end{align*}
Since all the highest weights are distinct, we conclude from
Proposition \ref{eigenvalue:super} that
\begin{align*}
{\rm ch}\big{[}{\rm H}_k(\bar{\mathfrak
u}_-;\tL(\SG,\widehat{\La}^{\mathfrak x}_f(\la)))\big{]} =\sum_{w\in
W^0_k(\mathfrak x)'} {\rm ch}\big{[}
\tL(\ov{\mf{l}},\vartheta(w\circ\La^{\mathfrak x}(\la)))\big{]},
\end{align*}
which is equal to $\omega^{\mathfrak x}\big{(}{\rm ch}\big{[}{\rm
H}_k({\mathfrak u}_-;\tL(\G,\La^{\mathfrak x}(\la)))\big{]}\big{)}$
by \eqref{homology:schur}.
\end{proof}

\begin{cor}\label{aux510}\mbox{}
\begin{itemize}
\item[(1)] The character of ${\rm H}_k(\ov{\mf
u}_-;L(\gl(p+m|q+n),\La^{\mf a}_f(\la)))$, i.e. the trace of the
operator $\prod_{i,j,s,t}x_i^{E_{ii}}\eta_j^{E_{\bar{j},\bar{j}}}
y_s^{E_{ss}}\xi_t^{E_{\bar{t},\bar{t}}}$, is given by
$$
\sum_{w\in W^0_k(\mf{a})'}
\left(\frac{\eta_{-q}\cdots\eta_{-1}}{x_{-p}\cdots
x_{-1}}\right)^{d} hs_{\la^+_w}(\xi,y)
hs_{\la^-_w}(\eta^{-1},x^{-1}).$$

\item[(2)] The character of ${\rm H}_k(\ov{\mf
u}_-;L(\spo(2m|2n+1),\La^{\mf b}_f(\la)))$, i.e. the trace of
$\prod_{i,j}\eta_i^{E_{i}}x_j^{E_{\bar{j}}}$, is
$$
\sum_{w\in W^0_k(\mf{b})'} \left(\frac{\eta_{1}\cdots
\eta_{m}}{x_{1}\cdots x_{n}}\right)^{\frac{d}{2}}
hs_{\la_w}(\eta,x).$$

\item[(3)] The character of ${\rm H}_k(\ov{\mf
u}_-;L(\osp(2m|2n),\La^{\mf c}_f(\la)))$, i.e. the trace of
$\prod_{i,j}\eta_i^{E_{i}}x_j^{E_{\bar{j}}}$, is
$$
\sum_{w\in W^0_k(\mf{c})'} \left(\frac{\eta_{1}\cdots
\eta_{m}}{x_{1}\cdots x_{n}}\right)^{\frac{d}{2}}
hs_{\la_w}(\eta,x).$$

\item[(4)] The character of  ${\rm H}_k(\ov{\mf u}_-;\tL
(\spo(2m|2n),\La^{\mf d}_f(\la)))$, i.e. the trace of
$\prod_{i,j}\eta_i^{E_{i}}x_j^{E_{\bar{j}}}$, is
$$
\begin{cases}
\sum_{w\in W^0_k(\mf{d})'} \left(\dfrac{\eta_{1}\cdots
\eta_{m}}{x_{1}\cdots x_{n}}\right)^{\frac{d}{2}}
hs_{\la_w}(\eta,x), & \text{if $G={\rm
O}(2\ell+1)$}, \\
\sum_{w\in W^0_k(\mf{d})'}\left(\dfrac{\eta_{1}\cdots
\eta_{m}}{x_{1}\cdots x_{n}}\right)^{\frac{d}{2}}\big[
hs_{\la_w}(\eta,x) +hs_{\tilde{\la}_w}(\eta,x)\big], & \text{if
$G={\rm O}(2\ell)$}.
\end{cases}$$
\end{itemize}
\end{cor}

These character formulas of ${\rm H}_k(\ov{\mf u}_-;\tL(\SG,\La^{\mf
x}_f(\la)))$ fit well with those of $\SG$-modules in
\secref{sec:char} by the Euler-Poincar\'e principle.

\thmref{mainthm} and \lemref{aux222} also imply the following
super analogue of \propref{eigenvalue:classical}~(2), which is the
converse of Proposition~\ref{eigenvalue:super}.

\begin{cor}
If $L(\overline{\mathfrak l},\eta)$ with $(\eta
+2\rho_s\vert\eta)_s=(\widehat{\La}^{\mathfrak
x}_f(\la)+2\rho_s\vert\widehat{\La}^{\mathfrak x}_f(\la))_s$ appears
in $\Lambda^k(\overline{\mathfrak u}_-)\otimes
\tL(\SG,\widehat{\La}^{\mathfrak x}_f(\la))$ with multiplicity
$m_\eta$, then $L(\overline{\mathfrak l},\eta)$ is a summand of
${\rm H}_k(\overline{\mathfrak u}_-;\tL(\SG,\widehat{\La}^{\mathfrak
x}_f(\la)))$ with the same multiplicity $m_\eta$.
\end{cor}

\begin{rem}
Specializing to $q=n=0$ in \corref{aux510} (1), (3) and (4) we obtain Kostant type
homology formulas of the unitarizable highest weight modules of the Lie groups ${\rm
SU}(p,m)$, ${\rm SO}^*(2m)$, and the double cover of ${\rm Sp}(2m,\mathbb R)$. For
unitarizable highest weight modules such homology groups were first computed by Enright
\cite[Theorem 2.2]{En}, and they involve a complicated subset of the corresponding finite
Weyl group \cite[Definition 3.6]{DES}. \corref{aux510} provides an alternative
description involving an infinite Weyl group. Ngau Lam has informed us that both forms of the formula are equivalent [HLT].

\end{rem}

\begin{rem}
Although all the discussions above dealt with ${\mf u}_-$-homology
and $\ov{\mf u}_-$-homology groups, our calculations also give
formulas for the corresponding (restricted) ${\mf u}_+$-cohomology
and $\ov{\mf u}_+$-cohomology groups (see \cite[Lemma 9]{L}).
\end{rem}

\section{Homology formulas for oscillator modules at negative levels}\label{sec:negative}

In this section, we shall compute the character formulas and the
${\mathfrak u}_-$-homology groups of various modules of $\mf{b}^{\mf
0}_\infty$, $\mf{c}_\infty$ and ${\mf d}_\infty$ at negative levels.
As the approach is parallel to that of the earlier sections, the
presentation here will be rather sketchy.
\subsection{The character formulas} We fix a positive integer $\ell\geq 1$. Consider
$\ell$ pairs of free bosons
$$\gamma^{\pm,i}(z)=\sum_{r\in\frac{1}{2}+\Z}\gamma^{\pm,i}_rz^{-r-1/2}$$
with $i=1,\ldots,\ell$. Let $\F^{-\ell}$ denote the corresponding
Fock space generated by the vaccum vector $|0\rangle$, which is
annihilated by $\gamma^{\pm,i}_r$ for $r>0$. We also denote by
$\F^{-\ell+\hf}$ the tensor product of $\F^{-\ell}$ and $\F^{\hf}$.

\subsubsection{The case of ${\mf d}_\infty$}

Suppose that $\G=\mf{d}_{\infty}$ and $d$ is even. By \cite[Theorem
5.2]{Wa} there exists a commuting action of ${\mf d}_\infty$ and
${\rm Sp}(d)$ on $\F^{-\frac{d}{2}}$. Furthermore, under this joint
action, we have
\begin{equation} \label{eq:dual-negative d}
\F^{-\frac{d}{2}}\cong\bigoplus_{\la\in {\mc P}({\rm Sp}(d))} L({\mf
d}_\infty,\La^{\mf d}_-(\la))\otimes V_{{\rm Sp}(d)}^\la,
\end{equation}
where $\La^{\mf d}_-(\la):=-d\La^{\mf
d}_0+\sum_{k=1}^{\frac{d}{2}}\la_k\epsilon_k$.

Computing the trace of the operator $\prod_{n\in
\N}x_n^{\widetilde{E}_n}\prod_{i=1}^\frac{d}{2} z_i^{{e}_{i}}$ on
both sides of \eqnref{eq:dual-negative d}, we obtain the following
identity:
\begin{equation}\label{combid-classical-negative d}
\prod_{i=1}^\frac{d}{2}\prod_{n\in\N}\frac{1}{(1-x_nz_i)(1-x_{n}z^{-1}_{i})}=\sum_{\la\in
{\mc P}({\rm Sp}(d))} {\rm ch}L(\mf{d}_{\infty},\Lambda^{\mf
d}_-(\la)){\rm ch}V^\la_{{\rm Sp}(d)}.
\end{equation}

By \eqnref{combid-classical-c} and similar arguments as in Theorem
\ref{superchar:A}-\ref{superchar:SPO}, we obtain the following.
\begin{thm}\label{char:negative d}
For $\la\in\mc{P}({\rm Sp}(d))$, we have
\begin{equation*}
{\rm ch}L({\mf d}_{\infty},\La^{\mf d}_-(\la))
=\frac{\sum_{k=0}^{\infty}\sum_{w\in
W^0_k(\mf{c})}(-1)^ks_{(\la_w)'}(x_1,x_2,\ldots)}{\prod_{i<j}(1-x_ix_j)}.
\end{equation*}
\end{thm}
The involution $\omega$ on the ring of symmetric functions in
$x_1,x_2,\ldots$ maps the left-hand side of
\eqnref{combid-classical-c} to that of
\eqnref{combid-classical-negative d}. Thus it follows that $\omega
\left({\rm ch}L(\mf{c}_{\infty},\La^{\mf c}(\la))\right)={\rm
ch}L(\mf{d}_{\infty},\La^{\mf d}_-(\la))$.

\subsubsection{The case of ${\mf c}_\infty$}

By \cite[Theorem 5.3 and 6.2]{Wa} there exists a commuting action of
${\mf c}_\infty$ and ${\rm O}(d)$ on $\F^{-\frac{d}{2}}$.
Furthermore, under this joint action, we have
\begin{equation} \label{eq:dual-negative c}
\F^{-\frac{d}{2}}\cong\bigoplus_{\la\in {\mc P}({\rm O}(d))} L({\mf
c}_\infty,\La^{\mf c}_-(\la))\otimes V_{{\rm O}(d)}^\la,
\end{equation}
where $\La^{\mf c}_-(\la):=-\frac{d}{2}\La^{\mf
c}_0+\sum_{k=1}^{d}\la_k\epsilon_k$.

If $d=2\ell$, then the calculation of the trace of the operator
$\prod_{n\in \N}x_n^{\widetilde{E}_n}\prod_{i=1}^\ell z_i^{{e}_{i}}$
on both sides of \eqnref{eq:dual-negative c} gives the following
identity
\begin{equation}\label{combid-classical-negative c1}
\prod_{i=1}^\ell\prod_{n\in\N}\frac{1}{(1-x_nz_i)(1-x_{n}z^{-1}_{i})}=\sum_{\la\in
{\mc P}({\rm O}(2\ell))} {\rm
ch}L(\mathfrak{c}_{\infty},\Lambda^{\mf c}_-(\la)){\rm
ch}V^\la_{{\rm O}(2\ell)}.
\end{equation}

If $d=2\ell+1$, then the trace of the operator $\prod_{n\in
\N}x_n^{\widetilde{E}_n}\prod_{i=1}^\ell z_i^{{e}_{i}}(-I_d)$  gives
\begin{equation}\label{combid-classical-negative c2}
\prod_{i=1}^\ell\prod_{n\in\N}\frac{1}{(1-\epsilon x_nz_i)(1-
\epsilon x_{n}z^{-1}_{i})(1-\epsilon x_n)}=\sum_{\la\in {\mc P}({\rm
O}(2\ell+1))} {\rm ch}L(\mathfrak{c}_{\infty},\Lambda^{\mf
c}_-(\la)){\rm ch}V^\la_{{\rm O}(2\ell+1)}.
\end{equation}

Applying the same arguments  as in \thmref{char:negative d} to
\eqnref{combid-classical-d1} and \eqnref{combid-classical-d2}, we
obtain the following.
\begin{thm}\label{char:negative c}
Let $\la\in\mc{P}({\rm O}(d))$ be given.
\begin{itemize}
\item[(1)] If $d=2\ell+1$, then we have
\begin{equation*}
{\rm ch}L({\mf c}_{\infty},\La^{\mf c}_-(\la))
=\frac{\sum_{k=0}^{\infty}\sum_{w\in
W^0_k(\mf{d})}(-1)^ks_{(\la_w)'}(x_1,x_2,\ldots)}{\prod_{i\le
j}(1-x_ix_j)}.
\end{equation*}

\item[(2)] If $d=2\ell$, then we have
\begin{equation*}
\begin{split}
&{\rm ch}L({\mf c}_{\infty},\La^{\mf c}_-(\la))+{\rm ch}L({\mf
c}_{\infty},\La^{\mf c}_-(\tilde{\la})) \\
&=\frac{\sum_{k=0}^{\infty}\sum_{w\in
W^0_k(\mf{d})}(-1)^k\big[s_{(\la_w)'}(x_1,x_2,\ldots)+s_{(\tilde{\la}_w)'}(x_1,x_2,\ldots)\big]}{\prod_{i\le
j}(1-x_ix_j)}.
\end{split}
\end{equation*}

\end{itemize}
\end{thm}
Note that $\omega$ maps \eqnref{combid-classical-d1} (resp.
\eqnref{combid-classical-d2}) to \eqnref{combid-classical-negative
c1} (resp. \eqnref{combid-classical-negative c2}).

\subsubsection{The case of $\mf{b}^{\mf 0}_\infty$}
Let $d$ be even. Recall the duality from \thmref{b'pin:duality}.
Computing the trace of the operator $\prod_{n\in
\N}x_n^{\widetilde{E}_n}\prod_{i=1}^{\frac{d}{2}} z_i^{{e}_{i}}$ on
both sides of \eqnref{eq:dual-negative b} gives the following
identity
\begin{equation}\label{combid-classical-negative b}
\prod_{i=1}^{\frac{d}{2}}\prod_{n\in\N}\frac{(z_i^\hf+z_i^{-\hf})}{(1-x_nz_i)(1-x_{n}z^{-1}_{i})}=\sum_{\la\in
{\mc P}({\rm Pin}(d))} {\rm ch}L(\mf{b}^{\mf
0}_{\infty},\Lambda^{\mf{b}^{\mf 0}}_-(\la)){\rm ch}V^\la_{{\rm
Pin}(d)}.
\end{equation}

By \eqnref{combid-classical-b} and similar arguments as in the proof
of \thmref{char:negative d} we obtain the following.

\begin{thm}\label{char:negative b}
For $\la\in\mc{P}({\rm Pin}(d))$, we have
\begin{equation*}
{\rm ch}L(\mf{b}^{\mf 0}_{\infty},\La^{\mf{b}^{\mf 0}}_-(\la))
=\frac{\sum_{k=0}^{\infty}\sum_{w\in
W^0_k(\mf{b})}(-1)^ks_{(\la_w)'}(x_1,x_2,\ldots)}{\prod_{i}(1+x_i)^{-1}\prod_{i\le
j}(1-x_ix_j)}.
\end{equation*}
\end{thm}
Now $\omega$ maps the left hand side of \eqnref{combid-classical-b}
to that of \eqnref{combid-classical-negative b}. Thus it follows
that $\omega \left({\rm ch}L(\mf{b}_{\infty},\La^{\mf
b}(\la))\right)={\rm ch}L(\mf{b}^{\mf 0}_{\infty},\La^{\mf{b}^{\mf
0}}_-(\la))$.\medskip

\subsection{Formulas for the ${\mathfrak u}_-$-homology groups}

Recall $\mc{P}_{{\mf l},c}^+$ and $\mc{P}_{{\mf l}}^+$ from
\secref{sec:Plc} and Appendix \ref{b':algebra}, which we shall now
denote by $\mc{P}_{{\mf l},c}^{+,\mf{x}}$ and $\mc{P}_{{\mf
l}}^{+,\mf{x}}$ to keep track that $\mf{l}$ is a subalgebra of
${\mathfrak x}_{\infty}$ for ${\mathfrak x}\in\{\mf{b, c,d}\}$.

Let $\mu=c\La^{\mf c}_0+\sum_{i\geq
1}\mu_i\epsilon_i\in\mc{P}_{{\mf l},c}^{+,\mf{c}}$ be given.
Define a bijective map  $\vartheta : \mc{P}_{{\mf l},c}^{+,\mf{c}}
\longrightarrow \mc{P}_{{\mf l},-2c}^{+,\mf{d}}$
\begin{equation}\label{vartheta-negative d}
\vartheta(\mu):=-2c\La^{\mf d}_0+\sum_{i\geq 1}\mu'_i\epsilon_i.
\end{equation}
In particular, we have $\vartheta(\La^{\mf c}_\pm(\la))=\La^{\mf
d}_\mp(\la)$, for $\la\in\mc{P}({\rm Sp}(d))$, where it is
understood that $\La^{\mathfrak x}_+(\la)=\La^{\mathfrak x}(\la)$
for ${\mathfrak x}\in\{\mf{c,d}\}$.

For $\mu=2c\La^{\mf b}_0+\sum_{i\ge
1}\mu_i\epsilon_i\in\mc{P}^{+,\mf b}_{\mf{l},2c}$, we define
$\vartheta : \mc{P}_{{\mf l},2c}^{+,\mf{b}} \longrightarrow
\mc{P}_{{\mf l},-c}^{+,\mf{b}^{\mf 0}}$
\begin{equation}\label{vartheta-negative b}
\vartheta(\mu):=-c\La_0^{\mf{b}^{\mf 0}}+\sum_{i\geq
1}\mu'_i\epsilon_i.
\end{equation}
We have $\vartheta(\La^{\mf b}(\la))=\La^{\mf{b}^{\mf 0}}_{-}(\la)$.

For $\la\in\mc{P}(G)$ and $w\in W^0$, we define
$\tL(\G,\La^{\mathfrak x}_-(\la))$ and $\tL(\mf{l},\vartheta^{\pm
1}(w\circ\La^{\mathfrak x}_+(\la)))$ similarly as in \eqnref{def:L}.
Then it is easy to see that $\omega ({\rm
ch}\tL(\mf{c}_{\infty},\La^{\mf c}_\pm(\la)))={\rm
ch}\tL(\mf{d}_{\infty},\La^{\mf d}_\mp(\la))$, and
\begin{equation*}
\omega\big{(}{\rm ch}\big{[}\Lambda^k\mathfrak u_-\otimes
\tL(\mf{c}_{\infty},\La^{\mf c}_\pm(\la))\big{]}\big{)}={\rm
ch}\big{[}\Lambda^k{{\mathfrak u}}_-\otimes
\tL(\mf{d}_{\infty},\La^{\mf d}_\mp(\la))\big{]}.
\end{equation*}

Using analogous arguments as in Lemma~\ref{aux112}, we can check
that for $\mu\in \mc{P}^{+,\mf{c}}_{\mf l}$, $L(\mf{l},\mu)$ is a
component in $\Lambda^k\mathfrak u_-\otimes
\tL(\mf{c}_\infty,\La^{\mf c}_\pm(\la))$ if and only if ${
L}(\mf{l},\vartheta(\mu))$ is a component in $\Lambda^k\mathfrak
u_-\otimes \tL(\mf{d}_\infty,\La_\mp^{\mf d}(\la))$ with the same
multiplicity, while $\omega({\rm ch}L(\mf{l},\mu))={\rm
ch}L(\mf{l},\vartheta(\mu))$. We have a similar correspondence
between $\mf{l}$-modules inside $\Lambda^k\mathfrak u_-\otimes
L(\mf{b}_\infty,\La^{\mf b}(\la))$ and $\Lambda^k\mathfrak
u_-\otimes L(\mf{b}^{\mf 0}_\infty,\La_-^{\mf{b}^{\mf 0}}(\la))$.

\begin{lem}\label{aux113-negative d}
For $\mu \in\mc{P}_{\mf l}^{+,\mf{x}}$, we have
$$(\mu+2\rho_c|\mu)_c=-(\vartheta(\mu)+2\rho_c|\vartheta(\mu))_c,\quad \mf{x}\in\{\mf{b,c}\}.$$ In
particular, we have
\begin{itemize}
\item[(1)] $(\La^{\mf c}_\pm(\la)+2\rho_c|\La^{\mf
c}_\pm(\la))_c=(\mu+2\rho_c|\mu)_c$ if and only if $(\La^{\mf
d}_\mp(\la)+2\rho_c|\La^{\mf
d}_\mp(\la))_c=(\vartheta(\mu)+2\rho_c|\vartheta(\mu))_c$.
\item[(2)] $(\La^{\mf b}(\la)+2\rho_c|\La^{\mf
b}(\la))_c=(\mu+2\rho_c|\mu)_c$ if and only if $(\La_-^{\mf{b}^{\mf
0}}(\la)+2\rho_c|\La_-^{\mf{b}^{\mf
0}}(\la))_c=(\vartheta(\mu)+2\rho_c|\vartheta(\mu))_c$.
\end{itemize}
\end{lem}

{
\begin{proof} Let $\mu=c\La^{\mf x}_0+\sum_{i\geq 1}\mu_i\epsilon_i$. Consider first the case of $\mf{x}=\mf{c}$. We have
\begin{align*}
(\mu+2\rho_c\vert\mu)_c  &= (\sum_{i\geq
1}\mu_i\epsilon_i+2\rho_c\vert\sum_{i\geq 1}\mu_i\epsilon_i)_{c}
+2c(\La^{\mf c}_0\vert\sum_{i\geq 1}\mu_i\epsilon_i)_c\allowdisplaybreaks \\
&=\sum_{i\geq 1}\mu_i(\mu_i-2i) -2c|\mu^\circ|
=(\mu^{\circ}+\rho_1|\mu^{\circ})_1 -2c|\mu^\circ|,
\end{align*}
where $\mu^\circ=(\mu_1,\mu_2,\ldots)$. On the other hand, we have
\begin{align*}
(\vartheta(\mu)+2\rho_c|\vartheta(\mu))_c &=(\sum_{i\geq
1}\mu'_i\epsilon_i+2\rho_c\vert\sum_{i\geq 1}\mu'_i\epsilon_i)_{c}
-4c(\La^{\mf d}_0\vert\sum_{i\geq 1}\mu'_i\epsilon_i)_c\allowdisplaybreaks \\
&=\sum_{i\geq 1} \mu'_i(\mu'_i-2(i-1))+ 2c|\mu^\circ|\allowdisplaybreaks \\
&=-(\mu^{\circ}+\rho_1|\mu^{\circ})_1+2c|\mu^\circ|
=-(\mu+2\rho_c\vert\mu)_c.
\end{align*}
The identity \eqref{macd} was used in the second last identity
above.

Now let $\mf{x}=\mf{b}$. One shows that
$$(\mu+2\rho_c|\mu)_c+(\vartheta(\mu)+2\rho_c|\vartheta(\mu))_c=\sum_{i\ge
1}\mu_i(\mu_i-2i+1)+\sum_{i\ge 1}\mu'_i(\mu'_i-2i+1).$$ Now
\eqnref{macd} says that the right-hand side is zero.
\end{proof}}

\begin{thm}\label{mainthm-negative c} Let $k\in\Z_+$, and $\vartheta$ as in {\rm \eqnref{vartheta-negative
d}} or {\rm \eqnref{vartheta-negative b}}. We have the following
isomorphisms of $\mf{l}$-modules:
\begin{align*}
{\rm H}_k({\mathfrak u}_-;{L}(\mf{d}_{\infty},\La^{\mf
d}_-(\la))&\cong \bigoplus_{w\in W^0_k(\mf{c})}L({\mf
l},\vartheta(w\circ \La^{\mf c}_+(\la))),\quad \la\in \mc{P}({\rm
Sp}(d)).\\
{\rm H}_k({\mathfrak u}_-;\tL(\mf{c}_{\infty},\La^{\mf
c}_-(\la))&\cong \bigoplus_{w\in W^0_k(\mf{d})}\tL({\mf
l},\vartheta^{-1}(w\circ \La^{\mf d}_+(\la))),\quad \la\in
\mc{P}({\rm O}(d)).\\
{\rm H}_k({\mathfrak u}_-;{L}(\mf{b}^{\mf
0}_{\infty},\La^{\mf{b}^{\mf 0}}_-(\la))&\cong \bigoplus_{w\in
W^0_k(\mf{b})}L({\mf l},\vartheta(w\circ \La^{\mf b}(\la))),\quad
\la\in \mc{P}({\rm Pin}(d)).
\end{align*}
In particular, we have ${\rm ch}\big{[}{\rm H}_k({\mathfrak
u}_-;L(\mf{d}_{\infty},\La^{\mf
d}_\mp(\la))\big{]}=\omega\big{(}{\rm ch}\big{[}{\rm H}_k({\mathfrak
u}_-;L(\mf{c}_{\infty},\La^{\mf c}_\pm(\la))\big{]}\big{)}$ and
${\rm ch}\big{[}{\rm H}_k({\mathfrak u}_-;L(\mf{b}^{\mf
0}_{\infty},\La^{\mf{b}^{\mf 0}}_-(\la))\big{]}=\omega\big{(}{\rm
ch}\big{[}{\rm H}_k({\mathfrak u}_-;L(\mf{b}_{\infty},\La^{\mf
b}(\la))\big{]}\big{)}$.
\end{thm}
\begin{proof} The result follows from the same type of argument as the one used in the proof of
\thmref{mainthm}, now using Lemma \ref{aux113-negative d}. We
leave the details to the reader.
\end{proof}

\appendix{}

\section{Two new reductive dual pairs}\label{appendix:A}

\subsection{The $\left(\spo(2m|2n+1),{\rm
Pin}(d)\right)$-duality}\label{spo:pin:duality} Let $d=2\ell$ be
even.

There exists a commuting action of $\spo(2m|2n+1)$ and ${\rm
Pin}(d)$ on $S\left(\C^{2m|2n+1}\otimes\C^\frac{d}{2}\right)$ as
follows. We have
\begin{equation*}
S(\C^{2m|2n+1}\otimes\C^\frac{d}{2})\cong
S(\C^{m|n}\otimes\C^{d})\otimes S(\C^{0|1}\otimes\C^\frac{d}{2}).
\end{equation*}
On $S(\C^{m|n}\otimes\C^{d})$ we have an action of the Howe dual pair $(\spo(2m|2n),{\rm
O}(d))$ by \propref{spo:duality}. On the other hand on
$S(\C^{0|1}\otimes\C^\frac{d}{2})\cong\Lambda(\C^\frac{d}{2})$ the Lie algebra
$\mf{so}(d)$ acts by two irreducible spin representations, giving rise to an irreducible
representation of ${\rm Pin}(d)$. Since representations of ${\rm O}(d)$ pull back to
representations of ${\rm Pin}(d)$ we obtain a commuting action of $\spo(2m|2n)$ and ${\rm
Pin}(d)$ on $S(\C^{2m|2n+1}\otimes\C^\frac{d}{2})$. This action of ${\rm Pin}(d)$ does
not factor through ${\rm O}(d)$. Furthermore the commuting action of $\spo(2m|2n)$
extends to a commuting action of $\spo(2m|2n+1)$. It follows from \secref{spo-action} and
arguments similar to \cite[Proposition 4.1]{CZ2} that
$S(\C^{2m|2n+1}\otimes\C^\frac{d}{2})$ is a unitarizable, and hence a completely
reducible, $\spo(2m|2n+1)$-module.

\subsubsection{Formulas for $\mf{so}(d)$- and
$\spo(2m|2n+1)$-action on $S(\C^{2m|2n+1}\otimes\C^{\frac{d}{2}})$}

We introduce even indeterminates ${\bf x}:=\{x_i^k,\ov{x}_i^k\}$ and
odd indeterminates ${\bf\xi}:=\{\xi_j^k,\ov{\xi}_j^k,\xi_0^k\}$,
where $1\le i\le m,1\le j\le n, 1\le k\le \ell$.  We identify
$S(\C^{2m|2n+1}\otimes\C^{\frac{d}{2}})$ with the polynomial
superalgebra $\C[{\bf x},\xi]$.  Then the actions of $\spo(2m|2n+1)$
and $\mf{so}(d)$ may be realized as differential operators as
follows.\vskip 5mm

\noindent $\bullet$ {\it Formulas for the
$\mf{so}(d)$-action.}

\begin{align}
&\frac{1}{2}\delta_{ij}-\xi_0^j \frac{\partial}{\partial \xi^i_0} +
\sum_{t=1}^n\left( \xi^i_t\frac{\partial}{\partial \xi^j_t} -
\ov{\xi}^j_t\frac{\partial}{\partial \ov{\xi}^i_t} \right) +
\sum_{s=1}^m \left( x^i_s \frac{\partial}{\partial x^j_s}
-\ov{x}^j_s \frac{\partial}{\partial \ov{x}^i_s}\right) ,\quad 1\le i,j\le\ell,\allowdisplaybreaks \label{so:action1}\\
&\frac{\partial}{\partial \xi^i_0}\frac{\partial}{\partial
\xi^j_0}+\sum_{t=1}^n (\xi^i_t\frac{\partial}{\partial
\ov{\xi}^j_t}- \xi^j_t\frac{\partial}{\partial \ov{\xi}^i_t}) +
\sum_{s=1}^m (x^i_s\frac{\partial}{\partial \ov{x}^j_s}-
x^j_s\frac{\partial}{\partial \ov{x}^i_s}),\quad 1\le i,j\le \ell; i\not=j,\allowdisplaybreaks \label{so:action2}\\
&\xi^i_0\xi^j_0 + \sum_{t=1}^n (\ov{\xi}^i_t\frac{\partial}{\partial
\xi^j_t}- \ov{\xi}^j_t\frac{\partial}{\partial \xi^i_t})
 + \sum_{s=1}^m (\ov{x}^i_s\frac{\partial}{\partial x^j_s}-
\ov{x}^j_s\frac{\partial}{\partial x^i_s}),\quad 1\le i,j\le \ell;
i\not=j.\nonumber
\end{align}\vskip 5mm

\noindent $\bullet$ {\it Formulas for the
$\spo(2m|2n+1)$-action.}\label{spo-action}

\begin{align}
&\sum_{k=1}^\ell\left( x^k_i\frac{\partial}{\partial x^k_j} +
\ov{x}^k_i\frac{\partial}{\partial \ov{x}^k_j}\right)
+\ell\delta_{ij},\quad 1\le i,j\le m,\label{spo:action1}\allowdisplaybreaks \\
&I_{x_ix_j}:=\sum_{k=1}^\ell \left( x_i^k\ov{x}_j^k + \ov{x}_i^k
x_j^k \right),\quad \Delta_{x_ix_j}:=\sum_{k=1}^\ell
\left(\frac{\partial}{\partial x_i^k}\frac{\partial}{\partial
\ov{x}_j^k} + \frac{\partial}{\partial \ov{x}_i^k}
\frac{\partial}{\partial x_j^k}
\right),\quad 1\le i,j\le m,\allowdisplaybreaks\nonumber\\
&\sum_{k=1}^\ell\left( \xi^k_i\frac{\partial}{\partial \xi^k_j} +
\ov{\xi}^k_i\frac{\partial}{\partial \ov{\xi}^k_j}\right)-\ell\delta_{ij},\quad 1\le i,j\le n, \label{spo:action2}\allowdisplaybreaks\\
&I_{\xi_i\xi_j}:=\sum_{k=1}^\ell \left( \xi_i^k\ov{\xi}^k_j +
\ov{\xi}^k_i\xi^k_j \right) ,\quad
\Delta_{\xi_i\xi_j}:=\sum_{k=1}^\ell \left( \frac{\partial}{\partial
\xi_i^k} \frac{\partial}{\partial \ov{\xi}^k_j } +
\frac{\partial}{\partial \ov{\xi}^k_i} \frac{\partial}{\partial
\xi^k_j}
\right),\quad 1\le i,j\le n;i\not=j,\nonumber\allowdisplaybreaks\\
&I_{\xi_0\xi_i}:=\sum_{k=1}^\ell \left( \xi_i^k\xi_0^k  +
\ov{\xi}_i^k \frac{\partial}{\partial {\xi}_0^k} \right),\quad
\Delta_{\xi_0\xi_i}:=\sum_{k=1}^\ell \left( \frac{\partial}{\partial
{\xi}_0^k}\frac{\partial}{\partial {\xi}_i^k}
+\xi_0^k\frac{\partial}{\partial \ov{\xi}_i^k} \right),\quad 1\le
i\le n.\nonumber
\end{align}

\begin{align}
&\Delta_{\xi_0 x_i}:=\sum_{k=1}^\ell \left( \frac{\partial}{\partial
\xi_0^k}\frac{\partial}{\partial x_i^k}+\xi_0^k
\frac{\partial}{\partial \ov{x}_i^k} \right),\quad I_{\xi_0
x_i}:=\sum_{k=1}^\ell \left( \xi_0^kx_i^k +
\ov{x}_i^k\frac{\partial}{\partial \xi_0^k} \right),\quad
1\le i\le m,\nonumber\allowdisplaybreaks\\
&\Delta_{x_i\xi_j}:=\sum_{k=1}^\ell \left( \frac{\partial}{\partial
\xi_j^k} \frac{\partial}{\partial \ov{x}_i^k} +
\frac{\partial}{\partial \ov{\xi}_j^k} \frac{\partial}{\partial
{x}_i^k} \right),\, I_{x_i\xi_j}:=\sum_{k=1}^\ell \left( {x}_i^k
 \ov{\xi}_j^k + \ov{x}_i^k{\xi}_j^k
\right),\,1\le i\le m, 1\le j\le n,\nonumber\allowdisplaybreaks\\
&\sum_{k=1}^\ell \left( x_i^k \frac{\partial}{\partial \xi_j^k} +
\ov{x}_i^k\frac{\partial}{\partial \ov{\xi}_j^k}
\right),\quad\sum_{k=1}^\ell \left(
{\xi}_j^k\frac{\partial}{\partial {x}_i^k} + \ov{\xi}_j^k
\frac{\partial}{\partial \ov{x}_i^k} \right),\quad 1\le i\le m, 1\le
j\le n.\label{spo:action3}
\end{align}

\subsubsection{The module decomposition} It is evident from \eqnref{so:action1} that the action of
${\rm Pin}(d)$ on $\C[{\bf x},\xi]$ does not factor through ${\rm
O}(d)$. Since ${\rm End}_{\C}(\C[{\bf x},\xi])^{{\rm Pin}(d)}$ is
generated by $\spo(2m|2n+1)$ in \secref{spo-action}, it follows from
the double commutant theorem that with respect to the
$\spo(2m|2n+1)\times{\rm Pin}(d)$-action we have
\begin{equation}\label{duality-pin}
\C[{\bf x},\xi]\cong\bigoplus_{\la\in\La} L(\spo(2m|2n+1),\La^{\mf
b}_f(\la))\otimes V^\la_{{\rm Pin}(d)},
\end{equation}
where $\La\subseteq\mc{P}({\rm Pin}(d))$, and $\La^{\mf
b}_f:\La\rightarrow\ov{\h}^*$ is an injection. It remains to
determine the set $\La$ and the map $\La^{\mf{b}}_f$.

\begin{thm}\label{spopin-duality} As an $\spo(2m|2n+1)\times{\rm Pin}(d)$-module we have
\begin{equation*}
S(\C^{2m|2n+1}\otimes\C^{\frac{d}{2}})\cong \bigoplus_{\la}
L(\spo(2m|2n+1),\La^{\mf b}_f(\la))\otimes V^\la_{{\rm Pin}(d)},
\end{equation*}
where the summation is over all $\la\in \mc{P}({\rm Pin}(d))$ with $\la_{m+1}\le n$, and
$\La^{\mf b}_f(\la)=\la^\natural+\frac{d}{2}{\bf 1}_{m|n}$.
\end{thm}

\begin{proof} Let $\ov{\U}_+$ be the algebra generated by the $\Delta$-operators, i.e.~$\Delta_{x_ix_j}$,
$\Delta_{\xi_i\xi_j}$, $\Delta_{\xi_0\xi_j}$, $\Delta_{\xi_0 x_i}$, and
$\Delta_{x_i\xi_j}$. Then $\ov{\U}_+$ is invariant under the adjoint action of
$\gl(m|n)$.

An element $f\in\C[\bf{x},\xi]$ is called {\em harmonic}, if $f$ is annihilated by
$\ov{\mf{u}}^+$. The space of harmonics will be denoted by $\mc{H}$ and it evidently
admits an action of $\glmn\times {\rm Pin}(d)$. Furthermore, since
$S(\C^{2m|2n+1}\otimes\C^\frac{d}{2})$ is a completely reducible $\glmn$-module,
$L(\spo(2m|2n+1),\mu)^{\ov{\mf{u}}^+}$ is also completely reducible over $\glmn$, for any
irreducible $\spo(2m|2n+1)$-module $L(\spo(2m|2n+1),\mu)$ that appears in
$S(\C^{2m|2n+1}\otimes\C^\frac{d}{2})$. By irreducibility of $L(\spo(2m|2n+1),\mu)$ we
must have
\begin{align*}
L(\spo(2m|2n+1),\mu)^{\ov{\mf{u}}^+}\cong L(\glmn,\mu).
\end{align*}
So, by \eqref{duality-pin} $(\text{Pin}(d),\glmn)$ forms a Howe dual pair on $\mc{H}$.
Thus, proving the theorem is equivalent to establishing the following decomposition of
$\mc{H}$ as a $\glmn\times\text{Pin}(d)$-module:
\begin{equation}  \label{harmonic}
\mc{H}\cong\bigoplus_{\la }
L(\glmn,\La^{\mf b}_f(\la))\otimes V^\la_{{\rm Pin}(d)},
\end{equation}
where the summation is over all $\la\in \mc{P}({\rm Pin}(d))$, i.e.~$\ell(\la)\le\ell$,
with $\la_{m+1}\le n$, and $\La^{\mf b}_f(\la)=\la^\natural+\ell{\bf 1}_{m|n}$.

We first consider the limit case $n=\infty$ with the space of harmonics denoted by $
\mc{H}^\infty$. Here the only restriction on $\la$ is $\ell(\la)\le \ell$, and we observe
that the vector given in \cite[Theorems 4.1 and 4.2]{CW1} associated to such a partition
$\la$ is indeed annihilated by $\ov{\mf{u}}^+$ and hence is a joint
$\glmn\times\text{Pin}(d)$-highest weight vector of weight $(\la,\la^\natural+\ell{\bf
1}_{m|n})$. Hence all the summands on the right hand side of (\ref{harmonic}) occur in
the space of harmonics, and in particular, all irreducible representations of
$\text{Pin}(d)$ occur. Therefore, we have established (\ref{harmonic}) in the case
$n=\infty$.

Now consider the finite $n$ case. We may regard
$S(\C^{2m|1+2n}\otimes\C^\frac{d}{2})\subseteq S(\C^{2m|1+2\infty}\otimes\C^\frac{d}{2})$
with compatible actions $\spo(2m|1+2n)\subseteq\spo(2m|1+2\infty)$.  From the formulas of
the $\Delta$-operators we see that $\mc{H}\subseteq\mc{H}^\infty$. Thus $\mc{H}$ is
obtained from $\mc{H}^\infty$ by setting the variables $\xi^k_j=\ov{\xi}^k_j=0$, for
$j>n$. However, it is clear, from the explicit formulas of the joint highest vectors in
$\mc{H}^\infty$, that, when setting the variables $\xi^k_j=\ov{\xi}^k_j=0$ for $j>n$,
precisely those vectors corresponding to $\la$ with $\la_{m+1}\le n$ will survive.
\end{proof}

\begin{rem}
\thmref{spopin-duality} is a finite-dimensional analogue of
\cite[Theorem 8.2]{CW2}.
\end{rem}

\subsection{The Lie superalgebra $\mf{b}^{\mf 0}_\infty$ and the
$(\mf{b}^{\mf 0}_\infty,{\rm Pin}(d))$-duality}\label{b':algebra}

Consider the superspace with basis $\{\,v_{\hf},v_k\,|\,k\in\Z\,\}$
 with ${\rm deg}v_k=\bar{0}$, for $k\in\Z$, and
${\rm deg}v_\hf=\bar{1}$. The Lie superalgebra $\ov{\mf{b}^{\mf
0}_\infty}$ is the subalgebra of the general linear superalgebra
preserving the even super-skewsymmetric bilinear form determined by
$(v_i|v_j):=(-1)^i\delta_{i,1-j}$, $i,j\in\Z$, and
$(v_\hf|v_\hf):=1$. Now consider the central extension of general
linear superalgebra corresponding to the $2$-cocycle
$\gamma(A,B):={\rm
Str}\left(\left(\sum_{j<1}{E_{jj}}\right)[A,B]\right)$, where ${\rm
Str}(c_{ij}):=-c_{\hf\hf}+\sum_{j\in\Z}c_{jj}$.  The Lie
superalgebra $\mf{b}^{\mf 0}_\infty$ is the central extension of
$\ov{\mf{b}^{\mf 0}_\infty}$ by a one-dimensional center $\C K$,
obtained via restriction of the cocycle $\gamma$.

By construction one sees that $\mf{c}_\infty\subseteq \mf{b}^{\mf
0}_\infty$, and also $\mf{c}_\infty$ and $\mf{b}^{\mf 0}_\infty$
share the same Cartan subalgebra $\h$. This allows us to identify
their Cartan and dual Cartan subalgebras.  For convenience of the
reader we list below the simple roots and coroots for $\mf{b}^{\mf
0}_\infty$.

\begin{align*}
\Pi^\vee=&\{\alpha_0^\vee=\widetilde{E}_1+K,\
\alpha_1^\vee=\widetilde{E}_i-\widetilde{E}_{i+1}\
(i\in\N)\}\\
\Pi=&\{\alpha_0=-\epsilon_1,\ \alpha_i=\epsilon_i-\epsilon_{i+1}\
(i\in\N)\}\\
\Delta_+=&\{\pm\epsilon_i-\epsilon_j,\ -\epsilon_i,\ -2\epsilon_i\
(i\in\N, i<j)\}.
\end{align*}
The associated Dynkin diagram is as follows:
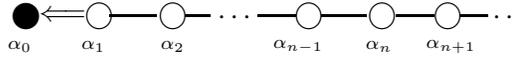
\begin{center}
\hskip -3cm \setlength{\unitlength}{0.16in}
\begin{picture}(24,4)
\put(5.6,2){\circle*{0.9}} \put(8,2){\makebox(0,0)[c]{$\bigcirc$}}
\put(10.4,2){\makebox(0,0)[c]{$\bigcirc$}}
\put(14.85,2){\makebox(0,0)[c]{$\bigcirc$}}
\put(17.25,2){\makebox(0,0)[c]{$\bigcirc$}}
\put(19.4,2){\makebox(0,0)[c]{$\bigcirc$}}
\put(8.35,2){\line(1,0){1.5}} \put(10.82,2){\line(1,0){0.8}}
\put(13.2,2){\line(1,0){1.2}} \put(15.28,2){\line(1,0){1.45}}
\put(17.7,2){\line(1,0){1.25}} \put(19.81,2){\line(1,0){0.9}}
\put(6.8,2){\makebox(0,0)[c]{$\Longleftarrow$}}
\put(12.5,1.95){\makebox(0,0)[c]{$\cdots$}}
\put(21.5,1.95){\makebox(0,0)[c]{$\cdots$}}
\put(5.4,1){\makebox(0,0)[c]{\tiny $\alpha_0$}}
\put(7.8,1){\makebox(0,0)[c]{\tiny $\alpha_1$}}
\put(10.4,1){\makebox(0,0)[c]{\tiny $\alpha_2$}}
\put(14.5,1){\makebox(0,0)[c]{\tiny $\alpha_{n-1}$}}
\put(17.15,1){\makebox(0,0)[c]{\tiny $\alpha_n$}}
\put(19.5,1){\makebox(0,0)[c]{\tiny $\alpha_{n+1}$}}
\end{picture}
\end{center}

Let $\rho_c,\La_0^{\mf{b}^{\mf 0}}\in\h^*$ be determined by
$\langle\rho_c,\widetilde{E}_j\rangle=-j+\hf$,
$\langle\rho_c,K\rangle=\langle\La_0^{\mf{b}^{\mf
0}},\widetilde{E}_j\rangle=0$, and $\langle\La_0^{\mf{b}^{\mf
0}},K\rangle=1$, $j\in\N$.  For $c\in\C$ let $\mc{P}^{+,\mf{b}^{\mf
0}}_{\mf{l},c}$ consist of elements in $\h^*$ of the form
$c\La_0^{\mf{b}^{\mf 0}}+\sum_{i\ge 1}\mu_i\epsilon_i$, where
$(\mu_1,\mu_2,\ldots)\in\mc{P}^+$.

In \secref{sec:negative} the following bilinear form
$(\cdot\vert\cdot)_c$ on $\mf{b}^{\mf 0}_\infty$ is used. We first
choose a bilinear form $(\cdot\vert\cdot)_c$ on $\h^*$ satisfying
\begin{align*}
&(\la\vert \epsilon_i)_c =\langle \la,\widetilde{E}_{i}-K\rangle,
\quad i\in\N,\\ &(\La^{\mf{b}^{\mf 0}}_0\vert\La^{\mf{b}^{\mf
0}}_0)_c=(\La^{\mf{b}^{\mf 0}}_0\vert\rho_c)_c =0.
\end{align*}
One checks that $(\epsilon_i\vert\epsilon_j)_c=\delta_{ij}$,
$(\Lambda^{\mf{b}^{\mf 0}}_0\vert\epsilon_i)_c=-1$ for $i,j\in \N$,
and $2(\rho_c\vert\alpha_i)_c=(\alpha_i\vert\alpha_i)_c$ for $i\in
\Z_+$. Let $\{\,s^{\mf{b}^{\mf 0}}_i\,\}_{i\in \Z_+}$ be the
sequence defined by
\begin{equation*}
s^{\mf{b}^{\mf 0}}_i:=
\begin{cases}
-1, & \text{if $i=0$}, \\
1, & \text{if $i\geq 1$}.
\end{cases}
\end{equation*}
It follows that by defining $s^{\mf{b}^{\mf 0}}_i s^{\mf{b}^{\mf
0}}_j({\alpha}^{\vee}_i\vert
{\alpha}^{\vee}_j)_c:=(\alpha_i\vert\alpha_j)_c$, we obtain a
symmetric bilinear form on $\h$. This form can be extended to a
non-degenerate invariant super-symmetric bilinear form on
$\mf{b}^{\mf 0}_\infty$ such that
\begin{equation}\label{aux:casimir2}
(e_i\vert f_j)_c=\delta_{ij}/s^{\mf{b}^{\mf 0}}_i,
\end{equation}
where $e_i$ and $f_j$ $(i,j\in \Z_+)$ denote the Chevalley
generators with $[e_i,f_i]=\alpha_i^{\vee}$. Now we can define the
Casimir operator $\Omega$ as in \cite[Section 4.1]{CK}.

Let $d=2\ell$ be even. By \cite[Theorem 5.3]{Wa} there exists a
$({\mf c}_\infty,{\rm O}(d))$-duality on $\F^{-\frac{d}{2}}$.  Now
$\La(\C^{\frac{d}{2}})$ is an irreducible ${\rm Pin}(d)$-module and
hence we have a commuting action of $\mf{c}_\infty$ and ${\rm
Pin}(d)$ on $\F^{-\frac{d}{2}}\otimes\La(\C^{\frac{d}{2}})$. One can
show that the action of $\mf{c}_\infty$ on
$\F^{-\frac{d}{2}}\otimes\La(\C^{\frac{d}{2}})$ extends to an action
of $\mf{b}^{\mf 0}_\infty$ that commutes with the action of ${\rm
Pin}(d)$, and furthermore $\mf{b}^{\mf 0}_\infty$ generates ${\rm
End}(\F^{-\frac{d}{2}}\otimes\La(\C^{\frac{d}{2}}))^{{\rm Pin}(d)}$.

In the sequel we will need to have this commuting action in a more
explicit form.  For this let us introduce odd indeterminates
$\xi_{\hf}:=\{\xi^{k}_{\hf}\vert k=1,\ldots,\ell\}$, and identify
$\La(\C^{\frac{d}{2}})$ with the Clifford superalgebra generated by
$\xi_{\hf}$. The action of the Lie algebra $\mf{so}(d)$ on
$\La(\C^{\frac{d}{2}})$ in terms of differential operators in
$\xi^k_{\hf}$ is explicitly given by the summands involving
$\xi_0^k$ in \eqnref{so:action1} and \eqnref{so:action2}. This
action combined with the action of $\mf{so}(d)$ on
$\F^{-\frac{d}{2}}$ as in \cite[(5.52)]{Wa}, gives the action of
$\mf{so}(d)$ on $\F^{-\frac{d}{2}}\otimes\La(\C^{\frac{d}{2}})$. The
commuting action of $\mf{b}^{\mf 0}_\infty$ is as follows. First the
subalgebra $\mf{c}_\infty$ acts on
$\F^{-\frac{d}{2}}\otimes\La(\C^{\frac{d}{2}})$ only on the first
factor as in \cite[Section 5.2]{Wa}.  To complete the description we
only need to give formulas for the action of the odd root vectors
$I_{\xi_\hf i}$ and $\Delta_{\xi_\hf i}$, corresponding to the roots
$\epsilon_i$ and $-\epsilon_i$, $i\in\N$, respectively. They are as
follows:
\begin{equation*}
I_{\xi_\hf
i}:=\sum_{k=1}^\ell\left(\xi_{\hf}^k\gamma^{+,k}_{-i+\hf}-\frac{\partial}{\partial
\xi_\hf^k}\gamma^{-,k}_{-i+\hf}\right),\quad \Delta_{\xi_\hf
i}:=\sum_{k=1}^\ell\left(\xi_{\hf}^k\gamma^{+,k}_{i-\hf}+\frac{\partial}{\partial
\xi_\hf^k}\gamma^{-,k}_{i-\hf}\right).
\end{equation*}

We have the following.

\begin{thm}\label{b'pin:duality} As a $\mf{b}^{\mf 0}_{\infty}\times{\rm
Pin}(d)$-module we have
\begin{equation} \label{eq:dual-negative b}
\F^{-\frac{d}{2}}\otimes\La(\C^{\frac{d}{2}})\cong\bigoplus_{\la\in
{\mc P}({\rm Pin}(d))} L(\mf{b}^{\mf 0}_\infty,\La^{\mf{b}^{\mf
0}}_-(\la))\otimes V_{{\rm Pin}(d)}^\la,
\end{equation}
where $\La^{\mf{b}^{\mf 0}}_-(\la):=-\frac{d}{2}\La^{\mf{b}^{\mf
0}}_0+\sum_{k=1}^{\frac{d}{2}}\la_k\epsilon_k$.
\end{thm}

\begin{proof}
One checks that the $\mf{c}_\infty\times\mf{so}(d)$-joint highest
weight vectors inside $\F^{-\frac{d}{2}}$ given in \cite[(5.54)]{Wa}
are also $\mf{b}^{\mf 0}_\infty$-highest weight vectors. For this it
is enough to check that they are annihilated by the root vector
$\Delta_{\xi_\hf 1}$ (corresponding to the simple odd root
$-\epsilon_1$). This, however, is easy. The duality then follows
from the fact that this set of vectors exhaust all irreducible
finite dimensional ${\rm Pin}(d)$-highest weights that do not factor
through ${\rm O}(d)$.
\end{proof}

\end{document}